\documentclass{amsart}
\usepackage{amssymb}
\usepackage{curves}
\usepackage{url}

\newtheorem{theorem}{Theorem}[section]
\newtheorem{prop}[theorem]{Proposition}
\newtheorem{lem}[theorem]{Lemma}
\newtheorem{cor}[theorem]{Corollary}

\theoremstyle{definition}
\newtheorem{defn}[theorem]{Definition}
\newtheorem{eg}[theorem]{Example}
\newtheorem{remark}[theorem]{Remark}
\newtheorem{problem}[theorem]{Problem}

\newcommand{\Aut}{\operatorname{Aut}}
\newcommand{\SAut}{\operatorname{SAut}}

\newcommand{\AGL}{\operatorname{AGL}}

\newcommand{\blobb}{\circle*{1}}

\newcommand{\Ham}{\operatorname{Ham}}
\newcommand{\Sym}{\operatorname{Sym}}
\newcommand{\E}{\mathcal{E}}
\newcommand{\quotient}{/\!\!/}

\newcommand{\dsl}[2]{\mathfrak D(#1,#2)}

\begin{document}
\title{The geometry of diagonal groups}
\author[Bailey, Cameron, Praeger and Schneider]{R. A. Bailey}
\author[]{Peter J. Cameron}
\address{School of Mathematics and Statistics\\
University of St Andrews\\
St Andrews, Fife KY16 9SS, UK}
\email{rab24@st-andrews.ac.uk, pjc20@st-andrews.ac.uk}
\author[]{Cheryl E. Praeger}
\address{Department of Mathematics and Statistics\\
University of Western Australia\\
Perth WA 6009, Australia}
\email{cheryl.praeger@uwa.edu.au}
\author[]{Csaba Schneider}
\address{Departimento di Matem\'atica\\
Instituto de Ci\^encias Exatas\\
Universidade Federal de Minas Gerais\\
Belo Horizonte, MG, Brazil}
\email{csaba.schneider@gmail.com}


\begin{abstract}
Diagonal groups are one of the classes of finite primitive permutation groups
occurring in the conclusion of the O'Nan--Scott theorem. Several of the
other classes have been described as the automorphism groups of geometric
or combinatorial structures such as affine spaces or Cartesian decompositions,
but such structures for diagonal groups have not been studied in general.

The main purpose of this paper is to describe and characterise such
structures, which we call \emph{diagonal semilattices}. Unlike the 
diagonal groups in the O'Nan--Scott theorem, which are defined over finite
characteristically simple groups, our construction works over arbitrary 
groups, finite or infinite.

A diagonal semilattice depends on a dimension~$m$ and a group~$T$. For
$m=2$, it is a Latin square, the Cayley table of~$T$, though in fact any
Latin square satisfies our combinatorial axioms. However, for \mbox{$m\geqslant3$}, 
the group $T$ emerges naturally and uniquely from the axioms. (The situation
somewhat resembles projective geometry, where projective planes exist in great
profusion but higher-dimensional structures are coordinatised by an algebraic
object, a division ring.)

A diagonal semilattice is contained in the partition lattice on a
set~$\Omega$, and we provide an introduction to the calculus of partitions. Many
of the concepts and constructions come from experimental design in statistics.

We also determine when a diagonal group can be primitive, or quasiprimitive
(these conditions turn out to be equivalent for diagonal groups).

Associated with the diagonal semilattice is a graph, the diagonal graph,
which has the same automorphism group as the diagonal semilattice except
in four small cases with $m\leqslant3$. The class of diagonal graphs includes some
well-known families, Latin-square graphs and folded cubes, and is potentially
of interest. We obtain partial results on the chromatic number of a
diagonal graph, and mention an application to the synchronization property
of permutation groups.
\end{abstract}


\maketitle

\section{Introduction}
\label{sec:start}

\subsection{The landscape}\label{sect:landsc}

In this paper, we give a combinatorial description of the structures on
which diagonal groups, including those arising in the O'Nan--Scott Theorem,
act.

This is a rich area, with links not only to finite group theory (as in the
O'Nan--Scott Theorem) but also to designed experiments, and the combinatorics
of Latin squares and their higher-dimensional generalisations. We do not
restrict our study to the finite case.

Partitions lie at the heart of this study. We express the Latin hypercubes
we need in terms of partitions, and our final structure for diagonal groups
can be regarded as a join-semilattice of partitions. Cartesian products of sets 
can be described in terms of the partitions induced by the coordinate projection
maps
and this approach was introduced  into the study of primitive permutation groups 
by L.~G.~Kov\'acs~\cite{kov:decomp}. He called the collection of these coordinate partitions 
a ``system  of product imprimitivity''. The concept was further developed 
in~\cite{ps:cartesian} where the same object was called a ``Cartesian decomposition''. 
In preparation for introducing the join-semilattice of partitions for the diagonal 
groups, we view Cartesian decompositions as lattices of partitions of the 
underlying set. 

Along the way, we also discuss a number of conditions on families of partitions
that have been considered in the literature, especially the statistical
literature.

\subsection{Outline of the paper}

As said above, our aim is to describe the geometry and combinatorics underlying
diagonal groups, in general. In the O'Nan--Scott Theorem, the diagonal groups
$D(T,m)$ depend on a non-abelian simple group $T$ and a positive integer~$m$.
But these groups can be defined for an arbitrary group $T$, finite or infinite,
and we investigate them in full generality.

Our purpose is to describe the structures on which diagonal groups act. This
takes two forms: descriptive, and axiomatic. In the former, we start with a
group $T$ and a positive integer $m$, build the structure on which the group
acts, and study its properties. The axiomatic approach is captured by the
following theorem, to be proved in Section~\ref{sec:diag}. Undefined terms
such as Cartesian lattice, Latin square,
paratopism, and diagonal semilattice will be introduced later, so that when
we get to the point of proving the theorem its statement should be clear.
We mention here that the automorphism group of a Cartesian lattice is,
in the simplest case,
a wreath product of two symmetric groups in its product action, while the
automorphism group of a diagonal semilattice $\dsl Tm$ is the diagonal group $D(T,m)$;
Latin squares, on the other hand, may (and usually do) have only the trivial
group of automorphisms.

\begin{theorem}\label{thm:main}
Let $\Omega$ be a set with $|\Omega|>1$, and $m$ an integer at least $2$. Let $Q_0,\ldots,Q_m$
be $m+1$ partitions of $\Omega$ satisfying the following property: any $m$
of them are the minimal non-trivial partitions in a Cartesian lattice on
$\Omega$.
\begin{enumerate}\itemsep0pt
\item If $m=2$, then the three partitions are the row, column, and letter
partitions of a Latin square on $\Omega$, unique up to paratopism.
\item If $m>2$, then there is a group $T$, unique up to isomorphism,
such that $Q_0,\ldots,Q_m$ are the minimal non-trivial partitions in a diagonal
semilattice $\dsl Tm$ on $\Omega$.
\end{enumerate}
\end{theorem}

The case $m=3$ in Theorem~\ref{thm:main}(b) can be phrased in the language 
of Latin cubes and may thus  be of independent interest. The proof is in
Theorems~\ref{thm:bingo} and \ref{th:upfront} (see also
Theorem~\ref{thm:regnice}). See Section~\ref{sec:whatis} for the definition
of a regular Latin cube of sort (LC2). 

\begin{theorem}
  \label{thm:bingo_}
Consider a Latin cube of sort (LC2) on an underlying set~$\Omega$,
  with coordinate partitions $P_1$, $P_2$ and $P_3$, and letter partition~$L$.
Then the Latin cube is regular if and only if there is a group~$T$ such that, up to relabelling the letters
  and the three sets of coordinates,
  $\Omega=T^3$ and  $L$ is the coset partition defined
  by the diagonal subgroup $\{(t,t,t) \mid t \in T\}$.
Moreover, $T$ is unique up to group isomorphism.
\end{theorem}

Theorem~\ref{thm:main}
has a similar form to the axiomatisation of projective geometry
(see \cite{vy}). We give simple axioms, and show that diagonal structures of smallest
dimension satisfying them are ``wild'' and exist in great profusion, while
higher-dimensional structures can be completely described in terms of an
algebraic object. In our case, the algebraic object is a group, whereas, 
for projective geometry, it is a division ring, also called a skew field.
Note that the group emerges naturally from the combinatorial axioms.

In Section~\ref{sec:prelim}, we describe the preliminaries required.
Section~\ref{sec:Cart} revisits Cartesian decompositions, as described
in~\cite{ps:cartesian}, and defines Cartesian lattices.
Section~\ref{sec:LC} specialises to the case that $m=3$.  Not only does this
show that this case is very different from $m=2$; it also underpins the
proof by induction of Theorem~\ref{thm:main}, which is given in
Section~\ref{sec:diag}.

In the last two sections, we give further results on diagonal groups. In
Section~\ref{s:pqp}, we determine which diagonal groups are primitive,
and which are quasiprimitive (these two conditions turn out to be equivalent).
In Section~\ref{s:diaggraph}, we define a graph having a given diagonal 
group as its automorphism group (except for four small diagonal groups),
examine some of its graph-theoretic properties, and briefly describe the
application of this to synchronization properties of permutation groups
from~\cite{bccsz} (finite primitive diagonal groups with $m\geqslant2$ are 
non-synchronizing).

The final section poses a few open problems related to this work.

\subsection{Diagonal groups}\label{sect:diaggroups}

In this section we define the diagonal groups, in two ways: a ``homogeneous''
construction, where all factors are alike but the action is on a coset space;
and an ``inhomogeneous'' version
which gives an alternative way of labelling the elements of the underlying set
which is better for calculation
even though one of the factors has to be treated differently.

Let $T$ be a group with $|T|>1$, and $m$ an integer with $m\geqslant1$. We define the 
\emph{pre-diagonal group} $\widehat D(T,m)$ as the semidirect 
product of $T^{m+1}$ by $\Aut(T)\times S_{m+1}$, where $\Aut(T)$ (the
automorphism group of $T$) acts in the same way on each factor, and $S_{m+1}$
(the symmetric group of degree $m+1$) permutes the factors.

Let $\delta(T,m+1)$ be the diagonal subgroup $\{(t,t,\ldots,t) \mid t\in T\}$
of $T^{m+1}$,
and $\widehat H=\delta(T,m+1)\rtimes (\Aut(T)\times S_{m+1})$.
We represent $\widehat D(T,m)$ as a permutation group on the set of
right cosets of $\widehat H$. If $T$ is finite, the degree of this
permutation representation is $|T|^m$. In general, the action is not
faithful, since $\delta(T,m+1)$ (acting by conjugation)
induces inner automorphisms of $T^{m+1}$, which agree with the inner
automorphisms induced by $\Aut(T)$. 
In fact, if $m\geqslant 2$ or $T$ is non-abelian, then the kernel of the $\widehat D(T,m)$-action 
is 
\begin{align}\label{eq:K}
  \begin{split}
    \widehat K
    &=\{(t,\ldots,t)\alpha\in T^{m+1}\rtimes \Aut(T)\mid t\in T\mbox{ and}\\ 
  &\mbox{$\alpha$ is the 
  inner automorphism induced by $
  t^{-1}$}\},
  \end{split}
\end{align}
and so $\widehat K\cong T$. Thus, if, in addition, $T$ is finite,
then the order of the permutation group induced by 
$\widehat D(T,m)$ is $|\widehat D(T,m)|/| \widehat K|=
|T|^m(|\Aut(T)|\times|S_{m+1}|)$. If $m=1$ and $T$ is abelian, then 
the factor $S_2$ induces the inversion automorphism $t\mapsto t^{-1}$ on $T$ and 
the permutation group induced by $\widehat D(T,m)$ is the holomorph 
$T\rtimes \Aut(T)$. 

We define the \emph{diagonal group} $D(T,m)$ to be the permutation group
induced by $\widehat D(T,m)$ on the set of right cosets of $\widehat H$ as above.
So $D(T,m)\cong \widehat D(T,m)/\widehat K$. 

To move to a more explicit representation of $D(T,m)$,
we choose coset representatives
for $\delta(T,m+1)$ in $T^{m+1}$. A convenient choice is to number the direct 
factors of
$T^{m+1}$ as $T_0,T_1,\ldots,T_m$, and use representatives of
the form $(1,t_1,\ldots,t_m)$, with $t_i\in T_i$. We will denote this
representative by $[t_1,\ldots,t_m]$, and let $\Omega$ be the set of all
such symbols. Thus, as a set, $\Omega$ is bijective with~$T^m$.

\begin{remark}\label{rem:diaggens}
Now we can describe the action of $\widehat D(T,m)$ on $\Omega$ as follows.
\begin{itemize}\itemsep0pt
\item[(I)] For $1\leqslant i\leqslant m$, the factor $T_i$ acts by right multiplication
  on symbols in the $i$th position in elements of $\Omega$. 
\item[(II)] $T_0$ acts by simultaneous left multiplication of all coordinates by
the inverse. This is because, for $x\in T_0$, $x$ maps the coset containing
$(1,t_1,\ldots,t_m)$ to the coset containing $(x,t_1,\ldots,t_m)$, which is
the same as the coset containing $(1,x^{-1}t_1,\ldots,x^{-1}t_m)$.
\item[(III)] Automorphisms of $T$ act simultaneously on all coordinates; but
inner automorphisms are identified with the action of elements in the diagonal
subgroup $\delta(T,m+1)$ (the element $(x,x,\ldots,x)$ maps the coset containing
$(1,t_1,\ldots,t_m)$ to the coset containing $(x,t_1x,\ldots,t_mx)$, which is
the same as the coset containing $(1,x^{-1}t_1x,\ldots,x^{-1}t_mx)$).
\item[(IV)] Elements of $S_m$ (fixing coordinate $0$) act by permuting the 
coordinates in elements of $\Omega$. 
\item[(V)] Consider the element of $S_{m+1}$ which transposes coordinates $0$ and~$1$.
This maps the coset containing $(1,t_1,t_2,\ldots,t_m)$ to the coset containing
the tuple $(t_1,1,t_2\ldots,t_m)$, which
also contains
$(1,t_1^{-1},t_1^{-1}t_2,\ldots,t_1^{-1}t_m)$. So the action of this
transposition is
\[[t_1,t_2,\ldots,t_m]\mapsto[t_1^{-1},t_1^{-1}t_2,\ldots,t_1^{-1}t_m].\]
Now $S_m$ and this transposition generate $S_{m+1}$.
\end{itemize}
\end{remark}

By~\eqref{eq:K}, the kernel $\widehat K$ of the $\widehat D(T,m)$-action
on~$\Omega$ is 
contained in the subgroup generated by elements of type (I)--(III).

For example, in the case when $m=1$, the set $\Omega$ is bijective  with
$T$; the factor $T_1$ acts by right multiplication, $T_0$ acts by left
multiplication by the inverse, automorphisms act in the natural way, and
transposition of the coordinates acts as inversion.

The following theorem states that the diagonal group $D(T,m)$  can be 
viewed as the automorphism group of the corresponding diagonal join-semilattice 
$\dsl Tm$ and the diagonal graph $\Gamma_D(T,m)$ defined in 
Sections~\ref{sec:diag1} and~\ref{sec:dgds}, respectively. The two parts of
this theorem comprise Theorem~\ref{t:autDTm} and Corollary~\ref{c:sameag}
respectively.

\begin{theorem}
  Let $T$ be a non-trivial group, $m\geqslant 2$, let $\dsl Tm$ be the diagonal semilattice and 
  $\Gamma_D(T,m)$ the diagonal graph. Then the following are valid.
  \begin{enumerate}
    \item The automorphism group of $\dsl Tm$ is $D(T,m)$.
    \item If $(|T|,m)\not\in\{(2,2),(3,2),(4,2),(2,3)\}$, then the automorphism group of $\Gamma_D(T,m)$ is $D(T,m)$. 
  \end{enumerate} 
\end{theorem}

\subsection{History}

The celebrated O'Nan--Scott Theorem describes the socle (the product of the
minimal normal subgroups) of a finite permutation group. Its original form
was different; it was a necessary condition for a finite permutation group
of degree~$n$ to be a maximal subgroup of the symmetric or alternating
group of degree~$n$. Since the maximal intransitive and imprimitive subgroups
are easily described, attention focuses on the primitive maximal subgroups.

The theorem was proved independently by Michael O'Nan and Leonard Scott,
and announced by them at the Santa Cruz conference on finite groups in 1979.
(Although both papers appeared in the preliminary conference proceedings, the
final published version contained only Scott's paper.) However, the roots
of the theorem are much older; a partial result appears in Jordan's
\textit{Trait\'e des Substitutions} \cite{jordan}
in 1870.  The extension to arbitrary primitive groups is due to Aschbacher
and Scott~\cite{aschsc} and independently to Kov\'acs~\cite{kov:sd}. Further
information on the history of the theorem is given in 
\cite[Chapter 7]{ps:cartesian} and~\cite[Sections~1--4]{kovacs}.

For our point of view, and avoiding various complications, the theorem
can be stated as follows:

\begin{theorem}\label{thm:ons}
Let $G$ be a primitive permutation group on a finite set $\Omega$. Then one
of the following four conditions holds:
\begin{enumerate}
\item $G$ is contained in an affine group $\AGL(d,p)\leqslant\Sym(\Omega)$,
with $d\geqslant1$ and  $p$ prime, and so preserves the affine geometry of
dimension $d$ over the field with $p$ elements with point set $\Omega$;
\item $G$ is contained in a wreath product in its product action contained in
$\Sym(\Omega)$, and so preserves a Cartesian decomposition of $\Omega$;
\item $G$ is contained in the diagonal group $D(T,m)\leqslant\Sym(\Omega)$,
with $T$ a non-abelian finite simple group and $m\geqslant1$;
\item $G$ is almost simple (that is, $T\leqslant G\leqslant\Aut(T)$, where $T$
is a non-abelian finite simple group). 
\end{enumerate}
\end{theorem}

Note that, in the first three cases of the theorem, the action of the group
is specified; indeed, in the first two cases, we have a geometric or
combinatorial structure which is preserved by the group. (Cartesian
decompositions are described in detail in~\cite{ps:cartesian}.) One of our
aims in this paper is to provide a similar structure preserved by diagonal
groups, although our construction is not restricted to the case where $T$ is
simple, or even finite.

It is clear that the Classification of Finite Simple Groups had a great
effect on the applicability of the O'Nan--Scott Theorem to the study of
finite primitive permutation groups; indeed, the landscape of the subject
and its applications has been completely transformed by CFSG.

In Section~\ref{s:pqp} we characterise primitive and quasiprimitive diagonal
groups as follows.

\begin{theorem}\label{th:primaut}
  Suppose that $T$ is a non-trivial group, $m\geqslant 2$, and consider $D(T,m)$
  as a permutation group on $\Omega=T^{m}$. Then the following
  are equivalent.
  \begin{enumerate}
  \item $D(T,m)$ is a primitive permutation group;
  \item $D(T,m)$ is a quasiprimitive permutation group;
  \item $T$ is a characteristically simple group, and, if $T$ is
    an elementary abelian $p$-group, then $p\nmid(m+1)$.
    \end{enumerate}
\end{theorem}

Diagonal groups and the structures they preserve have occurred in other
places too. Diagonal groups with $m=1$ (which in fact are not covered by
our analysis) feature in the paper ``Counterexamples to a theorem of Cauchy''
by Peter Neumann, Charles Sims and James Wiegold~\cite{nsw}, while
diagonal groups over the group $T=C_2$ are automorphism groups of the
\emph{folded cubes}, a class of distance-transitive graphs, see~\cite[p.~264]{bcn}.

Much less explicit information is available about related questions on infinite symmetric groups. 
Some maximal subgroups of infinite symmetric groups have been associated
with structures such as subsets, partitions~\cite{braziletal,macn,macpr},
and Cartesian decompositions~\cite{covmpmek}.  
However, it is still not known if infinite symmetric groups have 
maximal subgroups that are analogues of the maximal subgroups of simple
diagonal type in finite symmetric or alternating groups. If $T$ is a possibly
infinite simple group,  then the diagonal group $D(T,m)$ is primitive and,
by~\cite[Theorem~1.1]{uniform}, it cannot be embedded into a wreath product in
product action. On the other hand, if $\Omega$ is a countable set, then, by 
\cite[Theorem~1.1]{macpr}, simple diagonal type groups are properly contained
in maximal subgroups of $\Sym(\Omega)$. (This containment is proper since the
diagonal group itself is not maximal; its product with the finitary symmetric
group properly contains it.)

\section{Preliminaries}
\label{sec:prelim}

\subsection{The lattice of partitions}
\label{sec:part}
A partially ordered set (often abbreviated to \textit{poset}) is a set
equipped with a partial order, which we here write as $\preccurlyeq$.
A finite poset
is often represented by a \emph{Hasse diagram}.
This is a diagram drawn as a graph in the plane. The vertices of the diagram
are the elements of the poset; if $q$ \emph{covers} $p$ (that is, if $p\prec q$
but there is no element $r$ with $p \prec r \prec q$),
there is an edge joining $p$ to~$q$,
with $q$ above $p$ in the plane (that is, with larger $y$-coordinate).
Figure~\ref{f:hasse} represents the divisors of $36$, ordered by divisibility.

\begin{figure}[htbp]
\begin{center}
\setlength{\unitlength}{1mm}
\begin{picture}(20,20)
\multiput(0,10)(5,-5){3}{\circle*{2}}
\multiput(5,15)(5,-5){3}{\circle*{2}}
\multiput(10,20)(5,-5){3}{\circle*{2}}
\multiput(0,10)(5,-5){3}{\line(1,1){10}}
\multiput(0,10)(5,5){3}{\line(1,-1){10}}
\end{picture}
\end{center}
\caption{\label{f:hasse}A Hasse diagram}
\end{figure}
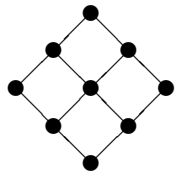

In a partially ordered set with order relation $\preccurlyeq$,
we say that an element $c$ is the \emph{meet}, or \emph{infimum},
of $a$ and $b$ if
\begin{itemize}
\renewcommand{\itemsep}{0pt}
\item $c\preccurlyeq a$ and $c\preccurlyeq b$;
\item for all $d$, $d\preccurlyeq a$ and $d\preccurlyeq b$ implies
  $d\preccurlyeq c$.
\end{itemize}
The meet of $a$ and $b$, if it exists, is unique; we write it $a\wedge b$.

Dually, $x$ is the \emph{join}, or \emph{supremum} of $a$ and $b$ if
\begin{itemize}
\item $a\preccurlyeq x$ and $b\preccurlyeq x$;
\item for all $y$, if $a\preccurlyeq y$ and $b\preccurlyeq y$,
  then $x\preccurlyeq y$.
\end{itemize}
Again the join, if it exists, is unique, and is written $a\vee b$.

The terms ``join'' and ``supremum'' will be used interchangeably. 
Likewise, so will  the terms ``meet'' and ``infimum''.

In an arbitrary poset, meets and joins may not exist. A poset in which every
pair of elements has a meet and a join is called a \emph{lattice}.
A subset of a lattice which is closed under taking joins is called a
\emph{join-semilattice}.

The poset shown in Figure~\ref{f:hasse} is a lattice. Taking it as described
as the set of divisors of $36$ ordered by divisibility, meet and join are
greatest common divisor and least common multiple respectively.

In a lattice, an easy induction shows that suprema and infima of arbitrary
finite sets exist and are unique. In particular, in a finite lattice there is
a unique minimal element and a unique maximal element. (In an infinite lattice,
the existence of least and greatest elements is usually assumed. But all
lattices in this paper will be finite.)

\medskip

The most important example for us is the \emph{partition lattice} on a set
$\Omega$, whose elements are all the partitions of $\Omega$.  There are
(at least) three different ways of thinking about partitions.  In one
approach, used in \cite{rab:as,pjc:ctta,ps:cartesian},
a partition of
$\Omega$ is a set $P$ of pairwise disjoint subsets of $\Omega$, called \textit{parts}
or \textit{blocks}, whose union is $\Omega$.
For $\omega$ in $\Omega$, we write $P[\omega]$ for the unique part of $P$
which contains~$\omega$.

A second approach uses equivalence relations. The ``Equivalence Relation
Theorem'' \cite[Section 3.8]{pjc:ctta} asserts that, if $R$ is an equivalence
relation on a set~$\Omega$, then the equivalence classes of~$R$ form a partition
of~$\Omega$. Conversely, if $P$~is a partition of~$\Omega$ then there is a
unique equivalence relation~$R$ whose equivalence classes are the parts of~$P$.
We call $R$ the \textit{underlying equivalence relation} of~$P$. We write
$x\equiv_Py$ to mean that $x$ and $y$ lie in the same part of~$P$ (and so are
equivalent in the corresponding relation).

The third approach to partitions, as kernels of functions,
is explained near the end of this subsection.

The ordering on partitions is given by
\begin{quote}
$P\preccurlyeq Q$ if and only if every part of $P$ is contained in a part of $Q$.
\end{quote}
Note that $P\preccurlyeq Q$ if and only if $R_P\subseteq R_Q$, where $R_P$
and $R_Q$ are the equivalence relations corresponding to $P$ and $Q$, and
a relation is regarded as a set of ordered pairs.

For any two partitions $P$ and $Q$, the parts of $P\wedge Q$ are all
\emph{non-empty} intersections of a part of $P$ and a part of $Q$. The join
is a little harder to define. The two elements $\alpha$, $\beta$ in $\Omega$
lie in the same part of $P\vee Q$ if and only if there is a finite sequence
$(\omega_0,\omega_1,\ldots,\omega_m)$ of elements of $\Omega$,
with $\omega_0=\alpha$ and $\omega_m=\beta$,  such that $\omega_i$ and
$\omega_{i+1}$ lie in the same part of $P$ if $i$ is even, and
in the same part of $Q$ if $i$ is odd. In other words, there is a walk of finite
length from $\alpha$ to~$\beta$ in which each step remains within a part of
either $P$ or~$Q$.

In the partition lattice on $\Omega$, the unique least element is the partition
(denoted by $E$) with all parts of size~$1$, 
and the unique greatest element (denoted by $U$) is
the partition with a single part $\Omega$. 
In a sublattice
of this, we shall call an element \textit{minimal} if it is minimal subject
to being different from~$E$.

(Warning: in some of the literature that we cite, this partial order is
written as~$\succcurlyeq$.  Correspondingly, the Hasse diagram is the other
way up and the meanings of $\wedge$ and $\vee$ are interchanged.)

For a partition~$P$, we denote by $|P|$ the number of parts of~$P$.
For example, $|P|=1$ if and only if $P=U$.  In the infinite case, we interpret
$|P|$ as the cardinality of the set of parts of~$P$.

There is a connection between partitions and functions which will be important
to us. Let $F\colon\Omega\to\mathcal{T}$ be a function, where $\mathcal{T}$
is an auxiliary set. We will assume, without loss of generality,
that $F$ is onto. Associated with~$F$ is a partition of $\Omega$,
sometimes denoted by $\widetilde F$, whose
parts are the inverse images of the elements of $\mathcal{T}$; in other words,
two points of $\Omega$ lie in the same part of~$\widetilde F$ if and only if they
have the same image under~$F$.  In areas of algebra such as semigroup theory
and universal algebra, the partition~$\widetilde F$ is  referred to as the
\emph{kernel} of $F$.

This point of view is common in experimental design in statistics, where
$\Omega$~is the set of experimental units, $\mathcal{T}$~the set of treatments
being compared, and $F(\omega)$~is the treatment applied to the unit~$\omega$:
see~\cite{rab:design}.
For example, an element $\omega$ in $\Omega$ might be a plot in an agricultural
field, or a single run of an industrial machine, or one person for one month.
The outcomes to be measured are thought of as functions on $\Omega$,
but variables like $F$ which partition~$\Omega$ in ways that may
affect the outcome are called \textit{factors}.  If $F$ is a factor, then the
values $F(\omega)$, for $\omega$ in $\Omega$, are called \textit{levels}
of~$F$.  In this context,
usually no distinction is made between the function~$F$ and the
partition $\widetilde F$ of $\Omega$ which it defines.

If $F\colon\Omega\to\mathcal{T}$ and $G\colon\Omega\to\mathcal{S}$ are two
functions on $\Omega$, then the partition $\widetilde F\wedge\widetilde G$ is the
kernel of the function $F\times G\colon\Omega\to\mathcal{T}\times\mathcal{S}$,
where $(F\times G)(\omega)=(F(\omega),G(\omega))$. In other words,
$\widetilde{F\times G}=\widetilde{F}\wedge\widetilde{G}$.

\begin{defn}
One type of partition which we make use of is the (right) \emph{coset
partition} of a group relative to a subgroup. Let $H$ be a subgroup of a 
group~$G$, and let $P_H$ be the partition of $G$ into right cosets of $H$.
\end{defn}

We gather a few basic properties of coset partitions.
\begin{prop}
\label{prop:coset}
  \begin{enumerate}
\item
If $H$ is a normal subgroup of $G$, then $P_H$ is the kernel (in the general
sense defined earlier) of the natural homomorphism from $G$ to $G/H$.
\item
$P_H\wedge P_K=P_{H\cap K}$.
\item 
$P_H\vee P_K=P_{\langle H,K\rangle}$.
\item
The map $H\mapsto P_H$ is an isomorphism from the lattice of subgroups of~$G$
to a sublattice of the partition lattice on~$G$.
\end{enumerate}
\end{prop}

\begin{proof}
(a) and (b) are clear. (c) holds because elements of $\langle H,K\rangle$
are composed of elements from $H$ and $K$. Finally, (d) follows from (b) and
(c) and the fact that the map is injective. 
\end{proof}

Subgroup lattices of groups have been extensively investigated: see, for
example, Suzuki~\cite{suzuki:book}.

\subsection{Latin squares}
\label{sec:LS}

A \emph{Latin square} of order~$n$ is usually defined as an $n\times n$
array~$\Lambda$ with entries from an alphabet~$T$ of size~$n$
with the property that each letter in~$T$ occurs once in each row and once
in each column of~$\Lambda$.

The diagonal structures in this paper can be regarded as generalisations, where
the dimension is not restricted to be $2$, and the alphabet is allowed to be
infinite. To ease our way in, we re-formulate the definition as follows. For
this definition we regard $T$ as indexing the rows and columns as well as the
letters. This form of the definition allows the structures to be infinite.

A \emph{Latin square} consists of a pair of sets $\Omega$ and $T$, together
with three functions $F_1,F_2,F_3\colon\Omega\to T$, with the property that, if
$i$ and $j$ are any two of $\{1,2,3\}$, the map
$F_i\times F_j\colon\Omega\to T\times T$ is a bijection.

We recover the original definition by specifying that the $(i,j)$ entry
of~$\Lambda$ is equal to~$k$ if the unique point $\omega$ of $\Omega$ for which
$F_1(\omega)=i$ and $F_2(\omega)=j$ satisfies $F_3(\omega)=k$. Conversely,
given the original definition, if we index rows and columns with $T$, then
$\Omega$ is the set of cells of the array, and $F_1,F_2,F_3$ map a cell to its
row, column, and entry respectively.

In the second version of the definition,
the set~$T$ acts as an index set for rows, columns and
entries of the square. We will need the freedom to change the indices
independently; so we now rephrase the definition in terms of the
three partitions $P_i=\widetilde F_i$ ($i=1,2,3$).

Two partitions $P_1$ and $P_2$ of $\Omega$ form a \emph{grid} if, 
for all $p_i\in P_i$ ($i=1,2$), there is a unique point of $\Omega$ lying in
both $p_1$ and $p_2$. In other words, there is a bijection $F$ from
$P_1\times P_2$ to $\Omega$ so that $F(p_1,p_2)$ is the unique point in
$p_1\cap p_2$. This implies that $P_1\wedge P_2=E$ and $P_1\vee P_2=U$, but
the converse is not true.
For example, if $\Omega = \{1,2,3,4,5,6\}$ the partitions
$P_1 =\{\{1,2\},\{3,4\},\{5,6\}\}$ and $P_2=\{\{1,3\}, \{2,5\}, \{4,6\}\}$
have these properties but do not form a grid.

Three partitions $P_1,P_2,P_3$ of $\Omega$ form a \emph{Latin square} if
any two of them form a grid. 

This third version of the definition is the one that we shall mostly use
in this paper.

\begin{prop}
\label{p:order}
If $\{P_1,P_2,P_3\}$ is a Latin square on $\Omega$, then $|P_1|=|P_2|=|P_3|$,
and this cardinality is also the cardinality of any part of any of the three
partitions.
\end{prop}

\begin{proof}
Let $F_{ij}$ be the bijection from $P_i\times P_j$ to $\Omega$, for
$i,j\in\{1,2,3\}$, $i\ne j$.
For any part~$p_1$ of~$P_1$,
there is a bijection $\phi$ between $P_2$ and~$p_1$:
simply put $\phi(p_2) = F_{12}(p_1,p_2) \in p_1$ for  each part $p_2$ of~$P_2$.
Similarly there is a bijection~$\psi$ between $P_3$ and $p_1$
defined by $\psi(p_3) = F_{13}(p_1,p_3) \in p_1$ for  each part $p_3$ of~$P_3$.
Thus $|P_2|=|P_3|=|p_1|$, and $\psi^{-1}\phi$ is an explicit bijection
from $P_2$ to $P_3$.
Similar bijections are defined by any part~$p_2$ of $P_2$ and any part~$p_3$
of~$P_3$.
The result follows. 
\end{proof}

The three partitions are usually called \emph{rows}, \emph{columns} and
\emph{letters}, and denoted by $R,C,L$ respectively. This refers to the
first definition of the Latin square as a square array of letters. Thus,
the Hasse diagram of the three partitions is shown in Figure~\ref{f:ls}.

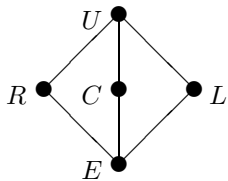
\begin{figure}[htbp]
\begin{center}
\setlength{\unitlength}{1mm}
\begin{picture}(30,30)
\multiput(15,5)(0,20){2}{\circle*{2}}
\multiput(5,15)(10,0){3}{\circle*{2}}
\multiput(5,15)(10,10){2}{\line(1,-1){10}}
\multiput(5,15)(10,-10){2}{\line(1,1){10}}
\put(15,5){\line(0,1){20}}
\put(10,3){$E$}
\put(10,13){$C$}
\put(10,23){$U$}
\put(0,13){$R$}
\put(27,13){$L$}
\end{picture}
\end{center}
\caption{\label{f:ls}A Latin square}
\end{figure}

The number defined in Proposition~\ref{p:order} is called the \emph{order} of
the Latin square. So, with our
second definition, the order of the Latin square is $|T|$.

\medskip

Note that the number of Latin squares of order $n$ grows faster than the
exponential of $n^2$, and the vast majority of these (for large $n$) are
not Cayley tables of groups. We digress slightly to discuss this.

The number of Latin squares of order $n$ is a rapidly growing function, so
rapid that allowing for paratopism (the natural notion of isomorphism for
Latin squares, regarded as sets of partitions; see before
Theorem~\ref{thm:albert} for the definition) does not affect the leading
asymptotics. There is
an elementary proof based on Hall's Marriage Theorem that the number is at least
\[n!(n-1)!\cdots1!\geqslant(n/c)^{n^2/2}\]
for a constant $c$. The van der Waerden permanent conjecture (proved by
Egory\v{c}ev and Falikman~\cite{e:vdwpc,f:vdwpc}) improves the
lower bound to $(n/c)^{n^2}$. An elementary argument using only Lagrange's
and Cayley's Theorems shows that the number of groups of order $n$ is much
smaller; the upper bound is $n^{n\log n}$. This has been improved to
$n^{(c\log n)^2}$ by Neumann~\cite{pmn:enum}. (His theorem was conditional on a
fact about finite simple groups, which follows from the classification of these
groups.) The elementary arguments referred to, which suffice for our claim,
can be found in \cite[Sections~6.3,~6.5]{pjc:ctta}.

Indeed, much more is true: almost all Latin squares have trivial
autoparatopism groups~\cite{pjc:asymm,mw}, whereas
the autoparatopism group of the Cayley table of a group of order~$n$
is the diagonal group, which has order at least
$6n^2$, as we shall see at the end of Section~\ref{sect:lsautgp}.

\medskip

There is a graph associated with a Latin square, as follows:  see
\cite{bose:SRG,pjc:rsrg,phelps}. The
vertex set is $\Omega$; two vertices are adjacent if they lie in the same part
of one of the partitions $P_1,P_2,P_3$. (Note that, if points lie in the same
part of more than one of these partitions, then the points are equal.)
This is the \emph{Latin-square graph} associated with the Latin square.
In the finite case,
if $|T|=n$, then it is a  regular graph with $n^2$~vertices, valency
$3(n-1)$, in which two adjacent vertices have $n$~common neighbours and two
non-adjacent vertices have $6$ common neighbours.
Any regular finite graph with the property that the number of common
neighbours of vertices $v$ and $w$ depends only on whether or not $v$ and $w$
are adjacent is called \textit{strongly regular}: see \cite{bose:SRG,pjc:rsrg}.
Its parameters are the number of vertices, the valency, and the numbers of
common neighbours of adjacent and non-adjacent vertices respectively. Indeed,
Latin-square graphs form one of the most prolific classes of strongly regular
graphs: the number of such graphs on a square number of vertices grows faster
than exponentially, in view of Proposition~\ref{p:lsgraphaut} below.

A \emph{clique} is a set of vertices, any two adjacent;
a \textit{maximum clique} means a maximal clique (with respect to inclusion)
such that there is no clique of strictly larger size.  Thus a maximum clique
must be maximal, but the converse is not necessarily true. The following
result is well-known; we sketch a proof.

\begin{prop}
A Latin square of order $n>4$ can be recovered uniquely from its Latin-square
graph, up to the order of the three partitions and permutations of the rows,
columns and letters.
\label{p:lsgraphaut}
\end{prop}

\begin{proof}
If $n>4$, then any clique of size greater than~$4$ is contained in a unique
clique which is a part of one of the three partitions~$P_i$ for
$i=1,2,3$. In particular, the maximum cliques are the parts of the three
partitions.

Two maximum cliques are parts of the same partition if and only if they are
disjoint (since parts of different partitions intersect in a unique point).
So we can recover the three partitions $P_i$ ($i=1,2,3$) uniquely up to order.
\end{proof}

This proof shows why the condition $n>4$ is necessary. Any Latin-square graph
contains cliques of size $3$ consisting of three cells, two in the same row,
two in the same column, and two having the same entry; and there may also be
cliques of size $4$ consisting of the cells of an \emph{intercalate}, a 
Latin subsquare of order~$2$.

We examine what happens for $n\leqslant 4$.

\begin{itemize}
  \renewcommand{\itemsep}{0pt}
\item For $n=2$, the unique Latin square is the Cayley table of the group $C_2$;
its Latin-square graph is the complete graph $K_4$.
\item For $n=3$, the unique Latin square is the Cayley table of $C_3$. The
Latin-square graph is the complete tripartite graph $K_{3,3,3}$: the nine
vertices are partitioned into three parts of size~$3$, and the edges join all
pairs of points in different parts.
\item For $n=4$, there are two Latin squares up to isotopy, the Cayley tables
  of the Klein group and the cyclic group. Their Latin-square graphs
  are most easily identified by
looking at their complements, which are strongly regular graphs on $16$ points
with parameters $(16,6,2,2)$: that is, all vertices have valency~$6$, and any
two vertices have just two common neighbours. Shrikhande~\cite{shrikhande}
showed that there are exactly two such graphs: the $4\times4$ square lattice
graph, sometimes written as $L_2(4)$, which is the line graph~$L(K_{4,4})$
of the complete bipartite graph $K_{4,4}$; and one further graph now called
the \emph{Shrikhande graph}. See Brouwer~\cite{Brouwer} for a detailed
description of this graph.
\end{itemize}

Latin-square graphs were introduced in two seminal papers by Bruck and Bose
in the \emph{Pacific Journal of Mathematics} in 1963~\cite{bose:SRG,Bruck:net}.
A special case of Bruck's main result is that a strongly regular graph having
the parameters $(n^2, 3(n-1), n, 6)$ associated with a Latin-square graph of
order~$n$ must actually be a Latin-square graph, provided that $n>23$.

\subsection{Quasigroups}
\label{sesc:quasi}

A \emph{quasigroup} consists of a set $T$ with a binary operation $\circ$ in
which each of the equations $a\circ x=b$ and $y\circ a=b$ has a unique solution
$x$ or $y$ for any given $a,b\in T$. These solutions are denoted by
$a\backslash b$ and $b/a$ respectively.

According to the second of our three equivalent definitions,
a quasigroup $(T,\circ)$ gives rise to a Latin square 
$(F_1,F_2,F_3)$ by the rules that $\Omega=T\times T$ and,
for $(a,b)$ in $\Omega$,
$F_1(a,b)=a$, $F_2(a,b)=b$, and $F_3(a,b)=a\circ b$.
Conversely, a Latin square with rows, columns and letters indexed by a
set $T$ induces a quasigroup structure
on $T$ by the rule that, if we use the pair $(F_1,F_2)$ to identify $\Omega$
with $T\times T$, then $F_3$ maps the pair $(a,b)$ to $a\circ b$. (More
formally, $F_1(\omega)\circ F_2(\omega)=F_3(\omega)$ for all $\omega\in\Omega$.)

In terms of partitions, if $a,b\in T$, and the unique point lying in the part
of $P_1$ labelled $a$ and the part of $P_2$ labelled $b$ also lies in the 
part of $P_3$ labelled~$c$, then $a\circ b=c$.

In the usual representation of a Latin square as a square array, the Latin
square is the \emph{Cayley table} of the quasigroup.

Any permutation of $T$ induces a quasigroup isomorphism, by simply
re\-labelling the elements.  However, the Latin square property is also
preserved if we choose three permutations
$\alpha_1$, $\alpha_2$, $\alpha_3$ of $T$ independently and define new functions
$G_1$, $G_2$, $G_3$ by $G_i(\omega)=(F_i(\omega))\alpha_i$ for $i=1,2,3$.
(Note that we write permutations on the right, but most other functions on
the left.)
Such a triple of maps is called an
\emph{isotopism} of the Latin square or quasigroup.

We can look at this another way. Each map $F_i$ defines a partition
$P_i$ of~$\Omega$, in which two points lie in the same part if their
images under $F_i$ are equal. Permuting elements of the three image sets
independently has no effect on the partitions. So an isotopism class of
quasigroups corresponds to a Latin square (using the partition definition)
with arbitrary labellings of rows, columns and letters by $T$.

A \emph{loop} is a quasigroup with a two-sided identity. Any quasigroup is
isotopic to a loop, as observed by Albert~\cite{albert}: indeed, any element
$e$ of the quasigroup can be chosen to be the identity. (Use the letters in
the row and column of a fixed cell containing $e$ as column, respectively row,
labels.)

A different equivalence on Latin squares is obtained by applying
a permutation to the three functions $F_1,F_2,F_3$. Two Latin squares (or
quasigroups) are said to be \emph{conjugate}~\cite{kd}  or \emph{parastrophic}
\cite{shch:quasigroups} if they are related by such a
permutation. For example, the transposition of $F_1$ and $F_2$ corresponds
(under the original definition) to transposition (as matrix) of the Latin
square. Other conjugations are slightly harder to define: for example, 
the $(F_1,F_3)$ conjugate is the square in which the $(i,j)$ entry is $k$ if
and only if the $(k,j)$ entry of the original square is $i$.

Combining the operations of isotopism and conjugation gives the relation of
\emph{paratopism}. The paratopisms form the group $\Sym(T)\wr S_3$. Given a
Latin square or quasigroup, its \emph{autoparatopism group} is the group of 
all those paratopisms which preserve it, in the sense that they map the set
$\{(x,y,x\circ y):x,y\in T\}$ of triples to itself. This coincides with the
automorphism group of the Latin square (as set of partitions): take $\Omega$
to be the set of triples and let the three partitions correspond to the values
in the three positions. An autoparatopism is called an \emph{autotopism} if it
is an isotopism. See \cite{paratop} for details. 

In the case of groups, a conjugation can be attained
by applying a suitable isotopism, and so the following result is a direct
consequence of Albert's well-known theorem~\cite[Theorem~2]{albert}.

\begin{theorem}\label{thm:albert}
  If $\Lambda$ and $\Lambda'$ are Latin squares, isotopic to Cayley tables
  of groups $G$ and $G'$ respectively, and if some paratopism maps $\Lambda$
  to $\Lambda'$, then the groups $G$ and $G'$ are isomorphic. 
\end{theorem}

Except for a small number of exceptional cases, the autoparatopism group of 
a Latin square coincides with the automorphism group of its Latin-square graph.

\begin{prop}
\label{p:autlsg}
  Let $\Lambda$ be a Latin square of order $n>4$. Then the automorphism group
  of the Latin-square graph of $\Lambda$ is isomorphic to the autoparatopism
  group of~$\Lambda$.
\end{prop}

\begin{proof}
It is clear that autoparatopisms of $\Lambda$ induce automorphisms of its
graph. The converse follows from Proposition~\ref{p:lsgraphaut}. 
\end{proof}

A question which will be of great importance to us is the following: How do
we recognise Cayley tables of groups among Latin squares? The answer is given
by the following theorem, proved in \cite{brandt,frolov}.  We first need
a definition, which is given in the statement of \cite[Theorem~1.2.1]{DK:book}.

\begin{defn}
  \label{def:quad}
  A Latin square satisfies the \textit{quadrangle criterion}, if, for all
  choices of $i_1$, $i_2$, $j_1$, $j_2$, $i_1'$, $i_2'$, $j_1'$ and $j_2'$,
  if the letter in $(i_1,j_1)$ is equal to the letter in $(i_1',j_1')$,
  the letter in $(i_1,j_2)$ is equal to the letter in $(i_1',j_2')$,
  and the letter in $(i_2,j_1)$ is equal to the letter in $(i_2',j_1')$,
  then the letter in $(i_2,j_2)$ is equal to the letter in $(i_2',j_2')$.
\end{defn}

In other words, any pair of rows and pair of columns define four entries in the
Latin square; if two pairs of rows and two pairs of columns have the property that
three of the four entries are equal, then the fourth entries are also equal.
If $(T,\circ)$ is a quasigroup, it satisfies the quadrangle criterion if and
only if, for any $a_1,a_2,b_1,b_2,a_1',a_2',b_1',b_2'\in T$, if
$a_1\circ b_1=a_1'\circ b_1'$, $a_1\circ b_2=a_1'\circ b_2'$, and
$a_2\circ b_1=a_2'\circ b_1'$, then $a_2\circ b_2=a_2'\circ b_2'$.

\begin{theorem}
\label{thm:frolov}
  Let $(T,\circ)$ be a quasigroup. Then $(T,\circ)$ is isotopic to a group if
  and only if it satisfies the quadrangle criterion. 
\end{theorem}

In \cite{DK:book}, the ``only if'' part of this result is proved in its
Theorem 1.2.1 and the converse is proved in the text following Theorem~1.2.1.

A Latin square which satisfies the quadrangle criterion is called a
\textit{Cayley matrix} in~\cite{DOP:quad}.

If $(T, \circ)$ is isotopic to a group then we may assume that the rows,
columns and letters have been labelled in such a way that $a \circ b= a^{-1}b$
for all $a$, $b$ in~$T$.  We shall use this format in the proof of
Theorems~\ref{t:autDT2} and~\ref{thm:bingo}.

\subsection{Automorphism groups}\label{sect:lsautgp}

Given a Latin square $\Lambda=\{R,C,L\}$ on a set $\Omega$, an
\emph{automorphism} of $\Lambda$ is a permutation of $\Omega$ preserving
the set of three partitions; it is a \emph{strong automorphism} if it
fixes the three partitions individually. (These maps are also called
\emph{autoparatopisms} and \emph{autotopisms}, as noted in the preceding
section.)
We will generalise this definition later, in Definition~\ref{def:weak}.
We denote the groups of automorphisms and strong automorphisms by
$\Aut(\Lambda)$ and $\SAut(\Lambda)$ respectively.

In this section we verify that, if $\Lambda$ is the Cayley table of a group
$T$, then $\Aut(\Lambda)$ is the diagonal group $D(T,2)$ defined in 
Section~\ref{sect:diaggroups}.

We begin with a principle which we will use several times.

\begin{prop}
  Suppose that the group $G$ acts transitively on a set~$\Omega$.
    Let $H$ be a subgroup of $G$, and assume that
\begin{itemize}
      \renewcommand{\itemsep}{0pt}
\item $H$ is also transitive on $\Omega$;
\item $G_\alpha=H_\alpha$, for some $\alpha\in\Omega$.
\end{itemize}
Then $G=H$.
\label{p:subgp}
\end{prop}

\begin{proof}
The transitivity of $H$ on $\Omega$ means that we can choose a set $X$ of
coset representatives for $G_\alpha$ in $G$ such that $X\subseteq H$. Then
$H=\langle H_\alpha,X\rangle=\langle G_\alpha,X\rangle=G$. 
\end{proof}

The next result applies to any Latin square. As noted earlier, given a
Latin square $\Lambda$, there is a loop $Q$ whose Cayley table is $\Lambda$.

\begin{prop}
Let $\Lambda$ be the Cayley table of a loop $Q$ with identity $e$. Then
the subgroup $\SAut(\Lambda)$ fixing the cell in row and
column $e$ is equal to the automorphism group of $Q$.
\label{p:autlatin}
\end{prop}

\begin{proof}
A strong automorphism of $\Lambda$ is given by an isotopism $(\rho,\sigma,\tau)$
of $Q$, where $\rho$, $\sigma$, and $\tau$ are permutations of rows, columns
and letters, satisfying
\[(ab)\tau=(a\rho)(b\sigma)\]
for all $a,b\in Q$. If this isotopism fixes the element $(e,e)$ of $\Omega$,
then substituting
$a=e$ in the displayed equation shows that $b\tau=b\sigma$ for all $b\in Q$,
and so $\tau=\sigma$. Similarly, substituting $b=e$ shows that $\tau=\rho$.
Now the displayed equation shows that $\tau$ is an automorphism of $Q$.

Conversely, if $\tau$ is an automorphism of $Q$, then $(\tau,\tau,\tau)$ is
a strong automorphism of $\Lambda$ fixing the cell $(e,e)$. 
\end{proof}

\begin{theorem}
Let $\Lambda$ be the Cayley table of a group $T$. Then $\Aut(\Lambda)$
is the diagonal group $D(T,2)$.
\label{t:autDT2}
\end{theorem}

\begin{proof}
First, we show that $D(T,2)$ is a subgroup of $\Aut(\Lambda)$. 
 We take $\Omega=T\times T$ and
 represent $\Lambda=\{R,C,L\}$ as follows, using notation introduced in
 Section~\ref{sec:part}:
\begin{itemize}
    \renewcommand{\itemsep}{0pt}
\item $(x,y)\equiv_R(u,v)$ if and only if $x=u$;
\item $(x,y)\equiv_C(u,v)$ if and only if $y=v$;
\item $(x,y)\equiv_L(u,v)$ if and only if $x^{-1}y=u^{-1}v$.
\end{itemize}
(As an array, we take the $(x,y)$ entry to be $x^{-1}y$. As noted at the end
of Section~\ref{sesc:quasi}, this
is isotopic to the usual representation of the Cayley table.)

Routine verification shows that the generators of $D(T,2)$ given in
Section~\ref{sect:diaggroups} of types (I)--(III) preserve these
relations, while the map $(x,y)\mapsto(y,x)$ interchanges $R$ and $C$
while fixing $L$, and the map $(x,y)\mapsto(x^{-1},x^{-1}y)$ interchanges $C$
and $L$ while fixing $R$. (Here is one case: the element $(a,b,c)$ in $T^3$ maps
$(x,y)$ to $(a^{-1}xb,a^{-1}yc)$. If $x=u$ then $a^{-1}xb=a^{-1}ub$, and
if $x^{-1}y=u^{-1}v$ then $(a^{-1}xb)^{-1}a^{-1}yc=(a^{-1}ub)^{-1}a^{-1}vc$.)
Thus $D(T,2)\leqslant\Aut(\Lambda)$.

Now we apply Proposition~\ref{p:subgp} in two stages.
\begin{itemize}
\item First, take $G=\Aut(\Lambda)$ and $H=D(T,2)$. Then $G$ and $H$ both induce
$S_3$ on the set of three partitions; so it suffices to prove that the
group of strong automorphisms of $\Lambda$ is generated by elements of
types (I)--(III) in $D(T,2)$.
\item Second, take $G$ to be $\SAut(\Lambda)$,
and $H$ the group generated by translations and automorphisms of $T$
(the elements of type (I)--(III) in Remark~\ref{rem:diaggens}). Both $G$
and $H$ act transitively on $\Omega$, so it is enough to show that the
stabilisers of a cell (which we can take to be $(1,1)$) in $G$ and $H$ are
equal. Consideration of elements of types (I)--(III)
shows that $H_{(1,1)}=\Aut(T)$,
while Proposition~\ref{p:autlatin} shows that $G_{(1,1)}=\Aut(T)$.
\end{itemize}
The statement at the end of the second stage completes the proof. 
\end{proof}

It follows from Proposition~\ref{p:lsgraphaut}
that, if $n>4$, the automorphism group of
the Latin-square graph derived from the Cayley table of a group $T$ of order~$n$
is also the diagonal group $D(T,2)$. For $n\leqslant4$, we described the
Latin-square graphs at the end of Section~\ref{sec:LS}. For the groups $C_2$,
$C_3$, and $C_2\times C_2$, the graphs are $K_4$, $K_{3,3,3}$, and 
$L(K_{4,4})$ respectively, with automorphism groups $S_4$,
$S_3\wr S_3$, and $S_4\wr S_2$ respectively. However, the automorphism group
of the Shrikhande graph is the group $D(C_4,2)$, with order $192$.
(The order of the automorphism group is $192$, see Brouwer~\cite{Brouwer},
and it contains $D(C_4,2)$, also with order $192$, as a subgroup.)

It also follows from Proposition~\ref{p:lsgraphaut} that,
if $T$ is a group, then the automorphism group of 
the Latin-square graph is transitive on the vertex set. Vertex-transitivity 
does not, however, 
characterise Latin-square graphs that correspond to groups, as can be 
seen by considering the examples in~\cite{wanlesspage}; the smallest example
which is not a group has order~$6$.

Finally, we justify the assertion made earlier, that the Cayley table of a
group of order $n$, as a Latin square, has at least $6n^2$ automorphisms. By 
Theorem~\ref{t:autDT2}, this automorphism group is the diagonal group
$D(T,2)$; this group has a quotient $S_3$ acting on the three partitions, and
the group of strong automorphisms contains the right multiplications by 
elements of $T^2$.

\subsection{More on partitions}\label{sect:moreparts}

Most of the work that we cite in this subsection has been about partitions of
finite sets.
See \cite[Sections 2--4]{rab:BCC} for a recent summary of this material.

\begin{defn}
  \label{def:uniform}
  A partition~$P$ of a  set~$\Omega$ is \emph{uniform} if all its
  parts have the same size in the sense that, whenever $\Gamma_1$
  and $\Gamma_2$ are parts of $P$,  there is a bijection from $\Gamma_1$
  onto $\Gamma_2$.
  \end{defn}

Many other words are used for this property for finite sets $\Omega$. 
Tjur \cite{tjur84,tjur91} calls
such a partition \emph{balanced}.  Behrendt \cite{behr} calls them
\emph{homogeneous}, but this conflicts with the use of this word
in \cite{ps:cartesian}.  Duquenne \cite{duq} calls them \textit{regular}, 
as does Aschbacher~\cite{asch_over}, while Preece \cite{DAP:Oz} calls them
\emph{proper}.

Statistical work has made much use of the notion of orthogonality between
pairs of partitions.  Here we explain it in the finite case, before
attempting to find a generalisation that works for infinite sets.

When $\Omega$ is finite, let $V$ be the real vector space $\mathbb{R}^\Omega$
with the usual inner product.  Subspaces $V_1$ and $V_2$ of $V$ are defined
in \cite{tjur84} to be \textit{geometrically orthogonal} to each other if
$V_1 \cap(V_1 \cap V_2)^\perp \perp V_2 \cap(V_1\cap V_2)^\perp$.
This is equivalent to saying that the matrices $M_1$ and $M_2$ of orthogonal
projection onto $V_1$ and $V_2$ commute.
If $V_i$ is the set of vectors which are constant on each part of partition
$P_i$ then we say that partition $P_1$ is \textit{orthogonal} to partition $P_2$
if $V_1$ is geometrically orthogonal to $V_2$.

Here are two nice results in the finite case.  See, for example,
\cite[Chapter 6]{rab:as}, \cite[Chapter 10]{rab:design} and \cite{tjur84}.

\begin{theorem}
  For $i=1$, $2$, let $P_i$ be a partition of the finite set $\Omega$ with
  projection matrix $M_i$.  If $P_1$ is orthogonal to $P_2$ then the matrix
  of orthogonal projection onto the  subspace consisting of those
  vectors which are constant on each part of the partition $P_1 \vee P_2$ is
  $M_1M_2$. 
  \end{theorem}

\begin{theorem}
  \label{thm:addon}
  If $P_1$, $P_2$ and $P_3$ are pairwise orthogonal partitions of a finite
  set $\Omega$ then $P_1\vee P_2$ is orthogonal to $P_3$. 
\end{theorem}

Let $\mathcal{S}$ be a set of partitions of $\Omega$ which are pairwise
orthogonal. A consequence of Theorem~\ref{thm:addon} is that, if $P_1$ and
$P_2$ are in $\mathcal{S}$, then $P_1 \vee P_2$ can be added to $\mathcal{S}$
without destroying orthogonality.  This is one motivation for the
following definition.

\begin{defn}
  \label{def:tjur}
  A set of partitions of a finite set $\Omega$ is a \emph{Tjur block structure}
  if every pair of its elements is orthogonal, it is closed under taking
  suprema, and it contains $E$.
\end{defn}

Thus the set of partitions in a Tjur block structure forms a join-semi\-lattice.

The following definition is more restrictive, but is widely used by
statisticians, based on the work of many people, including
Nelder \cite{JAN:OBS},
Throckmorton \cite{Thr61} and Zyskind \cite{Zy62}.

\begin{defn}
  A set of partitions of a finite set $\Omega$ is an \emph{orthogonal
    block structure} if it is a Tjur block structure, all of its partitions
  are uniform, it is closed under taking infima, and it contains $U$.
\end{defn}

The set of partitions in an orthogonal block structure forms a lattice.

These notions have been used by combinatorialists and group theorists as
well as statisticians.  For example, as explained in Section~\ref{sec:LS},
a Latin square can be regarded as an orthogonal block structure with the
partition lattice shown in Figure~\ref{f:ls}.

The following theorem shows how subgroups of a group can give rise to a Tjur
block structure: see \cite[Section 8.6]{rab:as} and
Proposition~\ref{prop:coset}(c).

\begin{theorem}
  Given two subgroups $H$, $K$ of a finite group $G$, the partitions
 $P_H$ and $P_K$ into right
cosets of $H$ and $K$ are orthogonal if and only if $HK=KH$ (that is, if and
only if $HK$ is a subgroup of $G$). If this happens, then the join of these
two partitions is the partition $P_{HK}$ into right cosets of $HK$. 
\end{theorem}

An orthogonal block structure is called a \textit{distributive block structure}
or a \textit{poset block structure} if each of $\wedge$ and $\vee$ is
distributive over the other.

The following definition is taken from \cite{rab:as}.

\begin{defn}
  \label{def:weak}
  An \textit{automorphism} of a set of
partitions is a permutation of the underlying set that preserves the set of
partitions.  Such an automorphism is a \textit{strong automorphism} if it
preserves each of the partitions.
  \end{defn}

The group of strong automorphisms of a poset block structure
is a \textit{generalised wreath product} of symmetric groups: see
\cite{GWP,tulliobook}.  One of the aims of the present paper is to
describe the automorphism group of the set of partitions defined by a
diagonal semilattice.

In \cite{CSCPWT}, Cheng and Tsai  state that the desirable properties of
a collection
of partitions of a finite set are that it is a Tjur block structure,
all the partitions are uniform, and it contains $U$.  This sits between Tjur
block structures and orthogonal block structures but does not seem to have been
named.

Of course, this theory needs a notion of inner product.  If the set is
infinite we
would have to consider the vector space whose vectors have all but finitely
many entries zero.  But if $V_i$ is the set of vectors which are constant on
each part of partition $P_i$ and if each part of $P_i$ is infinite then $V_i$
is the zero subspace.  So we need to find a different definition that will
cover the infinite case.

We noted in Section~\ref{sec:part} that each partition is defined by its
underlying equivalence relation.  If $R_1$ and $R_2$ are two equivalence
relations on $\Omega$ then their composition $R_1 \circ R_2$ is the relation
defined by 
\[
  \omega _1 (R_1 \circ R_2) \omega_2\mbox{ if and only if }
  \exists \omega_3\in\Omega\mbox{ such that } \omega_1 R_1 \omega_3\mbox{ and }\omega_3 R_2 \omega_2.
\]

\begin{prop}
\label{prop:commeq}
  Let $P_1$ and $P_2$ be partitions of $\Omega$ with underlying equivalence
  relations $R_1$ and $R_2$ respectively. For each part $\Gamma$ of $P_1$,
  denote by $\mathcal{B}_\Gamma$ the set of parts of $P_2$ whose intersection
  with $\Gamma$ is not empty.
  The following are equivalent.
(Recall that $P[\omega]$ is the part of $P$ containing $\omega$.)
  \begin{enumerate}
  \item
    The equivalence relations $R_1$ and $R_2$ commute with each other
    in the sense that
    $R_1 \circ R_2 = R_2 \circ R_1$.
  \item The relation $R_1 \circ R_2$ is an equivalence relation.
  \item For all $\omega_1$ and $\omega_2$ in $\Omega$, the set
$P_1[\omega_1] \cap P_2[\omega_2]$ is non-empty if and only if the set
      $P_2[\omega_1]\cap P_1[\omega_2]$ is non-empty.
  \item
      Modulo the parts of $P_1 \wedge P_2$, the restrictions of $P_1$
      and $P_2$ to any part of $P_1 \vee P_2$ form a grid.
      In other words, if $\Gamma$ and $\Xi$ are parts of $P_1$ and $P_2$
      respectively, both contained in the same part of $P_1\vee P_2$, then
      $\Gamma \cap \Xi \ne \emptyset$.
  \item For all parts $\Gamma$ and $\Delta$ of $P_1$, the sets
     $\mathcal{B}_\Gamma$ and $\mathcal{B}_\Delta$ are either equal or disjoint.
  \item If $\Gamma$ is a part of $P_1$ contained in a part $\Theta$
      of $P_1\vee P_2$ then $\Theta$ is the union of the parts of $P_2$
      in $\mathcal{B}_\Gamma$. 
    \end{enumerate}
  \end{prop}

In part (d), ``modulo the parts of $P_1\wedge P_2$'' means that, if each of
these parts is contracted to a point, the result is a grid as defined earlier.
In the finite case, if $P_1$ is orthogonal to $P_2$ then their underlying
equivalence relations $R_1$ and $R_2$ commute.

We need a concept that is the same as orthogonality in the
finite case (at least, in the Cheng--Tsai case).

\begin{defn}
  \label{def:compatible}
  Two uniform partitions $P$ and $Q$ of a set $\Omega$ (which may be finite or
  infinite) are \emph{compatible} if
  \begin{enumerate}
  \item their underlying equivalence relations commute, and
    \item their infimum $P\wedge Q$  is uniform.
    \end{enumerate}
\end{defn}

  If the partitions $P$, $Q$ and $R$ of a set $\Omega$ are pairwise
  compatible then the equivalence of statements (a) and (f) of 
  Proposition~\ref{prop:commeq} 
  shows that
  $P\vee Q$ and $R$ satisfy condition~(a) in
  the definition of compatibility.  Unfortunately,  they may not satisfy
  condition~(b), as the following example shows,
  so the analogue of Theorem~\ref{thm:addon} for compatibility is not true in
  general. However, it is true if we restrict attention to join-semilattices
  of partitions where all infima are uniform.  This is the case for
  Cartesian lattices and for semilattices defined
  by diagonal structures (whose definitions follow in
  Sections~\ref{sec:firstcd}  and \ref{sec:diag1} respectively).
  It is also true for group semilattices: if $P_H$ and $P_K$ are the 
  partitions of a group $G$ into right cosets of subgroups $H$ and $K$
  respectively, then $P_H\wedge P_K = P_{H \cap K}$,
  as remarked in Proposition~\ref{prop:coset}.

  \begin{eg}
  \label{eg:badeg}
  Let $\Omega$ consist of the $12$ cells in the three $2 \times 2$ squares
  shown in Figure~\ref{fig:badeg}.  Let $P$ be the partition of $\Omega$
  into six rows, $Q$ the partition into six columns, and $R$ the partition
  into  six letters.

  \begin{figure}
    \[
    \begin{array}{c@{\qquad}c@{\qquad}c}
      \begin{array}{|c|c|}
        \hline
        A & B\\
        \hline
        B & A\\
        \hline
      \end{array}
      &
            \begin{array}{|c|c|}
        \hline
        C & D\\
        \hline
        E & F\\
        \hline
            \end{array}
                  &
            \begin{array}{|c|c|}
        \hline
        C & D\\
        \hline
        E & F\\
        \hline
        \end{array}
      \end{array}
    \]
    \caption{Partitions in Example~\ref{eg:badeg}}
    \label{fig:badeg}
    \end{figure}
  Then $P\wedge Q = P\wedge R = Q \wedge R=E$, so each infimum is uniform.
  The squares are the parts of the supremum $P\vee Q$.
  For each pair of $P$, $Q$ and~$R$, their
  underlying equivalence relations commute.  However, the parts
  of $(P\vee Q)\wedge R$ in the first square have size two, while all of the
  others have size one.
\end{eg}

\section{Cartesian structures}
\label{sec:Cart}

We remarked just before Proposition~\ref{p:order} that three partitions of
$\Omega$ form a Latin square if and only if any two form a grid. The main
theorem of this paper is a generalisation of this fact to higher-dimensional
objects, which can be regarded as Latin hypercubes. Before
we get there, we need to consider the higher-dimensional analogue of grids.

\subsection{Cartesian decompositions and Cartesian lattices}
\label{sec:firstcd}

Cartesian decompositions are defined on \cite[p.~4]{ps:cartesian}. Since we
shall be taking a slightly different approach, we introduce these objects
rather briefly; we show that they are equivalent to those in our approach,
in the sense that each can be constructed from the other in a standard way,
and the automorphism groups of corresponding objects are the same.

\begin{defn}
\label{def:cart}
  A \emph{Cartesian decomposition} of a set~$\Omega$, of dimension~$n$, is a
set~$\mathcal{E}$ of $n$ partitions $P_1,\ldots,P_n$ of $\Omega$ such that
$|P_i|\geqslant2$ for all $i$, and for all $p_i\in P_i$ for $i=1,\ldots,n$,
\[|p_1\cap\cdots\cap p_n|=1.\]
A Cartesian decomposition is \emph{trivial} if $n=1$; in this case $P_1$ is
the partition of $\Omega$ into singletons.
\end{defn}

For the rest of this subsection, $P_1,\ldots,P_n$ form a Cartesian decomposition
of $\Omega$.

\begin{prop}\label{prop:CDbij}
There is a well-defined bijection between $\Omega$ and
$P_1\times\cdots\times P_n$, given by
\[\omega\mapsto(p_1,\ldots,p_n)\]
if and only if $\omega\in p_i$ for $i=1,\ldots,n$.
\end{prop}

For simplicity, we adapt the notation in Section~\ref{sec:part} by
writing $\equiv_i$ for the equivalence relation $\equiv_{P_i}$ underlying
the partition~$P_i$.
For any subset $J$ of the index set $\{1,\ldots,n\}$, define a partition
$P_J$ of $\Omega$ corresponding to the following equivalence relation
$\equiv_{P_J}$ written as $\equiv_J$:
\[\omega_1\equiv_J\omega_2 \Leftrightarrow (\forall i\in J)\ 
\omega_1\equiv_i\omega_2.\]
In other words, $P_J=\bigwedge_{i\in J}P_i$. 

\begin{prop}
\label{p:antiiso}
  For all $J,K\subseteq \{1,\ldots,n\}$, we have
\[P_{J\cup K}=P_J\wedge P_K,\quad\hbox{and}\quad P_{J\cap K}=P_J\vee P_K.\]
Moreover, the equivalence relations $\equiv_J$ and $\equiv_K$ commute with
each other.

\end{prop}

It follows from this proposition that the partitions $P_J$, for
$J\subseteq\{1,\ldots,n\}$, form a lattice (a sublattice of the partition
lattice on $\Omega$), which is anti-isomorphic to the Boolean lattice of
subsets of $\{1,\ldots,n\}$ by the map $J\mapsto P_J$. We call this lattice
the \emph{Cartesian lattice} defined by the Cartesian decomposition.

For more details we refer to the book~\cite{ps:cartesian}.

Following \cite{JAN:OBS},
most statisticians would call such a  lattice a \textit{completely crossed
  orthogonal block structure}:  see \cite{rab:DCC}.
It is called a \textit{complete factorial structure} in \cite{RAB:LAA}.

(Warning: a different common meaning of \textit{Cartesian lattice} is
$\mathbb{Z}^n$: for example, see \cite{Rand:CL}.)

The $P_i$ are the maximal non-trivial elements of this lattice. Our approach is
based on considering the dual description, the minimal non-trivial elements of
the lattice; these are the partitions $Q_1,\ldots,Q_n$, where
\[Q_i=P_{\{1,\ldots,n\}\setminus\{i\}}=\bigwedge_{j\ne i}P_j\]
and $Q_1,\ldots,Q_n$ generate the Cartesian lattice by repeatedly forming 
joins (see Proposition~\ref{p:antiiso}).

\subsection{Hamming graphs and Cartesian decompositions}
\label{sec:HGCD}

The Hamming graph is so-called because of its use in coding theory. The
vertex set is the set of all $n$-tuples over an alphabet $A$; more briefly,
the vertex set is $A^n$. Elements of $A^n$ will be written as
${a}=(a_1,\ldots,a_n)$. Two vertices $a$ and $b$ are joined if
they agree in all but one coordinate, that is, if
there exists~$i$ such that $a_i\ne b_i$ but $a_j=b_j$ for $j\ne i$.
We denote this graph by $\Ham(n,A)$.

The alphabet $A$ may be finite or infinite, but we restrict the number~$n$
to be finite. There is a more general form, involving alphabets
$A_1,\ldots,A_n$; here the $n$-tuples $a$ are required to satisfy $a_i\in A_i$
for $i=1,\ldots,n$ (that is, the vertex set is $A_1\times\cdots\times A_n$);
the adjacency rule is the same. We will call this a \emph{mixed-alphabet
Hamming graph}, denoted $\Ham(A_1,\ldots,A_n)$.

A Hamming graph is connected, and the graph distance between two vertices
$a$ and $b$ is the number of coordinates where they differ:
\[d({a},{b})=|\{i\mid a_i\ne b_i\}|.\]

\begin{theorem}\label{th:cdham}
\begin{enumerate}
\item Given a Cartesian decomposition of~$\Omega$, a unique mixed-alpha\-bet
Hamming graph can be constructed from it.
\item Given a mixed-alphabet Hamming graph on $\Omega$, a unique Cartesian
decomposition of~$\Omega$ can be constructed from it.
\item The Cartesian decomposition and the Hamming graph referred to above
  have the same automorphism group.
\end{enumerate}
\end{theorem}

The constructions from Cartesian decomposition to Hamming graph and back are
specified in the proof below.

\begin{proof}
  Note that the trivial Cartesian decomposition of $\Omega$ corresponds to the complete 
  graph and the automorphism group of both is the symmetric group $\Sym(\Omega)$. 
  Thus in the rest of the proof we assume that the Cartesian decomposition 
  in item~(a) is non-trivial and the Hamming graph in item~(b) is constructed with 
  $n\geqslant 2$. 
  \begin{enumerate}
    \item
  Let $\mathcal{E}=\{P_1,\ldots,P_n\}$ be a Cartesian decomposition
 of $\Omega$ of dimension~$n$:
each $P_i$ is a partition of $\Omega$.  By Proposition~\ref{prop:CDbij},
 there is a bijection $\phi$ from $\Omega$ to
$P_1\times\cdots\times P_n$: a point $a$ in $\Omega$ corresponds to
$(p_1,\ldots,p_n)$, where $p_i$ is the part of $P_i$ containing~$a$.
 Also, by Proposition~\ref{p:antiiso} and the subsequent discussion,
 the minimal partitions in the
Cartesian lattice generated by $P_1,\ldots,P_n$ have the form
\[Q_i=\bigwedge_{j\ne i}P_j\]
for $i=1,\ldots,n$; so $a$ and $b$ in $\Omega$ lie in the same part
of $Q_i$ if their
images under $\phi$ agree in all coordinates except the $i$th. So, if we define
$a$ and $b$ to be adjacent if they are in the same part of $Q_i$ for some
$i$, the resultant graph is isomorphic (by $\phi$) to the mixed-alphabet
Hamming graph on $P_1\times\cdots\times P_n$.

\item
Let $\Gamma$ be a mixed-alphabet Hamming graph on
$A_1\times\cdots\times A_n$. Without loss of generality, $|A_i|>1$ for all $i$
(we can discard any coordinate where this fails). We establish various facts
about $\Gamma$; these facts correspond to the claims on pages 271--276 of~\cite{ps:cartesian}.

Any maximal  clique in $\Gamma$ has the form
\[C({a},i)=\{{b}\in A_1\times\cdots\times A_n\mid b_j=a_j\hbox{ for }j\ne i\},\]
for some ${a}\in\Omega$, $i\in\{1,\ldots,n\}$. Clearly all vertices in
$C({a},i)$ are adjacent in~$\Gamma$. If ${b},{c}$ are distinct vertices in
$C({a},i)$, then
$b_i\ne c_i$, so no vertex outside $C({a},i)$ can be joined to both.
Moreover, if any two vertices are joined, they differ in a unique coordinate
$i$, and so there is some $a$ in $\Omega$ such that
they both lie in $C({a},i)$ for that value of~$i$.
Let $C=C({a},i)$ and $C'=C({b},j)$ be two maximal cliques.
Put $\delta = \min\{d({ x},{ y})\mid { x}\in C,{ y}\in C'\}$.
  \begin{itemize}
  \item 
    If $i=j$, then there is a bijection $\theta\colon C\to C'$ such
    that $d({ v},\theta({ v}))=\delta$ and
    $d({v},{ w})=\delta +1$ for ${ v}$ in $C$, ${ w}$ in $C'$ and
  ${ w}\ne\theta({v})$.
  (Here $\theta$ maps a vertex in $C$ to the unique vertex
in $C'$ with the same $i$th coordinate.)
\item If $i\ne j$, then there are unique ${ v}$ in $C$ and ${ w}$ in $C'$ with
  $d({ v},{ w})= \delta$;
  and distances between vertices in
  $C$ and $C'$ are $\delta$, $\delta+1$ and  $\delta+2$,
  with all values realised. (Here ${ v}$ and ${ w}$ are
the vertices which agree in both the $i$th and $j$th coordinates; if two
vertices agree in just one of these, their distance is $\delta+1$, otherwise it
is $\delta+2$.)
\end{itemize}
See also claims 3--4 on pages 273--274 of~\cite{ps:cartesian}.

It is a consequence of the above that the partition of the maximal cliques into \emph{types}, where
$C({a},i)$ has type $i$, is invariant under graph automorphisms; each type forms a
partition $Q_i$ of $\Omega$.

By Proposition~\ref{p:antiiso} and the discussion following it, the maximal non-trivial partitions in the
sublattice generated by $Q_1,\ldots,Q_n$ form a Cartesian decomposition
of~$\Omega$.

\item This is clear, since no arbitrary choices were made in either construction. 
\end{enumerate}
\end{proof}

We can describe this automorphism group precisely. Details will be given
in the case where all alphabets are the same; we deal briefly with the
mixed-alphabet case at the end.

Given a set $\Omega=A^n$, the wreath product $\Sym(A)\wr S_n$ acts on
$\Omega$: the $i$th factor of the base group $\Sym(A)^n$ acts on the entries
in the $i$th coordinate of points of $\Omega$, while $S_n$ permutes the
coordinates. (Here $S_n$ denotes $\Sym(\{1,\ldots,n\})$.)

\begin{cor}
The automorphism group of the Hamming graph $\Ham(n,A)$ is the wreath product
$\Sym(A)\wr S_n$ just described.
\end{cor}

\begin{proof}
By Theorem~\ref{th:cdham}(c), the automorphism group of $\Ham(n,A)$ coincides
with the stabiliser in $\Sym(A^n)$ of the natural Cartesian decomposition $\E$
of the set $A^n$. By~\cite[Lemma~5.1]{ps:cartesian},
the stabiliser of $\E$ in $\Sym(A^n)$ is $\Sym(A)\wr S_n$. 
\end{proof}

In the mixed alphabet case, only one change needs to be made. Permutations
of the coordinates must preserve the cardinality of the alphabets associated
with the coordinate: that is, $g\in S_n$ induces an automorphism of the
Hamming graph if and only if $ig=j$ implies $|A_i|=|A_j|$ for all $i,j$.
(This condition is clearly necessary. For sufficiency, if $|A_i|=|A_j|$,
then we may actually identify $A_i$ and $A_j$.)

So if $\{1,\ldots,n\}=I_1\cup\cdots\cup I_r$, where $I_k$ is the non-empty set
of those indices for which the corresponding alphabet has some given cardinality,
then the group $\Aut(\Ham(A_1,\ldots,A_n))$ is the direct product of $r$ groups, each
a wreath product $\Sym(A_{i_k})\wr\Sym(I_k)$, acting in its product action,
where $i_k$ is a member of $I_k$.

Part~(c) of Theorem~\ref{th:cdham} was also proved in~\cite[Theorem~12.3]{ps:cartesian}.
Our proof is a simplified version of the proof presented in~\cite{ps:cartesian}
and is included here as a nice application of the lattice theoretical framework
developed in Section~\ref{sec:prelim}. The automorphism group of the mixed-alphabet Hamming graph can also be determined 
using the characterisation of the automorphism groups of Cartesian products of graphs.
The first such characterisations were given by Sabidussi~\cite{Sabidussi} and
Vizing~\cite{Vizing}; see also~\cite[Theorem~6.6]{grhandbook}. 
The recent preprint~\cite{MZ} gives a self-contained  elementary proof in the case of 
finite Hamming graphs.

\section{Latin cubes}
\label{sec:LC}

\subsection{What is a Latin cube?}
\label{sec:whatis}
As pointed out in \cite{dap75oz,dap83enc,dap89jas,DAcube},
there have been many different definitions of
a Latin cube (that is, a three-dimensional generalisation of a Latin square) 
and of a Latin hypercube (a higher-dimensional generalisation).
Typically, the underlying set $\Omega$ is a Cartesian product
$\Omega_1 \times \Omega_2 \times\cdots \times \Omega_m$
where $\left|\Omega_1\right| = \left|\Omega_2\right| = \cdots =
\left|\Omega_m\right|$.  As for Latin squares in Section~\ref{sec:LS}, we often
seek to relabel the elements of $\Omega_1$, \ldots, $\Omega_m$ so that
$\Omega = T^m$ for some set~$T$.  The possible
conditions are concisely summarised in \cite{CRC}.  The alphabet is
a set of letters of cardinality $\left|T\right|^a$ with
$1\leqslant a\leqslant m-1$, and the \emph{type} is $b$ with
$1\leqslant b\leqslant m-a$. The definition is that if the values of any $b$
coordinates are fixed then all letters in the given alphabet occur
equally often on the subset of $\Omega$ so defined (which can be regarded
as a $(m-b)$-dimensional array, so that the $|T|^b$ arrays of this form
partition $T^m$; these are parallel lines or planes in a cubical array
according as $b=2$ or $b=1$).

One extreme case has $a=1$ and $b=m-1$.
This definition is certainly in current use
when $m \in \{3,4\}$: for example, see \cite{MWcube,MulWeb}.
The hypercubes in \cite{LMW}
have $a=1$ but allow smaller values of $b$.
The other extreme has $a=m-1$ and $b=1$,
which is what we have here.
(Unfortunately, the meaning of the phrase ``Latin hypercube design'' in
Statistics has completely changed in the last thirty years.  For example,
see \cite{tang2009,tang93}.)

Fortunately, it suffices for us to consider Latin cubes, where $m=3$.
Let $P_1$, $P_2$ and $P_3$ be the partitions which give the standard Cartesian
decomposition of the cube $\Omega_1 \times \Omega_2 \times \Omega_3$.
Following~\cite{DAcube}, we call the parts of
$P_1$, $P_2$ and $P_3$ \textit{layers}, and the parts of $P_1\wedge P_2$,
$P_1\wedge P_3$ and $P_2\wedge P_3$ \textit{lines}.  Thus a layer is a slice
of the cube parallel to one of the faces.
Two lines $\ell_1$ and
$\ell_2$ are said to be \textit{parallel} if there is some
$\{i,j\}\subset \{1,2,3\}$ with $i\ne j$ such that $\ell_1$ and $\ell_2$
are both parts of $P_i \wedge P_j$.

The definitions in \cite{CRC,DAcube} give us the following three possibilities
for the case that $|\Omega_i|=n$ for $i$ in $\{1,2,3\}$.
\begin{itemize}
\item[(LC0)]
  There are $n$ letters, each of which occurs once per line.
\item[(LC1)]
  There are $n$ letters, each of which occurs $n$ times per layer.
\item[(LC2)]
  There are $n^2$ letters, each of which occurs once per layer.
\end{itemize}
Because of the meaning of \textit{type} given in the first
paragraph of this section, we shall call
these possibilities \textit{sorts} of Latin cube.
Thus Latin cubes of sort (LC0) are a special case of Latin cubes of
sort (LC1), but Latin cubes of sort (LC2) are quite different.

Sort (LC0) is the definition of Latin cube used in
\cite{rab:as,ball,dscube,gupta,MWcube,MulWeb}, among many others in
Combinatorics and Statistics.
Fisher used sort (LC1) in \cite{RAF42}, where he gave constructions using
abelian groups.  Kishen called this a Latin cube
\textit{of first order}, and those of sort (LC2) Latin cubes \textit{of
  second order}, in \cite{kish42,kish50}.

Two of these sorts have alternative descriptions using the language of this
paper.  Let $L$ be the partition into letters. Then a Latin cube has sort
(LC0) if and only if $\{L,P_i,P_j\}$ is a Cartesian decomposition of the cube
whenever $i\ne j$ and $\{i,j\} \subset \{1,2,3\}$.
A Latin cube has sort (LC2) if and only if $\{L,P_i\}$
is a Cartesian decomposition of the cube for $i=1$, $2$, $3$.

The following definition is taken from \cite{DAcube}.
\begin{defn}
\label{def:reg}
  A Latin cube of sort (LC2) is \textit{regular} if, whenever $\ell_1$ and
  $\ell_2$ are parallel lines in the cube, the set of letters occurring in
  $\ell_1$ is either exactly the same as the set of letters occurring
  in $\ell_2$ or disjoint from it. 
  \end{defn}

(Warning: the word \textit{regular} is used by some authors with quite
a different meaning for some Latin cubes of sorts (LC0) and (LC1).)

\subsection{Some examples of Latin cubes of sort (LC2)}

In these examples, the cube is coordinatised by functions $f_1$, $f_2$ and
$f_3$ from $\Omega$ to $\Omega_1$, $\Omega_2$ and $\Omega_3$
whose kernels are the partitions $P_1$, $P_2$ and $P_3$.
For example, in Figure~\ref{fig:2}, one part of $P_1$ is $f_1^{-1}(2)$.
A statistician would typically write this as ``$f_1=2$''.
For ease of reading, we adopt the statisticians' notation.

\begin{eg}
  \label{eg:2}
  When $n=2$, the definition of Latin cube of sort (LC2) 
  forces the two occurrences of each of the four letters to be in
  diagonally opposite cells
  of the cube.  Thus, up to permutation of the letters, the only possibility
  is that shown in Figure~\ref{fig:2}.

  \begin{figure}
    \[
    \begin{array}{c@{\qquad}c}
      \begin{array}{c|c|c|}
        \multicolumn{1}{c}{} & \multicolumn{1}{c}{f_2=1} &
        \multicolumn{1}{c}{f_2=2}\\
     \cline{2-3}
     f_1=1 & A & B\\
     \cline{2-3}
     f_1=2 & C & D\\
     \cline{2-3}
      \end{array}
      &
      \begin{array}{c|c|c|}
     \multicolumn{1}{c}{} & \multicolumn{1}{c}{f_2=1} & \multicolumn{1}{c}{f_2=2}\\
     \cline{2-3}
     f_1=1 & D & C\\
     \cline{2-3}
     f_1=2 & B & A\\
     \cline{2-3}
      \end{array}
      \\[10\jot]
    \quad  f_3=1 &\quad f_3=2
      \end{array}
    \]   
    \caption{The unique (up to isomorphism)
      Latin cube of sort (LC2) and order~$2$}
\label{fig:2}
  \end{figure}

  This Latin cube of sort (LC2) is regular.
  The set of letters on each line of $P_1\wedge P_2$ is either $\{A,D\}$ or
  $\{B,C\}$; the set of letters on each line of $P_1\wedge P_3$ is either
  $\{A,B\}$ or $\{C,D\}$; and the set of letters on each line of $P_2\wedge P_3$
  is either $\{A,C\}$ or $\{B,D\}$.
\end{eg}

\begin{eg}
\label{eg:nice}
Here $\Omega=T^3$, where $T$~is the additive group of $\mathbb{Z}_3$.
For $i=1$, $2$ and~$3$, the function $f_i$ picks out the $i$th coordinate
of $(t_1,t_2,t_3)$.  The column headed~$L$ in Table~\ref{tab:cube2}
shows how the nine letters are allocated to the cells of the cube.
The $P_3$-layer of the cube with $f_3=0$ is as follows.
\[
\begin{array}{c|c|c|c|}
\multicolumn{1}{c}{} & \multicolumn{1}{c}{f_2=0} & \multicolumn{1}{c}{f_2=1}
& \multicolumn{1}{c}{f_2=2}\\
\cline{2-4}
f_1=0 & A & D & G\\
\cline{2-4}
f_1=1 & I & C & F\\
\cline{2-4}
f_1=2 & E & H & B\\
\cline{2-4}
\end{array}
\ .
\]
It has each letter just once.

Similarly, the $P_3$-layer of the cube with $f_3=1$ is 
\[
\begin{array}{c|c|c|c|}
\multicolumn{1}{c}{} & \multicolumn{1}{c}{f_2=0} & \multicolumn{1}{c}{f_2=1}
& \multicolumn{1}{c}{f_2=2}\\
\cline{2-4}
f_1=0 & B & E & H\\
\cline{2-4}
f_1=1 & G & A & D\\
\cline{2-4}
f_1=2 & F & I & C\\
\cline{2-4}
\end{array}
\]
and
the $P_3$-layer of the cube with $f_3=2$ is 
\[
\begin{array}{c|c|c|c|}
\multicolumn{1}{c}{} & \multicolumn{1}{c}{f_2=0} & \multicolumn{1}{c}{f_2=1}
& \multicolumn{1}{c}{f_2=2}\\
\cline{2-4}
f_1=0 & C & F & I\\
\cline{2-4}
f_1=1 & H & B & E\\
\cline{2-4}
f_1=2 & D & G & A\\
\cline{2-4}
\end{array}
\ .
\]
Similarly you can check that if you take the $2$-dimensional $P_1$-layer
defined by any fixed value of $f_1$ then
every letter occurs just once, and the same thing happens for~$P_2$.

\begin{table}[htbp]
\[
\begin{array}{cccccccc}
  \mbox{partition}& P_1 & P_2 & P_3 & Q & R & S & L\\
  \mbox{function}& f_1 & f_2 & f_3 & -f_1+f_2 &-f_3+f_1 & -f_2+f_3\\
  \mbox{value} &  t_1 & t_2 & t_3 & -t_1+t_2 & -t_3+t_1 & -t_2+t_3 & \\
\hline
& 0 & 0 & 0 & 0 & 0 & 0 & A \\
& 0 & 0 & 1 & 0 & 2 & 1 & B \\
& 0 & 0 & 2 & 0 & 1 & 2 & C \\
& 0 & 1 & 0 & 1 & 0 & 2 &D \\
& 0 & 1 & 1 & 1 & 2 & 0 & E \\
& 0 & 1 & 2 & 1 & 1 & 1 & F \\
& 0 & 2 & 0 & 2 & 0 & 1 & G \\
& 0 & 2 & 1 & 2 & 2 & 2 & H \\
& 0 & 2 & 2 & 2 & 1 & 0 & I \\
& 1 & 0 & 0 & 2 & 1 & 0 & I \\
& 1 & 0 & 1 & 2 & 0 & 1 & G \\
& 1 & 0 & 2 & 2 & 2 & 2 & H \\
& 1 & 1 & 0 & 0 & 1 & 2 & C \\
& 1 & 1 & 1 & 0 & 0 & 0 & A \\
& 1 & 1 & 2 & 0 & 2 & 1 & B \\
& 1 & 2 & 0 & 1 & 1 & 1 & F \\
& 1 & 2 & 1 & 1 & 0 & 2 & D \\
& 1 & 2 & 2 & 1 & 2 & 0 & E \\
& 2 & 0 & 0 & 1 & 2 & 0 & E \\
& 2 & 0 & 1 & 1 & 1 & 1 & F \\
& 2 & 0 & 2 & 1 & 0 & 2 & D \\
& 2 & 1 & 0 & 2 & 2 & 2 & H \\
& 2 & 1 & 1 & 2 & 1 & 0 & I \\
& 2 & 1 & 2 & 2 & 0 & 1 & G \\
& 2 & 2 & 0 & 0 & 2 & 1 & B \\
& 2 & 2 & 1 & 0 & 1 & 2 & C \\
& 2 & 2 & 2 & 0 & 0 & 0 & A \\
\end{array}
\]
\caption{Some functions and partitions on the cells of the cube
in Example~\ref{eg:nice}}
\label{tab:cube2}
\end{table} 

In addition to satisfying the property of being a Latin cube of sort (LC2),
this combinatorial structure has three other good properties.
\begin{itemize}
\item
  It is a regular in the sense of Definition~\ref{def:reg}.
  The set of letters in any
  $P_1\wedge P_2$-line is $\{A,B,C\}$ or $\{D,E,F\}$ or $\{G,H,I\}$.
  For $P_1\wedge P_3$ the letter sets are $\{A,D,G\}$, $\{B,E,H\}$ and
  $\{C,F,I\}$; for $P_2\wedge P_3$ they are $\{A,E,I\}$, $\{B,F,G\}$ and
  $\{C,D,H\}$.

\item
  The supremum of $L$ and   $P_1\wedge P_2$ is the partition $Q$ shown in
  Table~\ref{tab:cube2}.  This is the kernel of the function which maps
  $(t_1,t_2,t_3)$ to $-t_1+t_2 = 2t_1+t_2$.
  Statisticians normally write this partition
  as $P_1^2P_2$.  Likewise, the supremum of $L$ and $P_1\wedge P_3$ is $R$,
  which statisticians might write as $P_3^2P_1$,
  and the supremum of $L$ and   $P_2\wedge P_3$ is $S$, written by statisticians
  as $P_2^2P_3$.  The partitions $P_1$, $P_2$, $P_3$,
  $Q$, $R$, $S$, $P_1\wedge P_2$, $P_1\wedge P_3$, $P_2\wedge P_3$ and $L$
  are pairwise compatible, in the sense of Definition~\ref{def:compatible}.
  Moreover, each of them is a coset partition defined by a subgroup of $T^3$.
\item
  In anticipation of the notation used in Section~\ref{sec:dag},
  it seems fairly natural to rename $P_1$, $P_2$, $P_3$, $Q$, $R$ and $S$
  as $P_{01}$, $P_{02}$, $P_{03}$, $P_{12}$, $P_{13}$ and $P_{23}$, in order.
  For each $i$ in $\{0,1,2,3\}$, the three partitions $P_{jk}$ which have
  $i$ as one of the subscripts, that is, $i\in \{j,k\}$,
  form a Cartesian decomposition of the underlying  set.
\end{itemize}

However, the set of ten partitions that we have named is not closed under
infima, so they do not form an orthogonal block structure.
For example, the set does not contain the infimum $P_3\wedge Q$.
This partition has nine parts of size three, one of
which consists of the cells $(0,0,0)$, $(1,1,0)$ and $(2,2,0)$,
as can be seen from Table~\ref{tab:cube2}.

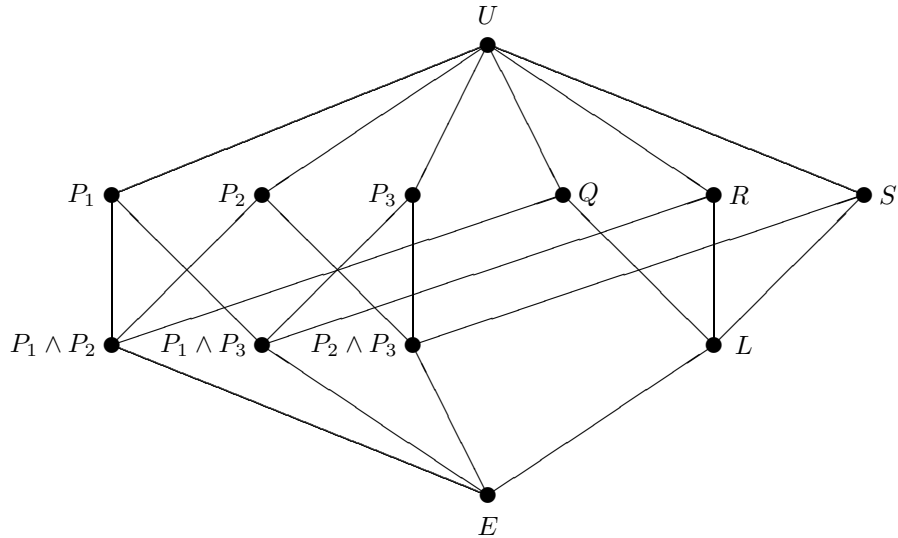
\begin{figure}
  \begin{center}
    \setlength{\unitlength}{2mm}
    \begin{picture}(60,40)
      \put(5,15){\line(0,1){10}}
      \put(5,15){\line(1,1){10}}
      \put(5,15){\line(3,1){30}}
      \put(15,15){\line(-1,1){10}}
      \put(15,15){\line(1,1){10}}
      \put(15,15){\line(3,1){30}}
      \put(25,15){\line(-1,1){10}}
      \put(25,15){\line(0,1){10}}
      \put(25,15){\line(3,1){30}}
      \put(45,15){\line(-1,1){10}}
      \put(45,15){\line(0,1){10}}
      \put(45,15){\line(1,1){10}}
      \put(30,5){\line(-1,2){5}}
      \put(30,5){\line(-3,2){15}}
      \put(30,5){\line(3,2){15}}
      \curve(30,5,5,15)
      \put(30,35){\line(-1,-2){5}}
      \put(30,35){\line(1,-2){5}}
      \put(30,35){\line(-3,-2){15}}
      \put(30,35){\line(3,-2){15}}
      \curve(30,35,5,25)
      \curve(30,35,55,25)
      \put(5,15){\blobb}
      \put(4,15){\makebox(0,0)[r]{$P_1\wedge P_2$}}
      \put(15,15){\blobb}
      \put(14,15){\makebox(0,0)[r]{$P_1\wedge P_3$}}
      \put(25,15){\blobb}
      \put(24,15){\makebox(0,0)[r]{$P_2\wedge P_3$}}
      \put(45,15){\blobb}
      \put(47,15){\makebox(0,0){$L$}}
      \put(30,5){\blobb}
      \put(30,3){\makebox(0,0){$E$}}
\put(5,25){\blobb}
\put(3,25){\makebox(0,0){$P_1$}}
\put(15,25){\blobb}
\put(13,25){\makebox(0,0){$P_2$}}
\put(25,25){\blobb}
\put(23,25){\makebox(0,0){$P_3$}}
\put(35,25){\blobb}
\put(36,25){\makebox(0,0)[l]{$Q$}}
\put(45,25){\blobb}
\put(46,25){\makebox(0,0)[l]{$R$}}
\put(55,25){\blobb}
\put(56,25){\makebox(0,0)[l]{$S$}}
\put(30,35){\blobb}
\put(30,37){\makebox(0,0){$U$}}
    \end{picture}
  \end{center}
  \caption{Hasse diagram of the join-semilattice formed by the pairwise
    compatible partitions in Example~\ref{eg:nice}}
    \label{fig:nice}
  \end{figure}

Figure~\ref{fig:nice} shows the Hasse diagram of the join-semilattice formed
by these ten named partitions, along with the two trivial partitions $E$
and $U$.
This diagram, along with the knowledge of compatibility, makes it clear that
any three of the minimal partitions $P_1 \wedge P_2$, $P_1 \wedge P_3$,
$P_2\wedge P_3$ and $L$ give the minimal
partitions of the orthogonal block structure defined by
a Cartesian decomposition of dimension three of the underlying set $T^3$.
Note that, although the partition $E$ is the highest point in the diagram
which is below both $P_3$ and $Q$, it is not their infimum, because their
infimum is defined in the lattice of all partitions of this set.
\end{eg}

\begin{figure}
  \[
  \begin{array}{c@{\qquad}c@{\qquad}c}
    \begin{array}{|c|c|c|}
      \hline
      A & E & F\\
      \hline
      H & I & D\\
      \hline
      C & G & B\\
      \hline
    \end{array}
    &
        \begin{array}{|c|c|c|}
      \hline
      D & B & I\\
      \hline
      E & C & G\\
      \hline
     F & A & H\\
      \hline
        \end{array}
        &
            \begin{array}{|c|c|c|}
      \hline
      G & H & C\\
      \hline
      B & F & A\\
      \hline
      I & D & E\\
      \hline
      \end{array}
    \end{array}
  \]
  \caption{A Latin cube of sort (LC2) which is not regular}
  \label{fig:sax}
  \end{figure}

\begin{eg}
  \label{eg:sax}
Figure~\ref{fig:sax} shows an example which is not regular.  This was originally
given in \cite{saxena}.  To save space, the three $P_3$-layers are shown
side by side.

For example, there is one $P_1\wedge P_3$-line whose set of letters is
$\{A,E,F\}$ and another whose set of letters is $\{A,F,H\}$.
These are neither the same nor disjoint.
\end{eg}

If we write the group operation in Example~\ref{eg:nice} multiplicatively,
then the cells
$(t_1,t_2,t_3)$ and $(u_1,u_2,u_3)$ have the same letter if and only if
$t_1^{-1}t_2 = u_1^{-1}u_2$ and $t_1^{-1}t_3 = u_1^{-1}u_3$.  This means that
$(u_1,u_2,u_3) = (x,x,x)(t_1,t_2,t_3)$ where $x=u_1t_1^{-1}$, so that
$(t_1,t_2,t_3)$ and $(u_1,u_2,u_3)$ are in the same right coset of the
diagonal subgroup $\delta(T,3)$ introduced in Section~\ref{sect:diaggroups}.

The next theorem shows that this construction can be generalised to any group,
abelian or not, finite or infinite.

\begin{theorem}
  \label{th:upfront}
  Let $T$ be a non-trivial group. Identify the elements of $T^3$ with the cells of a cube
  in the natural way.  Let $\delta(T,3)$ be the diagonal subgroup
  $\{(t,t,t) \mid t \in T\}$.  Then the parts of the right coset partition
  $P_{\delta(T,3)}$ form the letters of a regular Latin cube of sort 
  (LC2).
\end{theorem}

\begin{proof}
  Let $H_1$ be the subgroup $\{(1,t_2,t_3) \mid t_2 \in T, \ t_3 \in T\}$
  of $T^3$.  Define subgroups $H_2$ and $H_3$ similarly. Let $i \in \{1,2,3\}$.
Then $H_i \cap \delta(T,3) = \{1\}$ and $H_i\delta(T,3) = \delta(T,3)H_i = T^3$.
  Proposition~\ref{prop:coset} shows that $P_{H_i} \wedge P_{\delta(T,3)} = E$ and
  $P_{H_i} \vee P_{\delta(T,3)} = U$.  Because $H_i\delta(T,3) = \delta(T,3)H_i$,
  Proposition~\ref{prop:commeq} (considering statements (a) and~(d)) shows that $\{P_{H_i}, P_{\delta(T,3)}\}$ is a
  Cartesian decomposition of $T^3$ of dimension two.  Hence the parts
  of $P_{\delta(T,3)}$ form the letters of a Latin cube $\Lambda$ of sort~(LC2).

  Put $G_{12} = H_1 \cap H_2$ and
  $K_{12} = \{(t_1,t_1,t_3) \mid t_1 \in T,\ t_3 \in T\}$.
  Then the parts of $P_{G_{12}}$ are lines of the cube parallel to the $z$-axis.
  Also, $G_{12} \cap \delta(T,3)=\{1\}$ and $G_{12}\delta(T,3) = \delta(T,3)G_{12}
    = K_{12}$, so Propositions~\ref{prop:coset} and~\ref{prop:commeq} show that
    $P_{G_{12}} \wedge P_{\delta(T,3)} = E$, $P_{G_{12}} \vee P_{\delta(T,3)} = P_{K_{12}}$,
    and the restrictions of $P_{G_{12}}$ and $P_{\delta(T,3)}$ to any part
    of $P_{K_{12}}$ form a grid. Therefore,  within each coset of~$K_{12}$,
    all lines have the same subset of letters.  By the definition of supremum,
    no line in any other coset of $K_{12}$ has any letters in common
    with these.

    Similar arguments apply to lines in each of the other two directions.
    Hence $\Lambda$ is regular. 
\end{proof}

The converse of this theorem is proved at the end of this section.

The set of partitions in Theorem~\ref{th:upfront} form a join-semilattice whose
Hasse diagram is the same as the one shown in Figure~\ref{fig:nice}, apart from
the naming of the partitions.  We call this a \textit{diagonal semilattice
  of dimension three}.  The generalisation to arbitrary dimensions is given
in Section~\ref{sec:diag}.

\subsection{Results for Latin cubes}

As we hinted in Section~\ref{sec:LS},
the vast majority of Latin squares of order at least $5$
are not  isotopic to Cayley tables of groups.  For $m\geqslant 3$, the situation
changes dramatically as soon as we impose some more, purely combinatorial,
constraints.  We continue to use the notation $\Omega$, $P_1$, $P_2$, $P_3$
and $L$ as in Section~\ref{sec:whatis}.

A Latin cube of sort (LC0) is called an \textit{extended Cayley table} of
the group~$T$ if $\Omega=T^3$ and the letter in cell $(t_1,t_2,t_3)$ is
$t_1t_2t_3$. Theorem~8.21 of \cite{rab:as} shows that, in the finite case,
for a Latin cube of sort (LC0), the set $\{P_1,P_2,P_3,L\}$ is contained in
the set of partitions of an orthogonal block structure if and only if the
cube is isomorphic to the extended Cayley table of an abelian group.
Now we will prove something similar for Latin cubes of sort (LC2), by
specifying a property of the set
\[\{ P_1, P_2, P_3, (P_1\wedge P_2)\vee L,
(P_1\wedge P_3)\vee L, (P_2\wedge P_3)\vee L\}\]
of six partitions.  We do not restrict this
to finite sets.  Also, because we do not insist on closure under infima,
it turns out that the group does not need to be abelian.

 In Lemmas~\ref{lem:lc0} and~\ref{lem:lc3},
 the assumption is that we have a Latin cube of sort~(LC2),
 and that $\{i,j,k\} = \{1,2,3\}$.  Write 
\[
  L^{ij}= L\vee(P_i\wedge P_j).
\]
To clarify the proofs, we shall use the following refinement of
Definition~\ref{def:reg}. Recall that we refer to the parts of $P_i\wedge P_j$  
as $P_i\wedge P_j$-lines.

\begin{defn}
  \label{def:refine}
  A Latin cube of sort (LC2) is \textit{$\{i,j\}$-regular} if,
  whenever $\ell_1$ and $\ell_2$ are distinct $P_i\wedge P_j$-lines,
  the set of letters occurring in
  $\ell_1$ is either exactly the same as the set of letters occurring
  in $\ell_2$ or disjoint from it. 
\end{defn}

\begin{lem}
  \label{lem:lc0}
  The following  conditions are equivalent.
        \begin{enumerate}
        \item
          The partition $L$ is  compatible with $P_i\wedge P_j$.
        \item
          The Latin cube is  $\{i,j\}$-regular.  
        \item  The restrictions of $P_i\wedge P_j$, $P_k$ and $L$ to any 
        part of $L^{ij}$ form a Latin square.
        \item Every pair of  distinct $P_i\wedge P_j$-lines in the same 
        part of $L^{ij}$ lie in distinct parts of $P_i$.
        \item  The restrictions of $P_i$, $P_k$ and $L$ to any 
        part of $L^{ij}$ form a Latin square.
        \item The set
        $\{P_i,P_k,L^{ij}\}$ is a Cartesian decomposition of $\Omega$ of
        dimension three. 
        \item Each part of  $P_i\wedge P_k\wedge L^{ij}$ has size  one. 
        \end{enumerate}
\end{lem}

\begin{proof} We prove this result without loss of generality for
  $i=1$, $j=2$, $k=3$.

  \begin{itemize}
          \item[(a)$\Leftrightarrow$(b)]
    By the definition of a Latin cube of sort (LC2),
    each part of $P_1\wedge P_2$  has either zero or one cells in common
    with each part of~$L$.  Therefore  ${P_1\wedge P_2 \wedge L}=E$,
    which is uniform, so Definition~\ref{def:compatible} shows that
    compatibility is the same as commutativity of the  equivalence relations
    underlying $P_1\wedge P_2$ and~$L$.
    Consider Proposition~\ref{prop:commeq} with $P_1\wedge P_2$ and $L$ in place
    of $P_1$ and $P_2$.  Condition~(a) of Proposition~\ref{prop:commeq}
    is the same as condition~(a) here; and condition~(e) of
    Proposition~\ref{prop:commeq} is the same as condition~(b) here. Thus
    Proposition~\ref{prop:commeq} gives us the result.

  \item[(a)$\Rightarrow$(c)]
   Let $\Delta$ be a part of $L^{12}$.  If $L$ is compatible with $P_1\wedge P_2$
   then, because ${P_1\wedge P_2 \wedge L}=E$,
   Proposition~\ref{prop:commeq} shows that
   the restrictions of $P_1\wedge P_2$ and $L$ to $\Delta$ form a Cartesian
  decomposition of $\Delta$.  Each part of $P_3$ has precisely one cell in
  common with each part of $P_1\wedge P_2$,
  because $\{P_1,P_2,P_3\}$ is a Cartesian decomposition of $\Omega$,
  and precisely one cell in common with each part of $L$,
  because the Latin cube has sort (LC2).
  Hence the restrictions of $P_1\wedge P_2$, $P_3$ and $L$ to $\Delta$
  form a Latin square. (Note that $P_3$ takes all of its values within $\Delta$,
  but neither $P_1\wedge P_2$ nor $L$ does.)
  
\item[(c)$\Rightarrow$(d)]
  Let $\ell_1$ and $\ell_2$ be distinct $P_1\wedge P_2$-lines
  that are contained in the same part $\Delta$ of $L^{12}$. Every letter
  which occurs in $\Delta$ occurs in both of these lines. If $\ell_1$ and
  $\ell_2$ are contained in the same part of $P_1$, then that $P_1$-layer
  contains at least two occurrences of some letters, which contradicts the
  fact that $L\wedge P_1=E$ for a Latin cube of sort (LC2).

\item[(d)$\Rightarrow$(e)]
  Let $\Delta$ be a part of $L^{12}$ and let $\lambda$ be a part of~$L$
  inside~$\Delta$.   Let $p_1$ and $p_3$  be parts of $P_1$ and $P_3$.
  Then $\left| p_1 \cap \lambda \right| = \left| p_3 \cap \lambda \right|=1$
  by definition of a Latin cube of sort (LC2).  Condition (d) specifies that
  $p_1 \cap \Delta$ is a part of $P_1 \wedge P_2$.  Therefore
  $(p_1 \cap \Delta) \cap p_3$ is a part of ${P_1 \wedge P_2 \wedge P_3}$, so
  $ \left |(p_1 \cap \Delta) \cap (p_3 \cap \Delta)\right |=
  \left |(p_1 \cap \Delta) \cap p_3\right| =1$.
  Thus the restrictions of $P_1$, $P_3$, and $L$ to $\Delta$ form a Latin
  square.

\item[(e)$\Rightarrow$(f)]
  Let $\Delta$, $p_1$ and $p_3$ be parts of $L^{12}$, $P_1$ and $P_3$
  respectively.  By the definition of a Latin cube of sort (LC2),
  $p_1 \cap \Delta$ and $p_3 \cap \Delta$ are both non-empty.  Thus
  condition (e) implies that $\left | p_1 \cap p_3 \cap \Delta \right|=1$.
  Hence $\{P_1, P_3, L^{12}\}$ is a Cartesian
  decomposition of dimension three.

\item[(f)$\Rightarrow$(g)] This follows immediately 
  from the definition of a Cartesian decomposition (Definition~\ref{def:cart}).

\item[(g)$\Rightarrow$(d)]
  If (d) is false then there is a part~$\Delta$  of $L^{12}$ which
  contains distinct
  $P_1\wedge P_2$-lines $\ell_1$ and $\ell_2$ in the same part~$p_1$ of~$P_1$.
  Let $p_3$ be any part of $P_3$. Then, since $\{P_1,P_2,P_3\}$ is a 
  Cartesian decomposition, $\left |p_3\cap \ell_1\right | =
  \left | p_3\cap \ell_2\right | =1$ and so
$\left| p_1\cap p_3 \cap \Delta \right | \geqslant 2$.  This contradicts~(g).

\item[(d)$\Rightarrow$(b)]
  If (b) is false, there are  distinct $P_1\wedge P_2$-lines $\ell_1$
  and $\ell_2$ 
  whose sets of letters $\Lambda_1$ and $\Lambda_2$ are neither the same nor
  disjoint.  Because $\Lambda_1 \cap \Lambda_2 \ne \emptyset$, $\ell_1$
  and $\ell_2$ are contained in the same part of $L^{12}$.

    Let $\lambda \in \Lambda_2 \setminus \Lambda_1$.  By definition of a Latin
    cube of sort (LC2),
    $\lambda$ occurs on precisely one cell~$\omega$
  in the $P_1$-layer which contains $\ell_1$.  By assumption, $\omega \notin
  \ell_1$.  Let $\ell_3$ be the $P_1\wedge P_2$-line containing~$\omega$.
  Then $\ell_3$ and $\ell_2$ are in the same part of $L^{12}$, as are
  $\ell_1$ and $\ell_2$.  Hence $\ell_1$ and $\ell_3$ are in the
  same part of $L^{12}$ and the same part of $P_1$.  This contradicts~(d). 
  \end{itemize}
  \end{proof}

\begin{lem}
  \label{lem:lc3}
  The set $\{P_i,L^{ik},L^{ij}\}$ is a Cartesian decomposition of $\Omega$ if
  and only if $L$ is compatible with both $P_i\wedge P_j$ and $P_i \wedge P_k$.
\end{lem}

\begin{proof}
  If $L$ is not compatible with $P_i\wedge P_j$, then
  Lemma~\ref{lem:lc0} shows that there is a part of
  ${P_i \wedge P_k \wedge L^{ij}}$ of size at least two.
  This is contained in a part of $P_i\wedge P_k$. Since $P_i \wedge P_k
\preccurlyeq L^{ik}$, it is also contained in a part of~$L^{ik}$.  Hence
  $\{P_i, L^{ij}, L^{ik}\}$ is not a Cartesian decomposition of~$\Omega$.
  Similarly, if $L$ is not compatible with $P_i\wedge P_k$ then
    $\{P_i, L^{ij}, L^{ik}\}$ is not a Cartesian decomposition of~$\Omega$.

  For the converse,  Lemma~\ref{lem:lc0} shows
  that if $L$ is compatible with
  $P_i\wedge P_j$ then $\{P_i, P_k, L^{ij}\}$ is a Cartesian decomposition of
  $\Omega$.  Let $\Delta$ be a part of $L^{ij}$, and let $L^*$ be the
  restriction of $L$ to $\Delta$.  Lemma~\ref{lem:lc0} shows that
  $P_i$, $P_k$ and $L^*$ form a Latin square on~$\Delta$.  Thus distinct
  letters in~$L^*$ occur only in distinct parts of $P_i \wedge P_k$.

  If $L$ is also compatible with $P_i\wedge P_k$, then Lemma~\ref{lem:lc0}
  shows that  each part of $L^{ik}$ is a union of parts of $P_i\wedge P_k$,
  any two of which are in different parts of $P_i$ and different parts of~$P_k$,
  and all of which have the same letters.
  Hence any two different letters in $L^*$
  are in different parts of~$L^{ik}$.  Since $\{P_i,P_k,L^{ij}\}$ is a Cartesian
  decomposition of~$\Omega$,
  every part of $P_i\wedge P_k$ has a non-empty intersection with~$\Delta$, and
so   every part of $L^{ik}$ has a non-empty intersection with~$\Delta$.
Since $L\prec L^{ik}$, such an intersection consists of one or more
parts of $L^*$ in $\Delta$. We have already noted that distinct
letters in $L^*$ are in different parts of $L^{ik}$, and so it follows that the
restriction of $L^{ik}$ to $\Delta$ is the same as~$L^*$.
  Hence the restrictions of $P_i$, $P_k$ and $L^{ik}$ to $\Delta$ form a Latin
  square on $\Delta$, and so the restrictions of $P_i$ and $L^{ik}$ to $\Delta$
  give a Cartesian decomposition of~$\Delta$.

  This is true for every part $\Delta$ of $L^{ij}$, and so it follows that
  $\{P_i, L^{ij}, L^{ik}\}$ is a Cartesian decomposition of~$\Omega$. 
\end{proof}

\begin{lem}
  \label{lem:lc4}
  The set $\{P_i, L^{ij},L^{ik}\}$ is a Cartesian decomposition of $\Omega$
  if and only if the set $\{P_i \wedge P_j, P_i\wedge P_k,L\}$
 generates a Cartesian lattice under taking suprema.
\end{lem}

\begin{proof}
  If $\{P_i \wedge P_j, P_i\wedge P_k,L\}$ generates a Cartesian lattice under
  taking suprema then the maximal partitions in the Cartesian lattice are
  $(P_i \wedge P_j) \vee (P_i\wedge P_k)$,  $(P_i \wedge P_j) \vee L$ and
  $(P_i \wedge P_k) \vee L$. They form a Cartesian decomposition, and  
  are equal to $P_i$, $L^{ij}$ and $L^{ik}$
  respectively.

  Conversely, suppose that $\{P_i, L^{ij},L^{ik}\}$ is a Cartesian decomposition
  of~$\Omega$.  The minimal partitions in the corresponding Cartesian lattice
  are $P_i \wedge L^{ij}$, $P_i \wedge L^{ik}$ and $L^{ij} \wedge L^{ik}$.  Now,
  $L \preccurlyeq L^{ij}$ and $L \preccurlyeq L^{ik}$, so
  $L \preccurlyeq L^{ij} \wedge L^{ik}$.
  Because the Latin cube has sort~(LC2), $\{P_i,L\}$ and 
  $\{P_i,L^{ij}\wedge L^{ik}\}$ are both Cartesian decompositions of~$\Omega$.
  Since
  $L \preccurlyeq L^{ij} \wedge L^{ik}$, this forces $L=L^{ij}\wedge L^{ik}$. 

The identities of the other two infima are confirmed by a similar argument.
We have $P_i \wedge P_j \preccurlyeq P_i$, and
$P_i \wedge P_j \preccurlyeq L^{ij}$, by definition of~$L^{ij}$.  Therefore
  $P_i \wedge P_j \preccurlyeq P_i \wedge L^{ij}$.
Lemmas~\ref{lem:lc0} and~\ref{lem:lc3} show that $\{P_i,P_k,L^{ij}\}$ is a
Cartesian decomposition of~$\Omega$.  Therefore $\{P_k,P_i \wedge L^{ij}\}$
and $\{P_k, P_i \wedge P_j\}$ are both Cartesian decompositions of~$\Omega$.
Since $P_i \wedge P_j \preccurlyeq P_i \wedge L^{ij}$, this forces
$P_i \wedge P_j = P_i \wedge L^{ij}$.
Likewise, $P_i \wedge P_k = P_i \wedge L^{ik}$. 
\end{proof}

The following theorem is a direct consequence of 
Definitions~\ref{def:reg} and~\ref{def:refine} and
Lemmas~\ref{lem:lc0}, \ref{lem:lc3} and~\ref{lem:lc4}.

\begin{theorem}
\label{thm:regnice}
For a  Latin cube of sort~(LC2), the following conditions are equivalent.
\begin{enumerate}
\item
  The Latin cube is regular.
\item
  The Latin cube is $\{1,2\}$-regular, $\{1,3\}$-regular and $\{2,3\}$-regular.
\item
  The partition $L$ is compatible with each of $P_1\wedge P_2$, $P_1\wedge P_3$
  and $P_2\wedge P_3$.
\item Each of $\{P_1,P_2,P_3\}$,
  $\{P_1,L^{12},L^{13}\}$, $\{P_2,L^{12},L^{23}\}$ and $\{P_3, L^{13}, L^{23}\}$
  is a Cartesian decomposition.
\item
  Each of the sets ${\{P_1\wedge P_2, P_1 \wedge P_3, P_2\wedge P_3\}}$,
  ${\{P_1\wedge P_2, P_1 \wedge P_3, L\}}$, \linebreak
  ${\{P_1\wedge P_2,  P_2\wedge P_3, L\}}$ and
  ${\{P_1 \wedge P_3, P_2\wedge P_3, L\}}$
  generates a Cartesian lattice under taking suprema. 
\end{enumerate}
  \end{theorem}

  The condition that $\{P_1,P_2,P_3\}$ is a Cartesian decomposition 
  is a part of the definition of a Latin cube. This condition is 
  explicitly included in item~(d) of Theorem~\ref{thm:regnice} for clarity.

The final result in this section gives us the stepping stone for the proof of
Theorem~\ref{thm:main}.
The proof is quite detailed, and makes frequent use of the
relabelling techniques that we already saw in Sections~\ref{sec:LS}
and~\ref{sesc:quasi}.

\begin{theorem}
  \label{thm:bingo}
Consider a Latin cube of sort~(LC2) on an underlying set~$\Omega$,
  with coordinate partitions $P_1$, $P_2$ and $P_3$, and letter partition~$L$.
  If every three of $P_1 \wedge P_2$, $P_1 \wedge P_3$, $P_2\wedge P_3$ and $L$
  are the minimal  partitions in a Cartesian lattice on~$\Omega$
  then there is a group~$T$ such that, up to relabelling the letters
  and the three sets of coordinates,
  $\Omega=T^3$ and  $L$ is the coset partition defined
  by the diagonal subgroup $\{(t,t,t) \mid t \in T\}$.
Moreover, $T$~is unique up to group isomorphism.
\end{theorem}

\begin{proof}
  Theorem~\ref{thm:regnice} shows that a Latin cube satisfying this condition
  must be regular.
  As $\{P_1,P_2,P_3\}$ is a Cartesian decomposition of $\Omega$ and,
  by Lemma~\ref{lem:lc0}, $\{P_i,P_j,L^{ik}\}$ is also a Cartesian 
  decomposition of~$\Omega$ whenever $\{i,j,k\} = \{1,2,3\}$,
  the cardinalities of $P_1$, $P_2$, $P_3$, $L^{12}$, $L^{13}$ and $L^{23}$
  must all be equal
  (using the argument in the proof of Proposition~\ref{p:order}). 
  Thus we may label the parts of each by the same set~$T$.
  We start by labelling the parts of $P_1$, $P_2$ and $P_3$.  This identifies
  $\Omega$ with $T^3$.  At first, these three labellings are arbitrary, but
  they are made more specific as the proof progresses.
   
    Let $(a,b,c)$ be a cell of the cube.  Because
    $P_1\wedge P_2 \preccurlyeq L^{12}$, the part of $L^{12}$ which contains
    cell $(a,b,c)$ does not depend on the value of~$c$.  Thus
    there is a binary operation $\circ$ from $T \times T$ to $T$ such that
    $a \circ b$ is the label of the part of $L^{12}$ containing
    $\{(a,b,c)\mid c \in T\}$; in other words,  $(a,b,c)$ is in part
    $a \circ b$ of $L^{12}$, irrespective of the value of $c$.
       Lemma~\ref{lem:lc0} and Proposition~\ref{p:order} show that,
    for each $a$ in $T$, the function $b \mapsto a \circ b$ is a bijection
    from $T$ to~$T$.  Similarly, for each $b$ in~$T$, the function
    $a \mapsto a \circ b$ is a bijection.
    Therefore $(T,\circ)$ is a quasigroup.

    Similarly, there are binary operations $\star$ and $\diamond$ on $T$
    such that the labels of the parts of $L^{13}$ and $L^{23}$ containing
    cell $(a,b,c)$ are $c \star a$ and $b \diamond c$ respectively.
    Moreover, $(T,\star)$ and $(T,\diamond)$ are both quasigroups.

    Now we start the process of making explicit bijections between some pairs
    of the six partitions.
    Choose any part of $P_1$ and label it $e$.  Then the labels of the parts
    of $L^{12}$  can be aligned with those of $P_2$ so that $e \circ b= b$ for
    all values of~$b$.
    In the quasigroup $(T, \star)$, we may use the column headed $e$ to give
    a permutation $\sigma$ of $T$ to align the labels of the parts of~$P_3$
    and those of~$L_{13}$ so that $c\star e = c\sigma$ for all values of~$c$.

  Let $(a,b,c)$ be a cell of the cube.  Because $\{L,P_1\}$ is a Cartesian
  decomposition of the cube, there is a unique cell $(e,b',c')$
  in the same part of  $L$ as $(a,b,c)$.  Then
  \begin{eqnarray*}
    a \circ b & = & e \circ b' = b',\\
    c \star a & = & c' \star e = c'\sigma, \quad \mbox{and}\\
    b \diamond c& =& b'\diamond c'.
  \end{eqnarray*}
  Hence
  \begin{equation}
    b\diamond c = (a\circ b) \diamond ((c \star a)\sigma^{-1})
    \label{eq:threeops}
  \end{equation}
  for all values of $a$, $b$ and $c$ in~$T$.

  The quasigroup $(T,\diamond)$ can be viewed as a Latin square with rows
  labelled by parts of $P_2$ and columns labelled by parts of $P_3$.
  Consider the $2 \times 2$ subsquare shown in Figure~\ref{fig:subsq}.  It has
  $b_1 \diamond c_1 = \lambda$, $b_1 \diamond c_2 = \mu$,
  $b_2 \diamond c_1 = \nu$ and $b_2 \diamond c_2 = \phi$.

  \begin{figure}
    \[
    \begin{array}{c|c|c|}
      \multicolumn{1}{c}{} & \multicolumn{1}{c}{c_1} & \multicolumn{1}{c}{c_2}\\
      \cline{2-3}
      b_1 & \lambda & \mu\\
      \cline{2-3}
      b_2 & \nu & \phi\\
      \cline{2-3}
      \end{array}
    \]
    \caption{A $2 \times 2$ subsquare of the Latin square defined by
      $(T,\diamond)$}
    \label{fig:subsq}
  \end{figure}

  Let $b_3$ be any row of this Latin square.
  Then there is a unique $a$ in $T$ such
  that $a \circ b_1=b_3$.  By Equation~(\ref{eq:threeops}),
  \begin{eqnarray*}
    b_3 \diamond ((c_1 \star a)\sigma^{-1}) & = &
    (a \circ b_1) \diamond ((c_1 \star a)\sigma^{-1})  = b_1 \diamond c_1
    = \lambda, \quad \mbox{and}\\
     b_3 \diamond ((c_2 \star a)\sigma^{-1}) & = &
    (a \circ b_1) \diamond ((c_2 \star a)\sigma^{-1})  = b_1 \diamond c_2
    = \mu.
  \end{eqnarray*}
  The unique occurrence of letter $\nu$ in column $(c_1\star a)\sigma^{-1}$ of
this Latin square  is in row~$b_4$, where $b_4= a \circ b_2$, because
  \[
      b_4 \diamond ((c_1 \star a)\sigma^{-1})  = 
    (a \circ b_2) \diamond ((c_1 \star a)\sigma^{-1})  = b_2 \diamond c_1
    = \nu.
    \]
    Now
    \[
      b_4 \diamond ((c_2 \star a)\sigma^{-1})  = 
    (a \circ b_2) \diamond ((c_2 \star a)\sigma^{-1})  = b_2 \diamond c_2
    = \phi.
    \]

    This shows that whenever the letters in three cells of a $2 \times 2$
    subsquare are known then the letter in the remaining cell is forced.
    That is, the Latin square $(T,\diamond)$ 
    satisfies the quadrangle criterion (Definition~\ref{def:quad}).
    By Theorem~\ref{thm:frolov}, this property proves that $(T,\diamond)$ is
    isotopic to the Cayley table of a group. By \cite[Theorem~2]{albert},
    this group is unique up to group isomorphism.

    As remarked at the end of Section~\ref{sesc:quasi}, we can now relabel the
    parts of $P_2$, $P_3$ and $L^{23}$ so that $b \diamond c = b^{-1}c$ for
    all $b$, $c$ in $T$.  Then Equation~(\ref{eq:threeops}) becomes
    $b^{-1} c = (a\circ b)^{-1} ((c \star a)\sigma^{-1})$, so that
    \begin{equation}
      (a\circ b) b^{-1} c = (c \star a)\sigma^{-1}
      \label{eq:plod}
      \end{equation}
    for all $a$, $b$, $c$ in $T$.
    Putting $b=c$ in Equation~(\ref{eq:plod}) gives
    \begin{equation}
      (a \circ c)\sigma = c \star a
      \label{eq:plodonon}
    \end{equation}
    for all $a$, $c$ in $T$, while putting $b=1$ gives
    \[
    ((a \circ 1) c)\sigma = c\star a
    \]
        for all $a$,  $c$ in $T$.
        Combining these gives
        \begin{equation}
          \label{eq:plodon}
          a \circ c = (a \circ 1)c = (c\star a)\sigma^{-1}
          \end{equation}
        for all $a,c\in T$.

        We have not yet made any explicit use of the labelling of the parts
        of $P_1$ other than $e$, with $e \circ 1=1$.
        The map $a \mapsto a \circ 1$ is a bijection
        from $T$ to $T$, so we may label the parts of $P_1$ in such a way
        that $e=1$ and $a \circ 1 = a^{-1}$ for all $a$ in $T$.
        Then Equation~(\ref{eq:plodon}) shows that $a \circ b = a^{-1}b$
        for all $a$, $b$ in $T$.

Now that we have fixed the labelling of the parts of $P_1$, $P_2$ and $P_3$,
it is clear that they are the partitions of $T^3$
into right cosets of the subgroups as shown in the first three rows of
Table~\ref{tab:coset}.

Consider the partition $L^{23}$.  For $\alpha =(a_1,b_1,c_1)$ and
$ \beta =(a_2,b_2,c_2)$ in~$T^3$, we have (using the notation in Section~\ref{sec:part})
\begin{eqnarray*}
L^{23}[\alpha] = L^{23}[\beta]
  & \iff & b_1 \diamond c_1 = b_2 \diamond c_2\\
& \iff & b_1^{-1}c_1 = b_2^{-1}c_2\\
& \iff & \mbox{$\alpha$ and $\beta$ are in the same right coset of $K_{23}$,}
\end{eqnarray*}
where $K_{23} = \{(t_1,t_2,t_2) \mid t_1 \in T,\ t_2 \in T\}$.  In other words,
$L^{23}$ is the coset partition of $T^3$ defined by $K_{23}$.

Since $a \circ b = a^{-1}b$, a similar argument shows that $L^{12}$ is the
coset partition of $T^3$ defined by $K_{12}$, where
$K_{12} = \{(t_1,t_1,t_2) \mid t_1 \in T,\ t_2 \in T\}$.

Equation~(\ref{eq:plodonon}) shows that the kernel of the function
$(c,a) \mapsto c \star a$ is the same as the kernel of the function
$(c,a) \mapsto a^{-1}c$, which is in turn the same as the kernel of the function
$(c,a) \mapsto c^{-1}a$.  It follows that $L^{13}$ is the
coset partition of $T^3$ defined by $K_{13}$, where
$K_{13} = \{(t_1,t_2,t_1) \mid t_1 \in T,\ t_2 \in T\}$.  

Thus the partitions $P_i$ and $L^{ij}$ are the partitions of $T^3$
into right cosets of the subgroups as shown in Table~\ref{tab:coset}.
Lemma~\ref{lem:lc4} shows that the letter partition~$L$ is equal to
$L^{ij} \wedge L^{ik}$ whenever $\{i,j,k\} = \{1,2,3\}$.
Consequently, $L$ is the partition into right cosets of the diagonal
subgroup $\{(t,t,t) \mid t \in T\}$. 
\end{proof}

\begin{table}[htbp]
  \[
  \begin{array}{crcl}
    \mbox{Partition} & \multicolumn{3}{c}{\mbox{Subgroup of $T^3$}}\\
      \hline
      P_1 & & & \{(1,t_2,t_3)\mid t_2 \in T, \ t_3 \in T\}\\
      P_2 & & & \{(t_1,1,t_3) \mid t_1 \in T, \ t_3 \in T\}\\
      P_3 & & & \{(t_1,t_2,1)\mid t_1 \in T, \ t_2 \in T\}\\
      L^{12} & K_{12} & = & \{(t_1,t_1,t_3) \mid t_1 \in T, \ t_3 \in T\}\\
      L^{13} & K_{13} & = & \{(t_1,t_2,t_1) \mid t_1 \in T, \ t_2 \in T\}\\
       L^{23} & K_{23} & = & \{(t_1,t_2,t_2) \mid t_1 \in T, \ t_2 \in T\}\\
\hline
P_1\wedge P_2 & & & \{(1,1,t):t\in T\}\\
P_1\wedge P_3 & & & \{(1,t,1):t\in T\}\\
P_2\wedge P_3 & & & \{(t,1,1):t\in T\}\\
L & \delta(T,3) & = & \{(t,t,t):t\in T\}
    \end{array}
  \]
  \caption{Coset partitions at the end of the proof of Theorem~\ref{thm:bingo} 
and some infima}
  \label{tab:coset}
  \end{table}

The converse of Theorem~\ref{thm:bingo} was given in Theorem~\ref{th:upfront}.

For $\{i,j,k\}= \{1,2,3\}$, let $H_i$ be the intersection of the subgroups of
$T^3$ corresponding to partitions $P_i$ and $L^{jk}$ in Table~\ref{tab:coset},
so that the parts of $P_i \wedge L^{jk}$ are the right cosets of $H_i$.
Then $H_1 = \{(1,t,t)\mid t \in T\}$ and $H_2 = \{(u,1,u)\mid u \in T\}$. If
$T$ is abelian then $H_1H_2=H_2H_1$ and so the right-coset partitions
of $H_1$ and $H_2$ are compatible.  If $T$ is not abelian then $H_1H_2 \ne
H_2H_1$ and so these coset partitions are not compatible.  Because we do not
want to restrict our theory to abelian groups, we do not require our collection
of partitions to be closed under infima.  Thus we require a join-semilattice
rather than a lattice.

\subsection{Automorphism groups}

\begin{theorem}
Suppose that a regular Latin cube $M$ of sort (LC2) arises from a group $T$
by the construction of Theorem~\ref{th:upfront}. Then the group of
automorphisms of $M$ is equal to the diagonal group $D(T,3)$.
\label{t:autDT3}
\end{theorem}

\begin{proof}[Proof (sketch)]
It is clear from the proof of Theorem~\ref{th:upfront} that $D(T,3)$ is
a subgroup of $\Aut(M)$, and we have to prove equality.

Just as in the proof of Theorem~\ref{t:autDT2}, if $G$~denotes the
automorphism group of~$M$, then it suffices to prove that the group of strong
automorphisms of~$M$ fixing the cell $(1,1,1)$ is equal to $\Aut(T)$.

In the proof of Theorem~\ref{thm:bingo}, we choose a part of the partition
$P_1$ which will play the role of the identity of $T$, and using the partitions
we find bijections between the parts of the maximal partitions and show that
each naturally carries the structure of the group $T$. It is clear that
any automorphism of the Latin cube which fixes $(1,1,1)$ will preserve these
bijections, and hence will be an automorphism of $T$. So we have equality. 
\end{proof}

\begin{remark}
We will give an alternative proof of this theorem in the next section, in
Theorem~\ref{t:autDTm}.
\end{remark}

\section{Diagonal groups and diagonal semilattices}
\label{sec:diag}
\subsection{Diagonal semilattices}\label{sec:diag1}

Let $T$ be a group, and $m$ be an integer with $m\geqslant2$. Take $\Omega$ to
be the group~$T^m$. Following our convention in Section~\ref{sect:diaggroups},
we will now denote elements of $\Omega$ by $m$-tuples in square brackets.

Consider the following subgroups of $\Omega$:
\begin{itemize}
\item for $1\leqslant i\leqslant m$, $T_i$ is the $i$th coordinate subgroup, the set
of $m$-tuples with $j$th entry $1$ for $j\ne i$;
\item $T_0$ is the diagonal subgroup $\delta(T,m)$ of $T^m$, the set
$\{[t,t,\ldots,t] \mid t\in T\}$.
\end{itemize}
Let $Q_i$ be the partition of $\Omega$ into right cosets of $T_i$ for
$i=0,1,\ldots,m$.

Observe that, by Theorem~\ref{thm:bingo}, the partitions $P_2\wedge P_3$,
$P_1\wedge P_3$, $P_2\wedge P_3$ and $L$ arising from a regular Latin cube
of sort (LC2) are the coset partitions defined by the four subgroups $T_1$,
$T_2$, $T_3$, $T_0$ of $T^3$ just described in the case $m=3$ (see the last
four rows of Table~\ref{tab:coset}).

\begin{prop}
\label{p:diagsemi}
  \begin{enumerate}
\item The set $\{Q_0,\ldots,Q_m\}$ is invariant under the diagonal
group $D(T,m)$.
\item Any $m$ of the partitions $Q_0,\ldots,Q_m$ generate a
Cartesian lattice on $\Omega$ by taking suprema.
\end{enumerate}
\end{prop}

\begin{proof}
  \begin{enumerate}
\item It is clear that the set of partitions is invariant under 
right translations by elements of $T^m$ and left translations by elements of
the diagonal subgroup $T_0$, by automorphisms of $T$ (acting in the same
way on all coordinates), and under the symmetric group $S_m$ permuting the
coordinates. Moreover, it can be checked that the map
\[[t_1,t_2,\ldots,t_m]\mapsto[t_1^{-1},t_1^{-1}t_2,\ldots,t_1^{-1}t_m]\]
interchanges $Q_0$ and $Q_1$ and fixes the other partitions. So we have
the symmetric group $S_{m+1}$ acting on the whole set
$\{Q_0,\ldots,Q_m\}$. These transformations generate the diagonal group
$D(T,m)$; see Remark~\ref{rem:diaggens}. 

\item The set $T^m$ naturally has the structure of an $m$-dimensional
hypercube, and $Q_1,\ldots,Q_m$ are the minimal partitions in the
corresponding Cartesian lattice. For any other set of $m$ partitions,
the assertion follows because the symmetric group $S_{m+1}$ preserves
the set of $m+1$ partitions. 
\end{enumerate}
  \end{proof}

\begin{defn}
  Given a group~$T$ and an integer~$m$ with $m\geqslant 2$, define the partitions
  $Q_0$, $Q_1$, \ldots, $Q_m$ as above.
    For each subset $I$ of $\{0, \ldots, m\}$, put $Q_I = \bigvee_{i\in I}Q_i$.
    The \emph{diagonal semilattice} $\mathfrak{D}(T,m)$ is  the set
    $\{Q_I \mid I \subseteq \{0,1,\ldots,m\}\}$ of partitions of the set $T^m$.
\end{defn}

Thus the diagonal semilattice $\mathfrak{D}(T,m)$ is the set-theoretic union
of the ${m+1}$ Cartesian lattices in Proposition~\ref{p:diagsemi}(b).
Clearly it admits the diagonal group $D(T,m)$ as a group of automorphisms.

\begin{prop}
  \label{p:dsjs}
$\mathfrak{D}(T,m)$ is a join-semilattice, that is, closed under taking
joins. For $m>2$ it is not closed under taking meets.
\end{prop}

\begin{proof}
  For each proper subset $I$ of $\{0, \ldots, m\}$, the partition~$Q_I$ occurs
in the Cartesian lattice generated by $\{Q_i \mid i \in K\}$
  for every subset $K$ of $\{0,\ldots,m\}$ which contains $I$ and has
  cardinality~$m$.

  Let $I$ and $J$ be two proper subsets of $\{0,\ldots,m\}$.  If
  $\left |I \cup J\right| \leqslant m$ then there is a subset~$K$ of
  $\{0, \ldots, m\}$ with $\left|K\right|=m$ and $I\cup J \subseteq K$.
Then $Q_I\vee Q_J = Q_{I\cup J}$ in the Cartesian lattice defined by $K$, 
and this supremum does not depend on the choice of $K$.  Therefore
$Q_I\vee Q_J \in \mathfrak{D}(T,m)$.

On the other hand, if $I\cup J=\{0,\ldots,m\}$, then
\[Q_I\vee Q_J = Q_0 \vee Q_1 \vee \cdots \vee Q_m \succcurlyeq
Q_1 \vee Q_2 \vee \cdots \vee Q_m = U. 
\]
Hence $Q_I\vee Q_J=U$, and so $Q_I\vee Q_J\in \mathfrak{D}(T,m)$.

If $m=3$, consider the subgroups
\[H=T_0T_1=\{[x,y,y] \mid x,y\in T\}\quad\mbox{ and }
\quad K=T_2T_3=\{[1,z,w] \mid z,w\in T\}.\]
If $P_H$ and $P_K$ are the corresponding coset partitions, then
\[P_H=Q_{\{0,1\}}\quad \mbox{ and } \quad P_K=Q_{\{2,3\}},\]
which are both in $\mathfrak{D}(T,3)$. Now, by Proposition~\ref{prop:coset},
\[P_H\wedge P_K=P_{H\cap K},\]
where $H\cap K=\{[1,y,y] \mid y\in T\}$; this is a subgroup of $T^m$, but the
coset partition $P_{H\cap K}$ does not belong to $\mathfrak{D}(T,3)$. This example is
easily generalised to larger values of $m$. 
\end{proof}

When $T$~is finite, Propositions~\ref{p:diagsemi}(b) and~\ref{p:dsjs}
show that $\mathfrak{D}(T,m)$ is a Tjur block structure but is not an
orthogonal block structure when $m>2$ (see Section~\ref{sect:moreparts}).

We will see in the next section that the property in
Proposition~\ref{p:diagsemi}(b) is exactly what is required for the
characterisation of diagonal semilattices. First, we extend
Definition~\ref{def:weak}.

\begin{defn}\label{def:isomsl}
  For $i=1$, $2$, let $\mathcal{P}_i$ be a finite set of partitions of a
  set $\Omega_i$.  Then $\mathcal{P}_1$ is \textit{isomorphic} to
  $\mathcal{P}_2$ if there is a bijection $\phi$ from $\Omega_1$ to $\Omega_2$
  which induces a bijection from $\mathcal{P}_1$ to $\mathcal{P}_2$ which
  preserves the relation $\preccurlyeq$.
\end{defn}

As we saw in Section~\ref{sec:LS}, this notion of isomorphism
is called \textit{paratopism} in the context of Latin squares.

\medskip

The remark before Proposition~\ref{p:diagsemi} shows that a regular Latin
cube of sort (LC2) ``generates'' a diagonal semilattice $\mathfrak{D}(T,3)$
for a group $T$, unique up to isomorphism. The next step is to consider larger
values of $m$.

\subsection{The theorem}\label{sect:mt}

We repeat our axiomatisation of diagonal structures from the introduction.
We emphasise to the reader that we do not assume a Cartesian decomposition on 
the set $\Omega$ at the start; the $m+1$ Cartesian decompositions are imposed by
the hypotheses of the theorem, and none is privileged.

\begin{theorem}\label{th:main}
Let $\Omega$ be a set with $|\Omega|>1$, and $m$ an integer at least $2$. Let $Q_0,\ldots,Q_m$
be $m+1$ partitions of $\Omega$ satisfying the following property: any $m$
of them are the minimal non-trivial partitions in a Cartesian lattice on
$\Omega$.
\begin{enumerate}
\item If $m=2$, then the three partitions are the row, column, and letter
partitions of a Latin square on $\Omega$, unique up to paratopism.
\item If $m>2$, then there is a group $T$, unique up to group isomorphism,
such that $Q_0,\ldots,Q_m$ are the minimal non-trivial partitions in a diagonal
semilattice $\mathfrak{D}(T,m)$ on $\Omega$.
\end{enumerate}
\end{theorem}

Note that the converse of the theorem is true: Latin squares (with ${m=2}$)
and diagonal semilattices have the property that their minimal non-trivial
partitions do satisfy our hypotheses.

The general proof for $m\geqslant 3$ is by induction, the base case being $m=3$.
The base case follows from Theorem~\ref{thm:bingo}, as discussed in the 
preceding subsection, while the induction step
is given in Subsection~\ref{s:mtinduction}.

\subsection{Setting up}

First, we give some notation.
Let $\mathcal{P}$ be a set of partitions of $\Omega$,
and $Q$ a partition of~$\Omega$. We denote by $\mathcal{P}\quotient Q$ the
following object: take all partitions $P\in\mathcal{P}$ which satisfy
$Q\preccurlyeq P$; then regard each such $P$ as a partition, not of~$\Omega$, but
of~$Q$ (that is, of the set of parts of $Q$).
Then $\mathcal P\quotient Q$ is the set of these partitions of~$Q$.
(We do not write this as $\mathcal{P}/Q$, because this notation has almost the
opposite meaning in the statistical literature cited in
Section~\ref{sec:prelim}.)
The next result is routine but should help to familiarise this concept.

Furthermore, we will temporarily call a set $\{Q_0,\ldots,Q_m\}$ of partitions
of~$\Omega$ satisfying the hypotheses of
Theorem~\ref{th:main}
a \emph{special set of dimension $m$}.

\begin{prop}\label{p:quots}
Let $\mathcal{P}$ be a set of partitions of $\Omega$, and $Q$ a minimal
non-trivial element of $\mathcal{P}$.
\begin{enumerate}
\item If $\mathcal{P}$ is an $m$-dimensional Cartesian lattice, then
$\mathcal{P}\quotient Q$ is an $(m-1)$-dimensional Cartesian lattice.
\item If $\mathcal{P}$ is the join-semilattice generated by an $m$-dimensional
 special set $\mathcal{Q}$, and $Q\in\mathcal{Q}$, then $\mathcal{P}\quotient Q$
  is generated by a special set of dimension $m-1$.
\item If $\mathcal{P}\cong\mathfrak{D}(T,m)$ is a diagonal semilattice, then
$\mathcal{P}\quotient Q\cong\mathfrak{D}(T,m-1)$.
\end{enumerate}
\end{prop}

\begin{proof}
  \begin{enumerate}
  \item This follows from Proposition~\ref{p:antiiso}, because if $Q=P_I$
where $I = \{1, \ldots, m\} \setminus \{i\}$
    then we are effectively just limiting the set of indices to~$I$.
  \item
    This follows from part~(a).
  \item Assume that $\mathcal P=\mathfrak D(T,m)$. Then, since $\Aut(\mathcal P)$ 
  contains $D(T,m)$, which is transitive on $\{Q_0,\ldots,Q_m\}$, we may assume that $Q=Q_m$.
  Thus $\mathcal P\quotient Q$ is a set of partitions of $Q_m$. 
  In the group $T^{m+1}\rtimes\Aut(T)$ generated by elements of types
(I)--(III) in Remark~\ref{rem:diaggens}, the subgroup $T_m$ generated
by right multiplication of the last coordinate by elements of $T$ is normal,
and the quotient is $T^m\rtimes\Aut(T)$. Moreover, the subgroups $T_i$ commute
pairwise, so the parts of $Q_i\vee Q_m$ are the orbits of $T_iT_m$ (for
$i<m$) and give rise to a minimal partition in $\mathfrak{D}(T,m-1)$. 
    \end{enumerate}
\end{proof}

\subsection{Automorphism groups}
\label{sec:dag}
In the cases $m=2$ and $m=3$, we showed that the automorphism group of the
diagonal semilattice $\mathfrak{D}(T,m)$ is the diagonal group $D(T,m)$. The
same result holds for arbitrary $m$; but this time, we prove this result first,
since it is needed in the proof of the main theorem. The proof below also
handles the case $m=3$.

\begin{theorem}
For $m\geqslant2$, and any non-trivial group $T$, the automorphism group of the
diagonal semilattice $\mathfrak{D}(T,m)$ is the diagonal group $D(T,m)$.
\label{t:autDTm}
\end{theorem}

\begin{proof}
Our proof will be by induction on $m$. The cases $m=2$ and $m=3$ are given by
Theorems~\ref{t:autDT2} and~\ref{t:autDT3}. However, we base the induction at
$m=2$, so we provide an alternative proof for Theorem~\ref{t:autDT3}. So in
this proof we assume that $m>2$ and that the result holds with $m-1$
replacing~$m$.

Recall from Section~\ref{sect:diaggroups} that $\widehat D(T,m)$ denotes the 
pre-diagonal group, so that
$D(T,m)\cong \widehat D(T,m)/ \widehat K$, with $ \widehat K$
as in~\eqref{eq:K}.
Suppose that $\sigma:\widehat D(T,m)\to D(T,m)$ is the natural projection
with $\ker\sigma=\widehat K$. 

By Proposition~\ref{p:diagsemi}, we know that $D(T,m)$ is a subgroup of $\Aut(\mathfrak{D}(T,m))$, and we have
to show that equality holds. Using the principle of Proposition~\ref{p:subgp},
it suffices to show that the group $\SAut(\mathfrak{D}(T,m))$ of strong
automorphisms of $\mathfrak{D}(T,m)$ is the group $\sigma(T^{m+1}\rtimes\Aut(T))$ 
generated by
the images of the elements of the pre-diagonal group of types (I)--(III), as given in
Remark~\ref{rem:diaggens}.

Consider $Q_m$,  one of the minimal partitions in $\mathfrak{D}(T,m)$, and let
$\overline\Omega$ be the set of parts of $Q_m$. For $i<m$, the collection of
subsets of $\overline\Omega$ which are the parts of $Q_m$ inside a part of
$Q_i\vee Q_m$ is a partition $\overline Q_i$ of $\overline\Omega$.
Proposition~\ref{p:quots}(c) shows that the $\overline Q_i$ are the minimal
partitions of $\mathfrak{D}(T,m-1)$, a diagonal semilattice
on~$\overline\Omega$. 
Moreover, the group $\sigma(T_m)$ is the
kernel of the action of $\sigma(T^{m+1}\rtimes\Aut(T))$ on~$\overline\Omega$.
Further, since $T_m\cap \widehat K=1$, $\sigma(T_m)\cong T_m\cong T$. 
As in Section~\ref{sect:diaggroups}, let $\widehat H$ be the 
stabiliser in $\widehat D(T,m)$ of the element $[1,\ldots,1]$:
then  $T_m\cap \widehat H=1$ and so $T_m$ acts faithfully and
regularly on each part of $Q_m$.

So it suffices to show that the same is true of $\SAut(\mathfrak{D}(T,m))$;
in other words, it is enough to show that the subgroup $H$ of $\SAut(\mathfrak{D}(T,m))$ 
fixing setwise  all parts of $Q_m$
and any given point $\alpha$ of $\Omega$ is trivial. 

Any $m$ of the partitions $Q_0,\ldots,Q_m$ are the minimal partitions
in a Cartesian lattice of partitions of $\Omega$. Let $P_{ij}$ denote the supremum of the partitions
$Q_k$ for $k\notin\{i,j\}$. Then, for fixed $i$, the partitions $P_{ij}$
(as $j$ runs over $\{0,\ldots,m\}\setminus\{i\}$) are the maximal partitions
of the Cartesian lattice 
generated by $\{ Q_j \mid 0\leqslant j\leqslant~m \mbox{ and } j\ne i\}$
and form a Cartesian decomposition of~$\Omega$.
Hence  each  point of $\Omega$ is uniquely determined
by the parts of these partitions which contain it
(see Definition~\ref{def:cart}).

For distinct $i,j<m$, all parts of $P_{ij}$ are fixed by $H$, since each is a union of
parts of $Q_m$. Also, for $i<m$, the part of $P_{im}$ containing $\alpha$ is
fixed by $H$. By the defining property of the Cartesian decomposition 
$\{P_{ij}\mid 0\leqslant j\leqslant m\mbox{ and }j\neq i\}$, we conclude that $H$ fixes every point lying in
the same part of $P_{im}$ as $\alpha$ and this holds for all $i<m$.

Taking $\alpha=[1,\ldots,1]$, the argument in the last two paragraphs shows 
in particular that 
$H$ fixes pointwise the part $P_{0m}[\alpha]$ of $P_{0m}$ and the part 
$P_{1m}[\alpha]$ of $P_{1m}$ containing 
$\alpha$. In other words, $H$ fixes pointwise the sets 
\begin{align*}
  P_{0m}[\alpha]&=\{[t_1,\ldots,t_{m-1},1]\mid t_1,\ldots,t_{m-1}\in T\}\mbox{ and}\\
  P_{1m}[\alpha]&=\{[t_1,\ldots,t_{m-1},t_1]\mid t_1,\ldots,t_{m-1}\in T\}.
\end{align*}
Applying, for a given $t\in T$, the same argument to the element $\alpha'=[t,1,\ldots,1,t]$ 
of $P_{1m}[\alpha]$, we obtain that $H$ fixes pointwise the set 
\[
  P_{0m}[\alpha']=\{[t_1,\ldots,t_{m-1},t]\mid t_1,\ldots,t_{m-1}\in T\}.
\]
Letting $t$ run through the elements of $T$, the union of the 
parts $P_{0m}[\alpha']$ is $\Omega$, and 
this implies that  $H$ fixes all elements of $\Omega$ and we are done. 
\end{proof}

The particular consequence of Theorem~\ref{t:autDTm} that we require in the proof of the 
main theorem is the following.

\begin{cor}\label{c:forinduction}
Suppose that $m\geqslant3$. Let $\mathcal P$ and $\mathcal P'$ be 
diagonal semilattices isomorphic to $\mathfrak D(T,m)$, and let $Q$ and 
$Q'$ be minimal partitions in
$\mathcal P$ and $\mathcal P'$, respectively. 
Then each isomorphism $\psi:\mathcal P\quotient Q\to \mathcal P'\quotient Q'$
is induced by an isomorphism $\overline{\psi}: \mathcal P\to \mathcal P'$ 
mapping $Q$ to $Q'$. 
\end{cor}

\begin{proof}
We may assume without loss of generality that $\mathcal P=\mathcal P'=\mathfrak D(T,m)$ and, 
since $\Aut(\mathfrak D(T,m))$ induces $S_{m+1}$ on the minimal partitions
$Q_0,\ldots,Q_m$
of $\mathfrak D(T,m)$, we can also suppose that $Q=Q'=Q_m$. 
Thus $\mathcal P\quotient Q= \mathcal P'\quotient Q'\cong \mathfrak D(T,m-1)$.
Let $\sigma:\widehat D(T,m)\to D(T,m)$ be the natural projection map, 
as in the proof of 
Theorem~\ref{t:autDTm}.
The subgroup of $\Aut(\mathfrak D(T,m))$ fixing $Q_m$ is the image 
$X=\sigma(T^{m+1}\rtimes (\Aut(T)\times S_m))$ where the subgroup $S_m$ of $S_{m+1}$ 
is 
the stabiliser of the point $m$ in the action on $\{0,\ldots,m\}$. 
Moreover, the subgroup $X$ contains $\sigma(T_m)$, the copy of $T$ acting on the last
coordinate of the $m$-tuples, which is regular on each part
of $Q_m$.  Put $Y=\sigma(T_m)$. Then $Y$~is the kernel of the induced action
of $X$ on $\mathcal P\quotient Q_m$, which is isomorphic to $
\mathfrak D(T,m-1)$, and so $X/Y\cong D(T,m-1)$. Moreover since $m\geqslant 3$,
it follows from Theorem~\ref{t:autDTm} that 
$X/Y = \Aut(\mathfrak D(T,m-1))$. Thus the given map $\psi$ in
$\Aut(\mathfrak D(T,m-1))$ 
lies in $X/Y$, and we may choose $\overline{\psi}$ as any pre-image of $\psi$ in $X$. 
\end{proof}

\subsection{Proof of the main theorem}\label{s:mtinduction}

Now we begin the proof of Theorem~\ref{th:main}. The proof is by induction
on $m$. As we remarked in 
Section~\ref{sect:mt}, 
there is nothing to prove for $m=2$, and the case $m=3$ follows from 
Theorem~\ref{thm:bingo}. Thus we assume that $m\geqslant4$. The induction
hypothesis yields that the main theorem is true for dimensions~$m-1$ and~$m-2$.
Given a special set $\{Q_0,\ldots,Q_m\}$ generating a semilattice $\mathcal{P}$,
we know, by Proposition~\ref{p:quots}, that, for each $i$, $\mathcal{P}\quotient Q_i$ 
is generated by a special
set of dimension $m-1$, and so is isomorphic to $\mathfrak{D}(T,m-1)$ for
some group $T$. Now, $T$ is independent of the choice of $i$; for, if
$\mathcal{P}\quotient Q_i\cong\mathfrak{D}(T_i,m-1)$, and
$\mathcal{P}\quotient Q_j\cong\mathfrak{D}(T_j,m-1)$, then, 
by Proposition~\ref{p:quots}(c),
\[
\mathfrak{D}(T_i,m-2)\cong\mathcal{P} \quotient (Q_i\vee Q_j)
\cong\mathfrak{D}(T_j,m-2),
\]
so by induction $T_i\cong T_j$.
(This proof works even when $m=4$, because it is the reduction to $m=3$ that
gives the groups $T_i$ and $T_j$, so that the Latin squares 
$\mathfrak{D}(T_i,2)$ and $\mathfrak{D}(T_j,2)$ are both Cayley tables of groups, 
and so Theorem~\ref{thm:albert} implies that $T_i\cong T_j$.)

We call $T$ the \emph{underlying group} of the special set.

\begin{theorem}
  \label{th:QQ}
Let $\mathcal{Q}$ and $\mathcal{Q}'$ be special sets of dimension $m\geqslant4$
on sets $\Omega$ and $\Omega'$ with the same underlying group $T$.
Then $\mathcal{Q}$ and $\mathcal{Q'}$ are isomorphic in the sense of 
Definition~\ref{def:isomsl}.
\end{theorem}

\begin{proof}
  Let $\mathcal{P}$ and $\mathcal{P}'$ be the join-semilattices
  generated by $\mathcal{Q}$ and $\mathcal{Q}'$ respectively,
  where $\mathcal{Q} = \{Q_0, \ldots, Q_m\}$ and
  $\mathcal{Q}' = \{Q'_0, \ldots, Q'_m\}$.

We consider the three partitions $Q_1$, $Q_2$, and
$Q_1\vee Q_2$. Each part of $Q_1\vee Q_2$ is partitioned by $Q_1$ and $Q_2$;
these form a $|T|\times|T|$ grid, where the parts of $Q_1$ are the rows and
the parts of $Q_2$ are the columns. We claim that
\begin{itemize}
\item There is a bijection $F_1$ from the set of parts of $Q_1$ to the set of
parts of $Q_1'$ which induces an isomorphism from $\mathcal{P} \quotient Q_1$ to
$\mathcal{P}' \quotient Q_1'$.
\item There is a bijection $F_2$ from the set of parts of $Q_2$ to the set of
parts of $Q_2'$ which induces an isomorphism from $\mathcal{P} \quotient Q_2$ to
$\mathcal{P}' \quotient Q_2'$.
\item There is a bijection $F_{12}$ from the set of parts of $Q_1\vee Q_2$ to
the set of parts of $Q_1'\vee Q_2'$ which induces an isomorphism from
$\mathcal{P} \quotient (Q_1\vee Q_2)$ to $\mathcal{P}' \quotient (Q_1'\vee Q_2')$;
moreover, each of $F_1$ and $F_2$, restricted to the partitions of
$\mathcal{P}\quotient (Q_1\vee Q_2)$, agrees with $F_{12}$.
\end{itemize}

The proof of these assertions is as follows. 
As each part of $Q_1 \vee Q_2$ is a union of parts of $Q_1$, 
the partition $Q_1 \vee Q_2$ determines a  partition $R_1$ of 
$Q_1$ which is a minimal partition of $\mathcal P\quotient Q_1$. 
Similarly  $Q'_1 \vee Q'_2$ determines a minimal partition $R_1'$ of  $\mathcal P'\quotient 
Q_1'$. 
Then since $\mathcal P\quotient Q_1\cong \mathcal P'\quotient Q_1'\cong \mathfrak D(T,m-1)$, 
by the induction hypothesis, as discussed above,
we may choose an isomorphism 
$F_1: \mathcal P\quotient Q_1\to \mathcal P'\quotient Q_1'$ 
in the first bullet point such that $R_1$ is mapped to $R_1'$.
Now $F_1$ induces an isomorphism  
$(\mathcal P\quotient Q_1)\quotient R_1 \to (\mathcal P'\quotient Q'_1)\quotient R_1'$,
and since there are natural isomorphisms from 
$(\mathcal P\quotient Q_1)\quotient R_1$ to
$\mathcal P\quotient (Q_1 \vee Q_2)$ and  from
$(\mathcal P'\quotient Q'_1)\quotient R_1'$ to
$\mathcal P'\quotient (Q'_1 \vee Q'_2)$, 
$F_1$ induces an isomorphism 
\[F_{12}: \mathcal P\quotient (Q_1 \vee Q_2) \to 
\mathcal P'\quotient (Q'_1 \vee Q'_2).
\]
The join $Q_1 \vee Q_2$ determines a partition 
$R_2$ of $Q_2$ which is a minimal partition of $\mathcal P\quotient Q_2$, and   
$Q'_1 \vee Q'_2$ determines a minimal partition $R'_2$ of  
$\mathcal P'\quotient Q_2'$. Further,  we have natural isomorphisms from  
$(\mathcal P\quotient Q_2)\quotient R_2$ to $\mathcal P\quotient (Q_1 \vee Q_2)$ and from
$(\mathcal P'\quotient Q'_2)\quotient R'_2$ to $\mathcal P'\quotient (Q'_1 \vee Q'_2)$, 
so we may view  $F_{12}$ as an isomorphism from 
$(\mathcal P\quotient Q_2)\quotient R_2$ to $(\mathcal P'\quotient Q'_2)\quotient R'_2$. 
By Corollary~\ref{c:forinduction}, the isomorphism $F_{12}$ is induced by an
isomorphism from $\mathcal{P} \quotient Q_2$ to $\mathcal{P}' \quotient Q_2'$,
and we take $F_2$ to be this isomorphism.

Thus, $F_{12}$ maps each part $\Delta$ of $Q_1\vee Q_2$ to a part $\Delta'$ of
$Q_1'\vee Q_2'$, and $F_1$ maps the rows of the grid on $\Delta$ described above to the rows of
the grid on $\Delta'$, and similarly $F_2$ maps the columns.

Now the key observation is that there is a unique bijection~$F$ from the points
of $\Delta$ to the points of $\Delta'$ which maps rows to rows (inducing~$F_1$)
and columns to columns (inducing~$F_2$). For each point of $\Delta$ is the
intersection of a row and a column, and can be mapped to the
intersection of the image row and column in $\Delta'$.

Thus, taking these maps on each part of $Q_1\vee Q_2$ and combining them,
we see that there is a unique bijection $F\colon\Omega\to\Omega'$ which induces $F_1$
on the parts of~$Q_1$ and $F_2$ on the parts of~$Q_2$. Since $F_1$ is an
isomorphism from $\mathcal{P} \quotient Q_1$ to $\mathcal{P}' \quotient Q_1'$,
and similarly for $F_2$, we see that
\begin{quote}
$F$ maps every element of $\mathcal{P}$ which is above \emph{either}
$Q_1$ or $Q_2$ to the corresponding element of $\mathcal{P}'$.
\end{quote}

To complete the proof, we have to deal with the remaining partitions of $\mathcal P$
and $\mathcal P'$. 
We note that every partition in $\mathcal{P}$ has the form
\[Q_I=\bigvee_{i\in I}Q_i\]
for some  $I\subseteq\{0,\ldots,m\}$. By the statement proved in the previous paragraph, 
we may assume that $I\cap\{1,2\}=\emptyset$ and in particular that 
$|I|\leqslant m-1$. 

Suppose first  that $|I|\leqslant m-2$. Then there is some $k\in\{0,3,\ldots,m\}$
such that $k\not\in I$. Without loss of generality we may assume that 
$0\not\in I$.
Since $\{Q_1,\ldots,Q_m\}$ generates a Cartesian lattice, which is closed
under meet, we have
\[Q_I=Q_{I\cup\{1\}}\wedge Q_{I\cup\{2\}},\]
and since the partitions on the right are mapped by $F$ to $Q'_{I\cup\{1\}}$ and
$Q'_{I\cup\{2\}}$, it follows that $F$ maps $Q_I$ to $Q'_I$.

Consider finally the case when $|I|=m-1$; that is, $I=\{0,3,4,\ldots,m\}$. 
As $m\geqslant 4$, we have   $0, 3\in I$ and may put
$J = I\setminus \{0,3\}=\{4,\ldots,m\}$.
Then, for $i\in\{0,3\}$, $\left| J \cup \{i\} \right|= m-2$, so
the argument in the previous paragraph  shows that  $F$ maps $Q_{J \cup \{i\}}$ 
to $Q'_{J \cup \{i\}}$. 
Since $Q_I = Q_{J \cup \{0\}} \vee Q_{J \cup \{3\}}$, it follows
that $F$ maps $Q_I$ to $Q'_I$. \
\end{proof}

Now the proof of the main theorem follows. For let $\mathcal{Q}$ be a special
set of partitions of $\Omega$ with underlying group $T$.
By Proposition~\ref{p:diagsemi},
the set of minimal partitions in $\mathfrak{D}(T,m)$ has the same property.
By Theorem~\ref{th:QQ}, $\mathcal{Q}$~is isomorphic to this special set,
so the
join-semilattice it generates is isomorphic to~$\mathfrak{D}(T,m)$.

\section{Primitivity and quasiprimitivity}\label{s:pqp}

A permutation group is said to be \emph{quasiprimitive} if all its non-trivial
normal subgroups are transitive. In particular, primitive groups are
quasiprimitive, but a quasiprimitive group may be imprimitive. If $T$ is a
(not necessarily finite) simple group and $m\geqslant 2$, then the diagonal group
$D(T,m)$ is a primitive permutation group of simple diagonal type;
see~\cite{aschsc}, \cite{kov:sd}, or~\cite[Section~7.4]{ps:cartesian}.
In this section, we investigate the primitivity and quasiprimitivity of diagonal
groups for an arbitrary~$T$; our conclusions are in Theorem~\ref{th:primaut} in
the introduction.

The proof requires some preliminary lemmas.

A subgroup  of a group~$G$ is \emph{characteristic} if it is
invariant under $\Aut(G)$. We say that $G$~is \emph{characteristically simple}
if its only characteristic subgroups are itself and $1$. We require some
results about abelian characteristically simple groups.

An abelian group $(T,+)$ is said to be  \emph{divisible}
if, for every positive integer~$n$ and every $a\in T$,
there exists $b\in T$ such that $nb=a$. The group $T$ is
\emph{uniquely divisible} if, 
for all $a\in T$ and $n\in\mathbb{N}$, the element $b\in T$ is unique. Equivalently,
an abelian group $T$ is divisible if and only if
the map $T\to T$, $x\mapsto n x$ is surjective for all $n\in\mathbb{N}$, while
$T$ is uniquely divisible if and only if the same map is bijective
for all $n\in\mathbb{N}$. Uniquely divisible groups are also referred to as
\emph{$\mathbb{Q}$-groups}.  If $T$ is a uniquely divisible group,
$p\in\mathbb{Z}$, $q\in \mathbb{Z}\setminus\{0\}$ and $a\in T$, then there is
a unique $b\in T$ such that $qb=a$ and we define $(p/q)a=pb$. 
This defines a $\mathbb{Q}$-vector space
structure on~$T$. Also note that any non-trivial uniquely divisible group is
torsion-free.

In the following lemma, elements of $T^{m+1}$ are written as 
$(t_0,\ldots,t_m)$ with $t_i\in T$,
and $S_{m+1}$ is considered as the symmetric group
acting on the set $\{0,\ldots,m\}$. Moreover, we let $H$ denote
the group $\Aut(T)\times S_{m+1}$; then $H$ acts on $T^{m+1}$ by
\begin{equation}\label{eq:Gomegaact}
(t_0,\ldots,t_m)(\varphi,\pi)=(t_{0\pi^{-1}}\varphi,\ldots,t_{m\pi^{-1}}\varphi)  
\end{equation}
for all $(t_0,\ldots,t_m)$ in $T^{m+1}$, $\varphi$ in $\Aut(T)$,
and $\pi$ in $S_{m+1}$. 
The proof of statements (b)--(c) depends on the 
assertion that bases exist in an arbitrary vector space, which is a well-known
consequence of the Axiom
of Choice. Of course, in special cases, for instance when $T$ is finite-dimensional
over $\mathbb{F}_p$ or over $\mathbb{Q}$, then the use of the Axiom of Choice can be avoided.

\begin{lem}\label{lem:charab}
  The following statements  hold for any non-trivial abelian
  characteristically simple group~$T$.
  \begin{enumerate}
  \item Either $T$ is  an elementary abelian $p$-group or
   $T$ is a uniquely divisible group. Moreover, $T$ can be considered
   as an $\mathbb{F}$-vector space, where $\mathbb{F}=\mathbb{F}_p$ in the first
   case, while $\mathbb F=\mathbb{Q}$ in the second case.
 \item $\Aut (T)$ is transitive on the set $T\setminus\{0\}$.
  \item Suppose that $m\geqslant 1$ and put 
    \begin{align*}
    \Delta&=\delta(T,m+1)=\{(t,\ldots,t)\in T^{m+1}\mid t\in T\}\mbox{ and }\\
    \Gamma&=\left\{(t_0,\ldots,t_m)\in T^{m+1}\mid \sum_{i=0}^mt_i=0\right\}.
    \end{align*}
    Then $\Delta$ and $\Gamma$  are $H$-invariant subgroups of $T^{m+1}$. 
    Furthermore,  precisely one of the following holds.
    \begin{enumerate}
      \item $T$ is an elementary abelian $p$-group where $p\mid(m+1)$,
so that $\Delta\leqslant \Gamma$. In particular, $\Gamma/\Delta$ is an
$H$-invariant subgroup of $T^{m+1}/\Delta$, which is proper if $m\geqslant2$
\item Either $T$ is uniquely divisible or $T$ is an elementary
  abelian $p$-group with $p\nmid (m+1)$.  Further, in this  case, 
  $T^{m+1}=\Gamma\oplus \Delta$ and $\Gamma$ has no  proper, non-trivial
  $H$-invariant subgroup.
    \end{enumerate} 
  \end{enumerate}  
\end{lem}

\begin{proof}
  \begin{enumerate}
    \item
   First note that, for $n\in\mathbb{N}$,  both the image $nT$
   and the kernel $\{t\in T\mid nt=0\}$ of the map $t\mapsto nt$ are
   characteristic subgroups   of $T$.

   If $T$ is not a divisible group, then there exist $n\in\mathbb{N}$ and
   $a\in T$ such that $a \notin nT$. Thus 
  $nT\neq T$, and hence, since $T$ is characteristically simple, $nT=0$. 
  In particular,  $T$ contains a non-zero element of finite order,
  and hence $T$ also contains an element of order $p$ for some prime~$p$.
  Since $T$ is abelian, the set $Y=\{t\in T\mid pt=0\}$ is a non-trivial
  characteristic subgroup, and so $Y=T$; that is, $T$ is an
  elementary abelian $p$-group and it can be
  regarded as an $\mathbb F_p$-vector space. 

Hence we may assume that $T$ is a non-trivial divisible group. That is,
$nT=T$ for all $n\in\mathbb{N}$, but also, as $T$ is characteristically simple, 
  $\{t\in T\mid nt=0\}=\{0\}$
  for all $n\in \mathbb{N}$. Hence $T$ is uniquely divisible. In this case, $T$ can be viewed
  as a $\mathbb{Q}$-vector space, as explained before the statement of this lemma.
  \item
   By part~(a), $T$ can be considered as a vector space over some field
  $\mathbb F$. If $a,b\in T\setminus\{0\}$, then, by extending the sets $\{a\}$ and
  $\{b\}$ into $\mathbb F$-bases, we can construct an $\mathbb F$-linear
  transformation that takes $a$ to $b$.

  \item
    The definition of $\Delta$ and $\Gamma$ implies that they are
    $H$-invariant, and also that, if $T$ is an elementary abelian $p$-group
    such that $p$ divides $m+1$, then $\Delta<\Gamma$, and so $\Gamma/\Delta$ is a
    proper $H$-invariant subgroup of $T^{m+1}/\Delta$. 

Assume now that
    either $T$ is uniquely divisible or  $T$ is a $p$-group with $p\nmid(m+1)$.
Then $T^{m+1}=\Delta\oplus \Gamma$ where the decomposition is into the direct
sum of $H$-modules.  It suffices to show that,
if $\mathbf{a}=(a_0,\ldots,a_m)$ is a non-trivial element of $\Gamma$,
then the  smallest
$H$-invariant subgroup $X$ that contains
$\mathbf{a}$ is  equal to $\Gamma$. 
  The non-zero element $\mathbf a$ of $\Gamma$ cannot be of the form $(b,\ldots,b)$ 
  for $b\in T\setminus\{0\}$, 
  because $(m+1)b\neq 0$ whether $T$ is uniquely divisible or $T$ is a $p$-group 
  with $p\nmid(m+1)$. In 
  particular there exist distinct $i,j$ in $\{0,\ldots,m\}$ 
  such that $a_i\neq a_j$.
  Applying an element $\pi$ in $S_{m+1}$,
  we may assume without loss of generality
  that $a_0\neq a_1$. Applying the transposition $(0,1)\in S_{m+1}$,
  we have that
  $(a_1,a_0,a_2,\ldots,a_m)\in X$, and so
  \[
  (a_0,a_1,a_2,\ldots,a_m)-(a_1,a_0,a_2,\ldots,a_m)=(a_0-a_1,a_1-a_0,0,\ldots,0)\in X.
  \]
  Hence there is a non-zero element $a\in T$ such that $(a,-a,0,\ldots,0)\in X$.
  By part~(b),  $\Aut(T)$ is transitive on non-zero
  elements of $T$ and hence  $(a,-a,0,\ldots,0)\in X$ for
  all $a\in T$. As $S_{m+1}$ is transitive on pairs of indices $i,j\in\{0,\ldots,m\}$ with 
  $i\neq j$, this implies that
  all elements of the form $(0,\ldots,0,a,0,\ldots,0,-a,0,\ldots,0)\in T^{m+1}$ belong
  to $X$, but these elements generate $\Gamma$, and so $X=\Gamma$, as required.
  \end{enumerate}
\end{proof}

Non-abelian characteristically simple groups are harder to describe.
A direct product of pairwise isomorphic non-abelian simple groups is 
characteristically simple.
Every finite characteristically simple group is of this form, but in the
infinite case this is not true; the first example of a
characteristically simple group not of this form was published by
McLain~\cite{mclain} in 1954, see also Robinson~\cite[(12.1.9)]{djsr}.

\medskip

Now we work towards the main result of this section, the classification of
primitive or quasiprimitive diagonal groups. First we do the case where $T$ is
abelian.

\begin{lem}\label{lem:prabreg}
  Let $G$ be a permutation group on a set $\Omega$ and let $M$ be an
  abelian regular normal subgroup of $G$. If $\omega\in\Omega$, then
   $G=M\rtimes G_\omega$ and the following are
  equivalent:
  \begin{enumerate}
  \item $G$ is primitive;
  \item $G$ is quasiprimitive;
  \item $M$ has no proper non-trivial subgroup which is invariant under
    conjugation by elements of $G_\omega$.
  \end{enumerate}
  \end{lem}

\begin{proof}
  The product decomposition $G=MG_\omega$ follows from the transitivity of $M$, while 
  $M\cap G_\omega=1$ follows from the regularity of $M$. Hence $G=M\rtimes G_\omega$. 
  Assertion~(a) clearly implies assertion~(b). The fact that (b) implies (c) follows
  from~\cite[Theorem~3.12(ii)]{ps:cartesian} by noting that $M$, being abelian, has no non-trivial inner automorphisms.
  Finally, that (c) implies (a) follows directly from~\cite[Theorem~3.12(ii)]{ps:cartesian}. 
\end{proof}

%

To handle the case where $T$ is non-abelian, we need the following definition
and lemma.

A group $X$ is said to be \emph{perfect} if $X'=X$,
where $X'$ denotes the commutator subgroup.
The following lemma is Lemma 2.3 in \cite{charfact}, where the proof can be
found. For $X=X_1\times\cdots\times X_k$ a direct product of groups and
$S\subseteq\{1,\ldots,k\}$, we denote by $\pi_S$  the projection
from $X$ onto $\prod_{i\in S}X_i$.

\begin{lem}\label{comminside}
Let $k$ be a positive integer, let $X_1,\ldots,X_k$ be groups, and suppose, for
$i\in \{1,\ldots,k\}$, that $N_i$ is a perfect subgroup of $X_i$.
Let $X=X_1\times\cdots\times X_k$ and let $K$ be a subgroup of $X$ such that for
all $i$, $j$ with $1\leqslant i<j\leqslant k$, we have
$N_i\times N_j\leqslant \pi_{\{i,j\}}(K)$.  Then $N_1\times\cdots\times N_k\leqslant K$.

\end{lem}

Now we are ready to prove Theorem~\ref{th:primaut}. In this proof, $G$ denotes
the group $D(T,m)$ with $m\geqslant2$. As defined earlier in this section, we let
$H=A\times S$, where $A=\Aut(T)$ and $S=S_{m+1}$.
Various properties of diagonal groups whose proofs are straightforward are 
used without further comment.

\begin{proof}[Proof of Theorem~\ref{th:primaut}]
We prove (a)~$\Rightarrow$~(b)~$\Rightarrow$~(c)~$\Rightarrow$~(a).
\begin{itemize}
\item[(a)$\Rightarrow$(b)] Clear.
\item[(b)$\Rightarrow$(c)] We show that $T$ is characteristically simple
by proving the contrapositive. Suppose that $N$ is
a non-trivial proper characteristic subgroup of $T$.
Then $N^{m+1}$ is a normal subgroup of $G$, as is readily
checked. We claim that the orbit of the point $[1,1,\ldots,1]\in\Omega$ 
under $N^{m+1}$ is $N^m$. We have to check that this set is fixed by right
multiplication by $N^m$ (this is clear, and it is also clear that it is a
single orbit), and that
left multiplication of every coordinate by a fixed element
of $N$ fixes $N^m$ (this is also clear). So $D(T,m)$ has an intransitive
normal subgroup, and is not quasiprimitive.

If $T$ is abelian, then it is either an elementary abelian $p$-group or
uniquely divisible. In the former case, if $p\mid(m+1)$, the subgroup
$\Gamma$ from Lemma~\ref{lem:charab} acts intransitively
on $\Omega$, and is normalised by
$H$; so $G$ is not
quasiprimitive, by Lemma~\ref{lem:prabreg}. (The image of $[0,\ldots,0]$
under the element $(t_0,\ldots,t_m)\in\Gamma$ is 
$[t_1-t_0,t_2-t_0,\ldots,t_m-t_0]$, which has coordinate sum zero since
$-mt_0=t_0$. So the orbit of $\Gamma$ consists of $m$-tuples with coordinate
sum zero.)
\item[(c)$\Rightarrow$(a)] Assume that $T$ is characteristically simple, and
not an elementary abelian $p$-group for which $p\mid(m+1)$.

If $T$ is abelian, then it is either uniquely divisible or an elementary 
abelian $p$-group with $p\nmid(m+1)$. Then
Lemma~\ref{lem:charab}(c) applies; $T^{m+1}=\Gamma\oplus\Delta$, where
$\Delta$ is the kernel of the action of $T^{m+1}$ on $\Omega$, and $\Gamma$ contains no
proper non-trivial $H$-invariant subgroup; so by Lemma~\ref{lem:prabreg},
$G$ is primitive.

So we may suppose that $T$ is non-abelian and characteristically simple.
Then $Z(T)=1$, and so $T^{m+1}$ acts faithfully on $\Omega$,
and its subgroup $R=T^m$ (the set of elements of $T^{m+1}$ of the form
$(1,t_1,\ldots,t_m)$) acts regularly.

Let $L=\{(t_0,1,\ldots,1) \mid t_0\in T\}$.
Put $N=T^{m+1}$. Then $RL=LR=N \cong L\times R$.
We identify $L$ with $T_0$ and $R$ with $T_1 \times \cdots \times T_m$.
Then $N$ is normal in $G$, and $G=NH$.

Let $\omega=[1,\ldots,1]\in\Omega$ be fixed. Then
$G_\omega=H$ and $N_\omega=I$, where $I$ is the subgroup of $A$
consisting of inner automorphisms of~$T$.

To show that $G$ is primitive on $\Omega$, we show that $G_\omega$ is a
maximal subgroup of $G$. So let $X$ be a  subgroup of $G$ that properly
contains $G_\omega$. We will show that $X=G$.

Since $S\leqslant X$, we have that $X=(X\cap (NA))S$.
Similarly, as $N_\omega A \leqslant X \cap (NA)$, we find that
$X \cap (N A) = (X \cap N) A$.
So $X = (X \cap N) (A S) = (X \cap N) G_\omega$.
Then, since $G_\omega$ is a proper subgroup of $X$ and $G_\omega \cap N = N_\omega$,
it follows that $X \cap N$ properly contains $N_\omega$.
Set $X_0=X\cap N$.
  Thus there exist some pair $(i,j)$ of distinct indices
  and an element $(u_0,u_1,\ldots,u_m)$ in $X_0$ such that $u_i\neq u_j$. Since
  $(u_i^{-1},\ldots,u_i^{-1}) \in X_0$, it follows that there exists an
  element $(t_0,t_1,\ldots,t_m)\in X_0$ such that $t_i=1$ and $t_j\neq~1$.
  Since $S\cong S_{m+1}$ normalises $N_\omega A$ and permutes the
  direct factors of $N=T_0\times T_1\times \cdots \times T_m$ naturally, 
  we may assume without loss of generality that $i=0$ and $j=1$, and hence that
  there exists an
  element $(1,t_1,\ldots,t_m)\in X_0$ with $t_1\neq 1$; that is,
  $T_1\cap\pi_{0,1}(X_0)\neq 1$,
  where $\pi_{0,1}$ is the projection from $N$ onto $T_0\times T_1$.

  If $\psi\in A$, then $\psi$ normalises $X_0$ and acts
coordinatewise on $T^{m+1}$; so $(1,t_1^\psi,\ldots,t_m^\psi)\in X_0$, so that
$t_1^\psi\in T_1\cap \pi_{0,1}(X_0)$. Now,
$\{t_1^\psi \mid \psi \in A\}$ generates a characteristic subgroup of~$T_1$.
Since $T_1$ is characteristically simple, $T_1\leqslant\pi_{0,1}(X_0)$. A
similar argument shows that $T_0\leqslant \pi_{0,1}(X_0)$. Hence
$T_0\times T_1=\pi_{0,1}(X_0)$. Since the group $S\cong S_{m+1}$ acts 
$2$-transitively on the direct factors of $N$,  and since $S$ normalises $X_0$
(as $S< G_\omega<X$), we
obtain, for all distinct $i,\ j\in\{1,\ldots,m\}$,  that
$\pi_{i,j}(X_0)=T_i\times T_j$ (where $\pi_{i,j}$ is the projection onto
$T_i\times T_j$).

Since the $T_i$ are non-abelian characteristically simple groups, they are
perfect. Therefore Lemma~\ref{comminside} implies that $X_0=N$, and hence
$X=(X_0A)S=G$. Thus $G_\omega$ is a maximal subgroup of $G$, and $G$ is 
primitive, as required. 
\end{itemize}
\end{proof}

In the case $m=1$, diagonal groups behave a little differently. If $T$ is
abelian, then the diagonal group is simply the holomorph of $T$, which is
primitive (and hence quasiprimitive) if and only if $T$ is characteristically
simple. The theorem is true as stated if $T$ is non-abelian, in which case
the diagonal group is the permutation group on $T$ generated by left and right
multiplication, inversion, and automorphisms of~$T$.

\section{The diagonal graph}\label{s:diaggraph}

The diagonal graph is a graph which stands in a similar relation to the
diagonal semilattice as the Hamming graph does to the Cartesian lattice.
In this section, we define it, show that apart from a few small cases its
automorphism group is the diagonal group, and investigate some of its
properties, including its connection with the permutation group property 
of \emph{synchronization}.

We believe that this is an interesting class of graphs, worthy of study by
algebraic graph theorists. The graph $\Gamma_D(T,m)$ has appeared in some
cases: when $m=2$ it is the Latin-square graph associated with the Cayley
table of~$T$, and when $T=C_2$ it is the \emph{folded cube}, a
distance-transitive graph.

\subsection{Diagonal graph and diagonal semilattice}
\label{sec:dgds}

In this subsection we define the \emph{diagonal graph} $\Gamma_D(T,m)$ associated
with a diagonal semilattice $\mathfrak{D}(T,m)$. We show that, except for five
small cases (four of which we already met in the context of Latin-square graphs
in Section~\ref{sect:lsautgp}), the
diagonal semilattice and diagonal graph determine each other, and so they have
the same automorphism group, namely $D(T,m)$.

Let $\Omega$ be the underlying set of a diagonal semilattice
$\mathfrak{D}(T,m)$, for $m\geqslant2$ and for a not necessarily finite group $T$. Let $Q_0,\ldots,Q_m$ be the minimal partitions
of the semilattice  (as in Section~\ref{sec:diag1}). We define the diagonal graph as follows.
The vertex set is $\Omega$; two vertices are joined if they lie in the same
part of $Q_i$ for some $i$ with $0\leqslant i\leqslant m$. Since parts of distinct $Q_j$, $Q_{j'}$ intersect in at most one point, the value of $i$ is unique.  Clearly
the graph is regular with valency $(m+1)(|T|-1)$ (if $T$ is finite).

We represent the vertex set by $T^m$, with $m$-tuples in square brackets.
Then $[t_1,\ldots,t_m]$ is joined to all vertices obtained by changing one
of the coordinates, and to all vertices $[xt_1,\ldots,xt_m]$ for $x\in T$,
$x\ne1$. We say that the adjacency of two vertices differing in the $i$th
coordinate is of \emph{type $i$}, and that of two vertices differing by a
constant left factor is of \emph{type $0$}.

The semilattice clearly determines the graph. So, in particular, the group
$D(T,m)$ acts as a group of graph automorphisms.

If we discard one of the partitions $Q_i$, the remaining partitions form the
minimal partitions in a Cartesian lattice; so the corresponding edges
(those of all types other than~$i$) form a
Hamming graph (Section~\ref{sec:HGCD}). So the diagonal graph is the
edge-union of $m+1$ Hamming graphs $\Ham(T,m)$ on the same set of vertices.
Moreover, two vertices lying in a part of $Q_i$ lie at
maximal distance~$m$ in the Hamming graph obtained by removing $Q_i$.

\begin{theorem}
If $(T,m)$ is not $(C_2,2)$, $(C_3,2)$, $(C_4,2)$, $(C_2\times C_2,2)$, or
$(C_2,3)$, then the diagonal graph determines uniquely the diagonal semilattice.
\label{t:autdiaggraph}
\end{theorem}

\begin{proof}
We handled the case $m=2$ in Proposition~\ref{p:lsgraphaut} and the following
comments, so we can assume that $m\geqslant3$. 

The assumption that $m\geqslant3$ has as a consequence that the parts of the
partitions $Q_i$ are the maximal cliques of the graph. For clearly they are
cliques. Since any clique  of size $2$ or $3$ is contained in a Hamming graph,
we see that any clique of size greater than~$1$ is contained in a 
maximal clique, which has this form; and it is the unique maximal clique
containing the given clique. (See the discussion of cliques in Hamming
graphs in the proof of Theorem~\ref{th:cdham}.)

So all the parts of the partitions $Q_i$ are determined by the graph; we
need to show how to decide when two cliques are parts of the same partition.
We call each  maximal clique a \emph{line}; we say it is an \emph{$i$-line},
or has \emph{type~$i$}, if it is a part of $Q_i$. (So an $i$-line is a maximal
set any two of whose vertices are type-$i$ adjacent.) We have to show that the
partition of lines into types is determined by the graph structure. This
involves a closer study of the graph.

Since the graph admits $D(T,m)$, which induces the symmetric group $S_{m+1}$
on the set of types of line, we can assume (for example) that if we have
three types involved in an argument, they are types $1$, $2$ and $3$.

Call lines $L$ and $M$ \emph{adjacent} if they are disjoint but there are
vertices $x\in L$ and $y\in M$ which are adjacent. Now the following holds:

\begin{quote}
Let $L$ and $M$ be two lines.
\begin{itemize}\itemsep0pt
\item If $L$ and $M$ are adjacent $i$-lines, then every vertex in $L$ is
adjacent to a vertex in $M$.
\item If $L$ is an $i$-line and $M$ a $j$-line adjacent to $L$, with $i\ne j$,
then there are at most two vertices in $L$ adjacent to a vertex in $M$, and
exactly one such vertex if $m>3$.
\end{itemize}
\end{quote}

For suppose that two lines $L$ and $M$ are adjacent, and suppose first that
they have the same type, say type $1$, and that $x\in L$ and $y\in M$ are
on a line of type~$2$. Then $L=\{[*,a_2,a_3,\ldots,a_m]\}$ and
$M=\{[*,b_2,b_3,\ldots,b_m]\}$, where $*$ denotes an arbitrary element of $T$.
We have $a_2\ne b_2$ but $a_i=b_i$ for
$i=3,\ldots,m$. The common neighbours on the two lines
are obtained by taking the entries $*$ to be equal in the two lines.
(The conditions show that there cannot be an adjacency of type $i\ne 2$ between
them.)

Now suppose that $L$ has type~$1$ and $M$ has type~$2$, with a line of
type~$3$ joining vertices on these lines. Then we have $L=\{[*,a_2,a_3,\ldots,a_m]\}$ and
$M=\{[b_1,*,b_3,\ldots,b_m]\}$, where $a_3\ne b_3$ but $a_i=b_i$ for $i>3$;
the adjacent vertices are obtained
by putting ${*}=b_1$ in $L$ and ${*}=a_2$ in $M$.
If $m>3$, there is no adjacency of any other type between the lines.

If $m=3$, things are a little different. There is one type~$3$ adjacency between
the lines $L=\{[*,a_2,a_3]\}$ and $M=\{[b_1,*,b_3]\}$ with $a_3\ne b_3$, namely
$[b_1,a_2,a_3]$ is adjacent to $[b_1,a_2,b_3]$. There is also one type-$0$
adjacency, corresponding to multiplying $L$ on the left by $b_3a_3^{-1}$:
this makes $[x,a_2,a_3]$ adjacent to $[b_1,y,b_3]$ if and only if
$b_3a_3^{-1}x=b_1$ and $b_3a_3^{-1}a_2=y$, determining $x$ and $y$ uniquely.

So we can split adjacency of lines into two kinds: the first kind when the
edges between the two lines form a perfect matching
(so there are $|T|$ such edges); the second kind where
there are at most two such edges (and, if $m>3$, exactly one). Now two
adjacent lines have the same type if and only if the adjacency is of the first
kind. So, if either $m>3$ or $|T|>2$, the two kinds of adjacency are
determined by the graph.

Make a new graph whose vertices are the lines, two lines adjacent if their
adjacency in the preceding sense is of the first kind. Then lines in the
same connected component of this graph have the same type. The converse is
also true, as can be seen within a Hamming subgraph of the diagonal graph.

Thus the partition of lines into types is indeed determined by the graph
structure, and is preserved by automorphisms of the graph.

Finally we have to consider the case where $m=3$ and $T=C_2$. In general,
for $T=C_2$, the Hamming graph is the $m$-dimensional cube, and has a unique
vertex at distance $m$ from any given vertex; in the diagonal graph, these
pairs of antipodal vertices are joined. This is the graph known as the
\emph{folded cube} (see \cite[p.~264]{bcn}). The arguments given earlier apply
if $m\geqslant4$; but, if $m=3$, the graph is the complete bipartite graph $K_{4,4}$,
and any two disjoint edges are contained in a $4$-cycle. 
\end{proof}

\begin{cor}\label{c:sameag}
Except for the cases $(T,m)=(C_2,2)$, $(C_3,2)$, $(C_2\times C_2,2)$, and
$(C_2,3)$, the diagonal semilattice $\mathfrak{D}(T,m)$ and the
diagonal graph $\Gamma_D(T,m)$ have the same automorphism group, namely
the diagonal group $D(T,m)$.
\end{cor}

\begin{proof}
This follows from Theorem~\ref{t:autdiaggraph} and the fact that
$\Gamma_D(C_4,2)$ is the Shrikhande graph, whose automorphism group is
$D(C_4,2)$: see Section~\ref{sect:lsautgp}. 
\end{proof}

\subsection{Properties of finite diagonal graphs}

We have seen some graph-theoretic properties of $\Gamma_D(T,m)$ above.
In this subsection we assume that $T$ is finite and $m\geqslant2$, though we often have to exclude
the case $m=|T|=2$ (where, as we have seen, the diagonal graph is the complete
graph $K_4$).

The \emph{clique number} $\omega(\Gamma)$ of a graph~$\Gamma$
is the number of vertices in its largest clique; the
\emph{clique cover number} $\theta(\Gamma)$ is the smallest number of cliques
whose union contains every vertex; and the \emph{chromatic number}
$\chi(\Gamma)$ is the smallest number of colours required to colour the
vertices so that adjacent vertices receive different colours.

The following properties are consequences of Section~\ref{sec:dgds}, 
especially the proof of Theorem~\ref{t:autdiaggraph}. We give brief
explanations or pointers to each claim.
\begin{itemize}
\item There are $|T|^m$ vertices, and the valency is $(m+1)(|T|-1)$. (The
number of vertices is clear; each point $v$ lies in a unique part of size 
$|T|$ in each of the $m+1$ minimal partitions of the diagonal semlattice. 
Each of these parts is a maximal clique, the parts pairwise intersect 
only in $v$, and the union of the parts contains all the neighbours of $v$.)
\item Except for the case $m=|T|=2$, the clique number is $|T|$, and the
clique cover number is $|T|^{m-1}$. (The parts of each minimal partition
carry maximal cliques, and thus each minimal partition realises a minimal-size
partition of the vertex set into cliques.)
\item $\Gamma_D(T,m)$ is isomorphic to $\Gamma_D(T',m')$ if and only if
$m=m'$ and $T\cong T'$. (The graph is constructed from the semilattice; and
if $m>2$, or $m=2$ and $|T|>4$, the semilattice is recovered from the graph as
in Theorem~\ref{t:autdiaggraph}; for the remaining cases, see the discussion
after Proposition~\ref{p:lsgraphaut}.)

\end{itemize}

Distances and diameter can be calculated as follows. We define two sorts of
adjacency: (A1) is $i$-adjacency for $i\ne0$, while (A2) is $0$-adjacency.

\subsubsection*{Distances in $\Gamma_D(T,m)$} We observe first that, in any
shortest path, adjacencies of fixed type occur
at most once. This is because different factors of $T^{m+1}$ commute, so
we can group those in each factor together.

We also note that distances cannot exceed $m$, since any two vertices are
joined by a path of length at most $m$ using only edges of sort (A1) (which
form a Hamming graph). So a path of smallest length is contained within a
Hamming graph.

Hence, for any two vertices $t=[t_1,\ldots,t_m]$ and $u=[u_1,\ldots,u_m]$, we
compute the distance in the graph by the following procedure:
\begin{itemize}
\item[(D1)] Let $d_1=d_1(t,u)$ be the Hamming distance between the vertices 
$[t_1,\ldots,t_m]$
and $[u_1,\ldots,u_m]$. (This is the length of the shortest path not using a
$0$-adjacency.)
\item[(D2)] Calculate the quotients $u_it_i^{-1}$ for $i=1,\ldots,m$. Let
$\ell$ be the maximum number of times that a non-identity element of $T$ occurs
as one of these quotients, and set $d_2=m-\ell+1$. (We can apply left
multiplication by this common quotient to find a vertex at distance one from
$t$; then use right multiplication by $m-\ell$ appropriate elements to make the
remaining elements agree. This is the length of the shortest path using a
$0$-adjacency.)
\item[(D3)] Now the graph distance $d(u,v)=\min\{d_1,d_2\}$.
\end{itemize}

\subsubsection*{Diameter of $\Gamma_D(T,m)$} An easy argument shows that the diameter of the graph is 
$m+1-\lceil (m+1)/|T|\rceil$ which is at most
$m$, with equality if and only if $|T|\geqslant m+1$. The bound $m$ also follows
directly from the fact that, in the previous procedure, both $d_1$
and $d_2$ are at most $m$.

If $|T|\geqslant m+1$, let $1,t_1,t_2,\ldots,t_m$ be pairwise distinct elements
of~$T$. It is easily
checked that $d([1,\ldots,1],[t_1,\ldots,t_m])=m$. For clearly $d_1=m$;
and for $d_2$ we note that all the ratios are distinct so $l=1$.

\subsubsection*{Chromatic number} 
This has been investigated in two special cases: the case $m=2$ (Latin-square
graphs) in \cite{ghm}, and the case where $T$ is a non-abelian finite simple
group in \cite{bccsz} in connection with synchronization. 
We have not been able to compute the chromatic number in all cases;
this section describes
what we have been able to prove.

The argument in~\cite{bccsz} uses the truth of the
\emph{Hall--Paige conjecture} by
Wilcox~\cite{wilcox}, Evans~\cite{evans}
and Bray et al.~\cite{bccsz}, 
which we briefly discuss.
(See \cite{bccsz} for the history of the proof of this conjecture.)
  
\begin{defn}
A \emph{complete mapping} on a group $G$ is a bijection $\phi:G\to G$ for
which the map $\psi:G\to G$ given by $\psi(x)=x\phi(x)$ is also a bijection.
The map $\psi$ is the \emph{orthomorphism} associated with $\phi$.
\end{defn}
    
In a Latin square, a \emph{transversal} is a set of cells, one in each row,
one in each column, and one containing each letter; an \emph{orthogonal mate}
is a partition of the cells into transversals.
It is well known
(see also~\cite[Theorems~1.4.1 and 1.4.2]{DK:book})
that the following three conditions on a finite group $G$ are
equivalent. (The original proof is in \cite[Theorem~7]{paige}.)
\begin{itemize}\itemsep0pt
    \item $G$ has a complete mapping;
    \item the Cayley table of $G$ has a transversal;
    \item the Cayley table of $G$ has an orthogonal mate.
\end{itemize}
    
The \emph{Hall--Paige conjecture} \cite{hp} (now, as noted, a theorem), 
asserts the following:
    
\begin{theorem}\label{th:hp}
The finite group $G$ has a complete mapping if and only if either $G$ has odd
order or the Sylow $2$-subgroups of $G$ are non-cyclic. 
\end{theorem}

Now let $T$~be a finite group and let $m$~be an integer greater
than~$1$, and consider the diagonal graph $\Gamma_D(T,m)$. The chromatic
number of a graph cannot be smaller than its clique number. We saw at the 
start of this section that the clique number is $|T|$ unless $m=2$ and $|T|=2$.
\begin{itemize}
\item Suppose first that $m$ is odd. We give the vertex $[t_1,\ldots,t_m]$
  the colour
  $u_1u_2 \cdots u_m$ in $T$, where $u_i=t_i$ if $i$~is odd
  and $u_i=t_i^{-1}$ if $i$~is even.
  If two vertices lie in a part
of $Q_i$ with $i>0$, they differ only in the $i$th coordinate, and clearly
their colours differ. Suppose that $[t_1,\ldots,t_m]$ and $[s_1,\ldots,s_m]$
lie in the same part of $Q_0$, so that $s_i=xt_i$ for $i=1,\ldots,m$,
where $x\ne1$. Put $v_i=s_i$ if $i$ is odd and $v_i=s_i^{-1}$ if $i$~is even.
Then $v_iv_{i+1} = u_iu_{i+1}$ whenever $i$ is even, so
the colour of the second vertex is
\[v_1v_2 \cdots v_m = v_1 u_2 \cdots u_m =xu_1 u_2 \cdots u_m,\] 
which is different from that of the first vertex since $x\ne1$.
\item Now suppose that $m$ is even and assume in this case that the Sylow
$2$-subgroups of $T$ are are trivial or non-cyclic. Then, by
Theorem~\ref{th:hp}, $T$~has a complete mapping~$\phi$. Let $\psi$ be
the corresponding orthomorphism. We define the colour of the vertex
$[t_1,\ldots,t_m]$ to be 
\[t_1^{-1}t_2t_3^{-1}t_4\cdots t_{m-3}^{-1}t_{m-2}t_{m-1}^{-1}\psi(t_m).\]
An argument similar to but a little more elaborate than in the other case
shows that this is a proper colouring. We refer to \cite{bccsz} for details.
\end{itemize}

With a little more work we get the following theorem, a
contribution to the general question concerning the chromatic number of 
the diagonal graphs. Let $\chi(T,m)$ denote the chromatic
number of $\Gamma_D(T,m)$. 

\begin{theorem}\label{thm:chrom}
\begin{enumerate}
\item If $m$ is odd, or if $|T|$ is odd, or if the Sylow $2$-subgroups of
$T$ are non-cyclic, then $\chi(T,m)=|T|$.
\item If $m$ is even, then $\chi(T,m)\leqslant\chi(T,2)$.
\end{enumerate}
\end{theorem}

All cases in (a) were settled above; we turn to~(b).

A \emph{graph homomorphism} from $\Gamma$ to $\Delta$ is a map from the
vertex set of $\Gamma$ to that of $\Delta$ which maps edges to edges.
A proper $r$-colouring of a graph $\Gamma$ is a homomorphism from $\Gamma$
to the complete graph $K_r$. Since the composition of homomorphisms is a
homomorphism, we see that if there is a homomorphism from $\Gamma$ to
$\Delta$ then there is a colouring of $\Gamma$ with $\chi(\Delta)$ colours,
so $\chi(\Gamma)\leqslant\chi(\Delta)$. 

\begin{theorem}\label{thm:diagepi}
For any $m\geqslant 3$ and non-trivial finite group $T$, there is a homomorphism from $\Gamma_D(T,m)$
to $\Gamma_D(T,m-2)$.
\end{theorem}

\begin{proof} We define a map by mapping a vertex $[t_1,t_2,\ldots,t_m]$ of 
$\Gamma_D(T,m)$ to the vertex $[t_1t_2^{-1}t_3,t_4,\ldots,t_m]$ of
$\Gamma_D(T,m-2)$, and show that this map is a homomorphism. If
two vertices of $\Gamma_D(T,m)$ agree in all but position~$j$, then their
images agree in all but position $1$ (if $j\le 3$) or $j-2$ (if $j>3$).
Suppose that $t_i=xs_i$ for $i=1,\ldots,m$. Then
$t_1t_2^{-1}t_3=xs_1s_2^{-1}s_3$, so the images of $[t_1,\ldots,t_m]$ and
$[s_1,\ldots,s_m]$ are joined. This completes the proof. 
\end{proof}

This also completes the proof of Theorem~\ref{thm:chrom}. 

\medskip

The paper \cite{ghm} reports new results on the chromatic number of a
Latin-square graph, in particular, if $|T|\geqslant 3$ then
$\chi(T,2)\leqslant 3|T|/2$. They also report a conjecture of Cavenagh,
which claims that $\chi(T,2)\leqslant |T|+2$,
and prove this conjecture in the case where $T$ is abelian.

Payan~\cite{Payan} showed that graphs in a class he called ``cube-like''
cannot have chromatic number~$3$. Now $\Gamma_D(C_2,2)$,
which is the complete graph~$K_4$, has chromatic number~$4$; and the 
folded cubes $\Gamma_D(C_2,m)$ are ``cube-like'' in Payan's sense.
It follows from Theorems~\ref{thm:chrom} and~\ref{thm:diagepi} that the
chromatic number of the folded cube $\Gamma_D(C_2,m)$ is $2$ if $m$~is odd and
$4$ if $m$~is even. So the bound in Theorem~\ref{thm:chrom}(b) is attained if
$T\cong C_2$.

\subsection{Synchronization}\label{sec:Synch}

A permutation group $G$ on a finite set $\Omega$ is said to be
\emph{synchronizing} if, for any map $f:\Omega\to\Omega$ which is not a
permutation, the transformation monoid $\langle G,f\rangle$ on $\Omega$
generated by $G$ and $f$ contains a map of rank~$1$ (that is, one which maps
$\Omega$ to a single point). For the background of this notion in automata
theory, we refer to \cite{acs:synch}.

The most important tool in the study of synchronizing groups is the following
theorem \cite[Corollary 4.5 ]{acs:synch}.  A graph is \emph{trivial} if it
is complete or null.
  
\begin{theorem}\label{th:nonsynch}
A permutation group $G$ is synchronizing if and only if no non-trivial
$G$-invariant graph has clique number equal to chromatic number. 
\end{theorem}
  
From this it immediately follows that a synchronizing group is transitive
(if $G$ is intransitive, take a complete graph on one orbit of~$G$), and
primitive (take the disjoint union of complete graphs on the blocks in a
system of imprimitivity for~$G$). Now, by the O'Nan--Scott theorem 
(Theorem~\ref{thm:ons}), a
primitive permutation group preserves a Cartesian or diagonal semilattice or
an affine space, or else is almost simple.

\begin{theorem}
If a group $G$ preserves a Cartesian decomposition, then it is non-synchro\-nizing.
\end{theorem}
  
This holds because the Hamming graph has clique number equal to chromatic
number. (We saw in the proof of Theorem~\ref{th:cdham} that the clique number of
the Hamming graph is equal to the
cardinality of the alphabet. Take the alphabet $A$ to be an abelian group;
also use $A$ for the set of colours, and give the $n$-tuple
$(a_1,\ldots,a_n)$ the colour $a_1+\cdots+a_n$. If two $n$-tuples are 
adjacent in the Hamming graph, they differ in just one coordinate, and so
get different colours.)

In \cite{bccsz}, it is shown that a primitive diagonal group whose socle
contains $m+1$ simple factors with $m>1$ is non-synchronizing.
In fact, considering Theorem~\ref{th:primaut}, the following more general result 
is valid.
  
\begin{theorem}
If $G$ preserves a diagonal semilattice $\mathfrak{D}(T,m)$ with $m>1$ and $T$
a finite group of order greater than~$2$, then $G$ is non-synchronizing.
\end{theorem}

\begin{proof}
If $T$ is not characteristically simple then Theorem~\ref{th:primaut} implies
that $G$~is imprimitive and so it is non-synchronizing. Suppose that $T$ is
characteristically simple and  let $\Gamma$ be the diagonal graph 
$\Gamma_D(T,m)$. Since we have excluded the case $|T|=2$, the clique number of
$\Gamma$ is $|T|$, as we showed in the preceding subsection. Also, either $T$
is an elementary abelian group of odd order or the Sylow 2-subgroups of $T$ 
are non-cyclic. (This is clear unless $T$ is simple, in which case it follows
from Burnside's Transfer Theorem, see \cite[(39.2)]{asch}.) So, by
Theorem~\ref{thm:chrom}, $\chi(\Gamma)=|T|$. Now Theorem~\ref{th:nonsynch}
implies that  $D(T,m)$ is non-synchronizing;
since $G\leqslant D(T,m)$, also $G$~is non-synchronizing. 
\end{proof}

\begin{remark}
It follows from the above that a synchronizing permutation group must be of one
of the following types: affine (with the point stabiliser a primitive linear
group); simple diagonal with socle the product of two copies of a non-abelian
simple group; or almost simple. In the first and third cases, some but not all
such groups are synchronizing; in the second case, no synchronizing example
is known.
\end{remark}

\section{Open problems}\label{s:problems}

Here are a few problems that might warrant further investigation.

For $m\geqslant 3$, Theorem~\ref{th:main} characterised $m$-dimensional special sets of 
partitions as minimal partitions in join-semilattices $\mathfrak D(T,m)$ for a 
group $T$. However, for $m=2$, such special sets arise from an arbitrary quasigroup $T$.
The automorphism group of the join-semilattice generated by a 2-dimensional special 
set is the autoparatopism group of the quasigroup $T$ and, for $|T|>4$,  
it also coincides with 
the automorphism group of the corresponding Latin-square graph 
(Proposition~\ref{p:autlsg}).

Since we wrote the first draft of the paper, Michael Kinyon has pointed out to
us that the Paige loops~\cite{paige:loops} (which were shown by
Liebeck~\cite{liebeck} to be the only finite simple Moufang loops which are not
groups) have vertex-primitive autoparatopism groups.

\begin{problem}
Determine whether there exists a quasigroup $T$, not isotopic to a group or
a Paige loop, whose
autoparatopism group is primitive.
This is equivalent to requiring that the automorphism group of the corresponding
Latin-square graph is vertex-primitive; see Proposition~\ref{p:autlsg}.
\end{problem}

    If $T$ is a non-abelian finite simple group and $m\geqslant 3$, 
    then the diagonal group $D(T,m)$ is a  maximal subgroup of the 
    symmetric or alternating group~\cite{LPS}.  What happens in the infinite
    case?  
    
    \begin{problem}
    Find a maximal subgroup of $\Sym(\Omega)$ that contains 
    the diagonal group $D(T,m)$ if $T$ is an infinite simple group. If $\Omega$
is countably infinite, then by~\cite[Theorem~1.1]{macpr},
     such a maximal subgroup exists. 
    (For a countable set, \cite{covmpmek} describes maximal subgroups 
    that stabilise a Cartesian lattice.)
    \end{problem}

\begin{problem}
Investigate the chromatic number $\chi(T,m)$ of the
diagonal graph $\Gamma_D(T,m)$ if $m$ is even and $T$ has no complete mapping.
In particular, either show that the bound in Theorem~\ref{thm:chrom}(b)
is always attained (as we noted, this is true for $T=C_2$) or improve this bound. 
\end{problem}

For the next case where the Hall--Paige conditions fail, namely $T=C_4$,
the graph $\Gamma_D(T,2)$ is the complement of the Shrikhande graph, and has
chromatic number $6$; so, for any even $m$, the chromatic number of
$\Gamma_D(T,m)$ is $4$, $5$ or $6$, and the sequence of chromatic numbers is
non-increasing.

If $T$ is a direct product of $m$ pairwise isomorphic non-abelian simple groups, 
with $m$ an integer and $m>1$, then $D(T,m)$ preserves a Cartesian lattice 
by \cite[Lemma~7.10(ii)]{ps:cartesian}. Here $T$ is  not necessarily finite, 
and groups with this property are called FCR (finitely completely reducible) groups.  
However there are other infinite characteristically simple groups, 
for example the McLain group~\cite{mclain}.

\begin{problem}
Determine whether there exist characteristically simple (but not simple) groups $T$ 
which are not FCR-groups,  and integers $m>1$, such that $D(T,m)$ preserves a 
Cartesian lattice.  
It is perhaps the case that $D(T,m)$ does not preserve a Cartesian lattice 
for these groups $T$; and we ask further whether $D(T,m)$ might still preserve some 
kind of structure that has more automorphisms than the diagonal semilattice.
\end{problem}

\begin{problem}\label{p2}
Describe sets of more than $m+1$ partitions of
$\Omega$, any $m$ of which are the minimal elements in a Cartesian lattice.
\end{problem}

For $m=2$, these are equivalent to sets of mutually orthogonal Latin squares.

For $m>2$, any $m+1$ of the partitions are the minimal elements in a
diagonal semilattice $\mathcal{D}(T,m)$. Examples are known when $T$ is
abelian. One such family is given as follows. Let $T$ be the additive group
of a field $F$ of order $q$, where $q>m+1$; let $F=\{a_1,a_2,\ldots,a_q\}$.
Then let $W=F^m$. For $i=1,\ldots,q$, let $W_i$ be the subspace
spanned by $(1,a_i,a_i^2,\ldots,a_i^{m-1})$, and let $W_0$ be the subspace
spanned by $(0,0,\ldots,0,1)$. The coset partitions of $W$ given by these
$q+1$ subspaces have the property that any $m$ of them are the minimal elements
in a Cartesian lattice of dimension $m$ (since any $m$ of the given vectors
form a basis of $W$.) Note the connection with MDS codes and geometry: the
$1$-dimensional subspaces are the points of a normal rational curve in
$\mathrm{PG}(m-1,F)$. See~\cite{btb}.

For which non-abelian groups $T$ do examples with $m>2$ exist?

\begin{problem}
With the hypotheses of Problem~\ref{p2}, find a good upper bound
for the number of partitions, in terms of $m$ and $T$.
\end{problem}

We note one trivial bound: the number of such partitions cannot exceed
$m+|T|-1$. This is well-known when $m=2$ (there cannot be more than $|T|-1$
mutually orthogonal Latin squares of order $|T|)$. Now arguing inductively
as in the proof of Proposition~\ref{p:quots}, we see that increasing $m$ by
one can increase the number of partitions by at most one.

\medskip

Since the first draft of this paper was written, three of the authors and
Michael Kinyon have written a paper \cite{bckp} addressing (but by no means
solving) the last two problems above.

\section*{Acknowledgements} Part of the work was done while the authors were visiting
the South China University of Science and Technology (SUSTech), Shenzhen, in 2018, and 
we are grateful (in particular to Professor Cai Heng Li) 
for the hospitality that we received. 
The authors would like to thank the 
Isaac Newton Institute for Mathematical Sciences, Cambridge, 
for support and hospitality during the programme
\textit{Groups, representations and applications: new perspectives}
(supported by \mbox{EPSRC} grant no.\ EP/R014604/1), 
where further work on this paper was undertaken.
In particular we acknowledge a Simons Fellowship (Cameron) and a Kirk Distinguished 
Visiting Fellowship (Praeger) during this programme.  Schneider thanks the Centre for the Mathematics of Symmetry
and Computation 
of The University of Western Australia and 
Australian Research Council Discovery Grant DP160102323
for  hosting his visit 
in 2017 and 
acknowledges the 
support of the CNPq projects \textit{Produtividade em Pesquisa}
(project no.: 308212/2019-3)  
and \textit{Universal} (project no.: 421624/2018-3).

We are grateful to Michael Kinyon for comments on an earlier version of the
paper and to the anonymous referee for his or her careful reading of the manuscript.


\end{document}